\documentclass[reqno,11pt]{amsart}

\usepackage{amscd}
\usepackage{amsmath}
\usepackage[new]{old-arrows}
\usepackage{amssymb}
\usepackage[american]{babel}
\usepackage{xcolor}
\usepackage{bbm}
\usepackage{bookmark}
\usepackage{cmap} 
\usepackage{dsfont}
\usepackage{enumerate}
\usepackage{epigraph}
\usepackage[mathscr]{euscript}
\usepackage[myheadings]{fullpage}
\usepackage{graphicx}
\usepackage[geometry]{ifsym}
\usepackage{mathabx}
\usepackage{accents}
\usepackage{mathrsfs}
\usepackage{mathtools}
\usepackage{enumitem}
\usepackage{refcount}
\usepackage{setspace}
\usepackage{stmaryrd}
\usepackage{thinsp}
\usepackage{verbatim}
\usepackage[all,2cell]{xy} \UseTwocells
\usepackage{xr}
\usepackage[multiple]{footmisc}
\usepackage{hyperref}

\usepackage{eufrak}
\usepackage[T1]{fontenc}

\usepackage{tikz}
\usepackage{tikz-cd}
\usetikzlibrary{decorations.pathmorphing}
\usetikzlibrary{cd}
\usetikzlibrary{calc}

\tikzset{curve/.style={settings={#1},to path={(\tikztostart)
    .. controls ($(\tikztostart)!\pv{pos}!(\tikztotarget)!\pv{height}!270:(\tikztotarget)$)
    and ($(\tikztostart)!1-\pv{pos}!(\tikztotarget)!\pv{height}!270:(\tikztotarget)$)
    .. (\tikztotarget)\tikztonodes}},
    settings/.code={\tikzset{quiver/.cd,#1}
        \def\pv##1{\pgfkeysvalueof{/tikz/quiver/##1}}},
    quiver/.cd,pos/.initial=0.35,height/.initial=0}

\tikzset{tail reversed/.code={\pgfsetarrowsstart{tikzcd to}}}
\tikzset{2tail/.code={\pgfsetarrowsstart{Implies[reversed]}}}
\tikzset{2tail reversed/.code={\pgfsetarrowsstart{Implies}}}
\tikzset{no body/.style={/tikz/dash pattern=on 0 off 1mm}}

\numberwithin{equation}{subsection}

\newtheorem{thm}{Theorem}[subsubsection]
\newtheorem*{thm*}{Theorem}

\newtheorem{cor}[thm]{Corollary}
\newtheorem*{cor*}{Corollary}
\newtheorem{lem}[thm]{Lemma}

\newtheorem{prop}[thm]{Proposition}
\newtheorem{prop-const}[thm]{Proposition-Construction}

\newtheorem*{conjecture*}{Conjecture}

\newtheorem*{princ*}{Principle}

\theoremstyle{remark}
\newtheorem{rem}[thm]{Remark}

\newcounter{steps}[thm]
\newtheorem{step}[steps]{Step}

\newcommand{\into}{\hookrightarrow}
\newcommand{\onto}{\twoheadrightarrow}

\newcommand{\bB}{{\mathbb B}}

\newcommand{\bD}{{\mathbb D}}

\newcommand{\bL}{{\mathbb L}}

\newcommand{\bZ}{{\mathbb Z}}

\newcommand{\cF}{{\mathcal F}}
\newcommand{\cG}{{\mathcal G}}

\newcommand{\cN}{{\mathcal N}}
\newcommand{\cO}{{\mathcal O}}

\newcommand{\cX}{{\mathcal X}}
\newcommand{\cY}{{\mathcal Y}}

\newcommand{\sA}{{\EuScript A}}
\newcommand{\sB}{{\EuScript B}}
\newcommand{\sC}{{\EuScript C}}
\newcommand{\sD}{{\EuScript D}}

\newcommand{\sH}{{\EuScript H}}

\newcommand{\sJ}{{\EuScript J}}

\newcommand{\sL}{{\EuScript L}}

\newcommand{\sP}{{\EuScript P}}

\newcommand{\sS}{{\EuScript S}}

\newcommand{\fL}{{\mathfrak L}}

\newcommand{\fU}{{\mathfrak U}}

\newcommand{\fb}{{\mathfrak b}}

\newcommand{\fg}{{\mathfrak g}}

\newcommand{\fm}{{\mathfrak m}}
\newcommand{\fn}{{\mathfrak n}}

\newcommand{\fp}{{\mathfrak p}}

\newcommand{\ft}{{\mathfrak t}}

\newcommand{\on}{\operatorname}
\newcommand{\ol}[1]{\overline{#1}{}}
\newcommand{\ot}[1]{\widetilde{#1}{}}

\newcommand{\Hom}{\on{Hom}}

\newcommand{\Vect}{\mathsf{Vect}}

\newcommand{\Res}{\on{Res}}
\newcommand{\Gr}{\on{Gr}}

\newcommand{\Bun}{\on{Bun}}

\newcommand{\IC}{\on{IC}}

\renewcommand{\lim}{\on{lim}}

\newcommand{\Oblv}{\on{Oblv}}
\newcommand{\Av}{\on{Av}}

\newcommand{\Hecke}{\on{Hecke}}

\newcommand{\sIC}{\on{IC}_{P,\on{Ran}}^{\frac{\infty}{2}}}

\newcommand{\Ln}{\on{\Lambda}_{G,P}^{\on{neg}}}
\newcommand{\Lp}{\on{\Lambda}_{G,P}^{\on{pos}}}
\newcommand{\tildeP}{\ot{\on{Bun}}_P}
\newcommand{\Ploc}{S_{P,\on{Ran}}^0}
\newcommand{\tildePloc}{\ot{S}_{P,\on{Ran}}^0}

\renewcommand{\subset}{\subseteq}

\newcommand{\Ran}{\on{Ran}}

\newcommand{\cnp}{\check{\fn}_P}
\newcommand{\chM}{\check{M}}
\newcommand{\chP}{\check{P}}
\newcommand{\Mran}{\on{Rep}(\check{M})_{\mathrm{Ran}}}

\makeatletter
\newcommand{\biggg}{\bBigg@{4}}
\newcommand{\Biggg}{\bBigg@{5}}
\makeatother

\begin{document}

\frenchspacing

\setlength{\epigraphwidth}{0.4\textwidth}
\renewcommand{\epigraphsize}{\footnotesize}

\title{Parabolic geometric Eisenstein series and constant term functors}

\author{Joakim F\ae rgeman and Andreas Hayash}

\begin{abstract}

We prove a compatibility between parabolic restriction of Whittaker sheaves and restriction of representations under the geometric Casselman-Shalika equivalence. To do this, we establish various Hecke structures on geometric Eisenstein series functors, generalizing results of Braverman-Gaitsgory in the case of a principal parabolic. Moreover, we relate compactified and non-compactified geometric Eisenstein series functors via Koszul duality.

We sketch a proof that the spectral-to-automorphic geometric Langlands functor commutes with constant term functors.
    
\end{abstract}

\maketitle

\tableofcontents

\section{Introduction}

\subsection{Framework} Let $G$ be a reductive connected algebraic group over a field $k$ of characteristic zero. Fix a parabolic subgroup $P \subseteq G$ with Levi quotient $M$ and unipotent radical $N_P$. Let $\check{G}$ be the Langlands dual reductive group of $G$ with corresponding parabolic subgroup $\check{P} = \check{M} \check{N}_P$. We fix a smooth projective curve $X$ over $k$.

\subsubsection{} This paper is concerned with geometric Eisenstein series, which we presently recall. By induction of princial bundles along the correspondence $M \leftarrow P \rightarrow G$, we obtain a commutative diagram 
\begin{equation}\label{eq: fundamental diagram}\begin{tikzcd}
	& {\Bun_P} \\
	& {\ot{\Bun}_P} \\
	{\Bun_M} && {\Bun_G,}
	\arrow["j", from=1-2, to=2-2]
	\arrow["q"', curve={height=12pt}, from=1-2, to=3-1]
	\arrow["p", curve={height=-12pt}, from=1-2, to=3-3]
	\arrow["{\ot{q}}"', from=2-2, to=3-1]
	\arrow["{\ot{p}}", from=2-2, to=3-3]
\end{tikzcd}\end{equation}

\noindent where for an algebraic group $H$, $\Bun_H$ is the moduli stack of principal $H$-bundles over $X$. The algebraic stack $\ot{\Bun}_P$ denotes Drinfeld's relative compactification of $p$. 

\subsubsection{} For or a prestack $\cY$, let $D(\cY)$ be the category of D-modules on $\cY$. Pull-push along the diagram \eqref{eq: fundamental diagram} provides us with Eisenstein series functors, the protagonists of this paper. More precisely, we have the functors:
\[
\on{Eis}_!: D(\on{Bun}_M)\to D(\on{Bun}_G),\;\; \cF\mapsto p_*\circ q^!(\cF);
\]
\[
\on{Eis}_{!*}: D(\on{Bun}_M)\to D(\on{Bun}_G),\;\; \cF\mapsto \ot{p}_*(\on{IC}_{\ot{\on{Bun}}_P}\overset{!}{\otimes}\ot{q}^!(\cF)),
\]

\noindent which we call (geometric) \emph{Eisenstein series} and \emph{compactified Eisenstein series}, respectively.

\subsubsection{}\label{s:theorems} We will also consider several local variations of the functors $\on{Eis}_!$ and $\on{Eis}_{!*}$ as well as their dual counterparts, the \emph{constant term functors}.\footnote{The name \emph{constant term functor} is sometimes used in the literature to denote an adjoint functor to Eisenstein series and sometimes to denote its dual functor. We allow ourselves to be vague for now on which of the two meanings we are using.}

The goal of this paper is to establish some fundamental properties of geometric Eisenstein series and constant term functors. More precisely:
\begin{enumerate}\label{list: things we did}

    \item We establish a certain compatibility between geometric Eisenstein series and Hecke functors. This generalizes a result of Braverman-Gaitsgory \cite{braverman1999geometric}, \cite{braverman2006deformations} in the case of a principal parabolic.

    \item We prove that restriction of representations along $\check{M}\to \check{G}$ goes to a suitable Jacquet functor on the spherical Whittaker category under the geometric Casselman-Shalika equivalence.

    \item As a corollary to the second point above, we characterize the Koszul dual Jacquet functor as invariants with respect to the Lie algebra $\check{\fn}_P$ of the unipotent radical of $\check{P}$. This generalizes a result of Raskin \cite{raskin2021chiral} in the principal case. 
\end{enumerate}

\subsubsection{Relation to the proof of the geometric Langlands conjecture}\label{s:proofofGLC} A fundamental step in the proof of the geometric Langlands conjecture is to establish that the Langlands functor
\[
\bL_G: D(\on{Bun}_G)\to \on{IndCoh}_{\cN}(\on{LocSys}_{\check{G}})
\]

\noindent interchanges geometric and spectral Eisenstein series, suitably understood. This was proven in \cite{campbell2024proof}. The arguments use in an essential way results of Campbell-Raskin \cite{campbell2023langlands} who establish a semi-infinite geometric Satake theorem. However, we wish to emphasize that the contents of the paper \cite{campbell2023langlands} use our results stated in point $(3)$ above in an essential way, and in particular, neither \cite{campbell2023langlands} nor \cite{campbell2024proof} should be viewed as implying the theorems established in the present text.

\subsection{Some context}\label{s:somecontext}

To set the stage of our main results, let us review the results of Braverman-Gaitsgory \cite{braverman1999geometric}, \cite{braverman2006deformations} on the compatibility between Eisenstein series and Hecke functors mentioned in point $(1)$ of §\ref{s:theorems}. They work under the assumption that the parabolic $P$ is a Borel subgroup $B$. In this case, we denote Drinfeld's compactification by $\ol{\Bun}_B$.

\subsubsection{}\label{ss:basicstructure} The basic results can be summarized as follows, although we refer to \cite{finkelberg1997semiinfinite},\cite{feigin1997semiinfinite}, \cite{braverman2002intersection},\cite{braverman1999geometric},\cite{braverman2006deformations} for more details.

Let $\on{Conf}$ be the space of divisors on $X$ valued in the monoid $\Lambda^{\on{pos}}$ of positive linear combinations of positive roots in $G$. One has two factorization algebras $\cO(\check{N})_{\on{Conf}}, \Omega(\check{\fn})_{\on{Conf}}\in D(\on{Conf})$ obtained from the coalgebra $\cO(\check{N})$ of functions on $\check{N}$ and the algebra $C^{\bullet}(\check{\fn})$ given by the cohomological Chevalley complex of $\check{\fn}_P$, respectively. We have actions:
\[
\cO(\check{N})_{\on{Conf}}\curvearrowright \on{IC}_{\overline{\on{Bun}}_B},\;\;\Omega(\check{\fn})_{\on{Conf}}\curvearrowright j_!(\omega_{\on{Bun}_B}).
\]

\noindent The usual Koszul duality between $\cO(\check{N})$ and $C^{\bullet}(\check{\fn})$ provides a 'Koszul duality' between $\on{IC}_{\overline{\on{Bun}}_B}$ and $j_!(\omega_{\on{Bun}_B})$. That is, taking invariants of $\on{IC}_{\overline{\on{Bun}}_B}$ as a module sheaf for $\cO(\check{N})_{\on{Conf}}$ yields $j_!(\omega_{\on{Bun}_B})$, and vice versa. This similarly shows that the functors given by the kernels $\on{IC}_{\overline{\on{Bun}}_B}, j_!(\omega_{\on{Bun}_B})$, namely $\on{Eis}_{!*}$ and $\on{Eis}_{!}$, respectively, are 'Koszul dual'.

We have an equivalence:
\[
\Omega(\check{\fn})_{\on{Conf}}\on{-mod}(D(\on{Bun}_T))\simeq D(\on{Bun}_T)\underset{\on{QCoh}(\on{LocSys}_{\check{T}})}{\otimes}\on{QCoh}(\on{LocSys}_{\check{B}}).
\]

\noindent This shows that the functor $\on{Eis}_!$ factors through a functor:
\[
\on{Eis}_!^{\on{enh}}: D(\on{Bun}_T)\underset{\on{QCoh}(\on{LocSys}_{\check{T}})}{\otimes}\on{QCoh}(\on{LocSys}_{\check{B}})\to D(\on{Bun}_G).
\]

\noindent In the language introduced in this paper, this reflects the fact that $j_!(\omega_{\on{Bun}_B})$ is equipped with an \emph{enhanced Drinfeld-Plücker structure}, see §\ref{ss:heckestructures} below.

\subsection{Why are we writing this paper?}

\subsubsection{} The proofs of the results stated in §\ref{ss:basicstructure} required a tremendous amount of control of the singularities of $\on{IC}_{\overline{\on{Bun}}_B}$. The above actions ultimately come from describing how $\on{IC}_{\overline{\on{Bun}}_B}$ decomposes when restricted to a natural stratification of $\overline{\on{Bun}}_B$, and verifying that the action maps are compatible with the (co)algebra structures is a very subtle affair.

While one of the main motivations for writing this paper is to generalize the results stated in §\ref{ss:basicstructure} from the principal case to a general parabolic, another major motivation is to provide a short, self-contained proof of these results that completely avoids the just-mentioned subtleties.

\subsubsection{} Our starting point is a certain local analogue of $\IC_{\ot{\Bun}_P}$ known as the \emph{semi-infinite intersection cohomology sheaf} $\IC^{\frac{\infty}{2}}_{P,\Ran}$. Its definition, fundamental properties, and local-to-global compatibilities in the case when $P=B$ is a Borel were first established by Gaitsgory in \cite{gaitsgory2018semi}, \cite{gaitsgory2021semi}. 

\begin{rem} While writing this paper, a definition of the fiber $\IC^{\frac{\infty}{2}}_{P,x}$ of $\IC^{\frac{\infty}{2}}_{P,\Ran}$ at a point $x$ of $X$ appeared in \cite{dhillon2025semiinfiniteic}. It is our understanding that Dhillon-Lysenko plan to release a second paper studying $\IC^{\frac{\infty}{2}}_{P,\Ran}$ factorizably.

Another definition of $\IC^{\frac{\infty}{2}}_{P,\Ran}$ as a factorization algebra also appears in the paper \cite{hayash2025zastava} of the second author, using Zastava spaces. It should be noted that the methods used in the present paper differ from the above references in that we use Hecke structures to define $\IC^{\frac{\infty}{2}}_{P,\Ran}$ rather than employing colimit constructions or t-structures directly. 
\end{rem}

In more concrete terms, $\on{IC}_{P,\on{Ran}}^{\frac{\infty}{2}}$ is a sheaf on a suitable prestack
\begin{equation}\label{eq:ltog}
\widetilde{\on{Gr}}_{P,\on{Bun}_M}\to \ot{\Bun}_P,
\end{equation}

\noindent see Section \ref{S:LOCTOGLOB} for its definition. The semi-infinite IC-sheaf is designed to satisfy many of the properties required of $\on{IC}_{\widetilde{\on{Bun}}_P}$ to establish the results in §\ref{ss:basicstructure}. As such, instead of starting with $\on{IC}_{\ot{\Bun}_P}$ and proving its desired properties, one starts with $\on{IC}_{P,\on{Ran}}^{\frac{\infty}{2}}$, where these properties are much more evident, and then identifies the latter with $\on{IC}_{\ot{\on{Bun}}_P}$ under the map (\ref{eq:ltog}). The upshot is that it is often easy to identify a given sheaf with an IC-sheaf: one simply has to control the perverse degrees of the $!$ and $*$-restrictions to the singular locus as well as its restriction to the smooth locus.

\subsubsection{} In fact, one can replace $\on{IC}_{\widetilde{\on{Bun}}_P}$ with the pushforward of $\on{IC}_{P,\on{Ran}}^{\frac{\infty}{2}}$ along the map (\ref{eq:ltog}) in the definition of $\on{Eis}_{!*}$ and not lose any information.\footnote{For example, in principle one never needs to consider $\on{IC}_{\widetilde{\on{Bun}}_P}$ to prove that $\bL_G$ interchanges the two Eisenstein series functors.} It is our understanding that this approach will be taken by Hamann-Hansen-Scholze in their work\footnote{The work we allude to here is the second of their series of papers on geometric Eisenstein series and is unavailable at the time of writing. See, however, the first paper in the series: \cite{hamann2024geometric}.} on geometric Eisenstein series on stacks of bundles on the Fargues-Fontaine curve; the point being that in this setting, perverse t-structures and operations of intermediate extensions do not obviously have any good properties in general.

\subsubsection{} The other main purpose of this paper is to prove a certain compatibility between geometric and spectral parabolic restriction under the geometric Casselman-Shalika equivalence (see §\ref{S:JACquet} below for a precise statement). This generalizes a result of Raskin in the principal case \cite{raskin2021chiral}.

Given this result, we sketch how to prove that the spectral-to-automorphic Langlands functor
\[
\bL_G^{\on{spec}}: \on{QCoh}(\on{LocSys}_{\check{G}})\to D(\on{Bun}_G)
\]

\noindent commutes with constant term functors. One should consider this equivalent to the assertion that the automorphic-to-spectral Langlands functor $\bL_G$ commutes with (suitable) Eisenstein series functors, see §\ref{s:interchangeEIS} below for more details.

\subsubsection{} Finally, we note that the results in this paper have been cited in the literature already. Besides the relation to the work of Campbell-Raskin mentioned in §\ref{s:proofofGLC}, the Hecke structure on geometric Eisenstein series for an arbitrary parabolic (Theorem \ref{t:glob_Hecke} below) played a key role in many important results in geometric Langlands, such as:
\begin{itemize}
    \item The spectral decomposition of $D(\on{Bun}_G)$ over $\on{QCoh}(\on{LocSys}_{\check{G}})$ \cite{gaitsgory2010generalized}.

    \item The existence of Hecke eigensheaves associated to irreducible local systems (a folklore result written up in \cite[§11]{faergeman2022non}).

    \item (A more recent application:) combining the characterization of nilpotent singular support in terms of Hecke functors given in \cite{arinkin2020stack} with the Hecke structure on geometric Eisenstein series proves that geometric Eisenstein series preserves nilpotent singular support.\footnote{In forthcoming joint work with Marius Kjærsgaard, the first author establishes a purely geometric proof that geometric Eisenstein series and constant term functors preserve nilpotent singular support, independent of the current paper and \cite{arinkin2020stack}.} In the $\ell$-adic setting, this is a crucial input if one wants to study pseudo-Eisenstein series of automorphic forms via geometric Eisenstein series through categorical trace of Frobenius (see \cite{raskin2024arithmetic} for one such example).
\end{itemize}

\subsection{Jacquet Functors}\label{S:JACquet}

In the next two subsections, we describe the main results in this paper in more detail.

\subsubsection{Semi-infinite IC-sheaf}\label{s:psic}

The main player in the local story is the parabolic semi-infinite IC-sheaf:
\[
\on{IC}_{P,\on{Ran}}^{\frac{\infty}{2}}\in D(\on{Gr}_{G,\on{Ran}}).
\]

The first basic properties of $\sIC$ are:
\begin{itemize}

\item $\sIC$ naturally belongs to the \emph{semi-infinite category} $\on{SI}_{P,\on{Ran}}:=D(\fL_{\on{Ran}}N_P\fL^+_{\on{Ran}}M\backslash \on{Gr}_{G,\on{Ran}})$.

\item Denote by $\widetilde{S}^0_{P,\on{Ran}}\subset\on{Gr}_{G,\on{Ran}}$ the closure of the 0'th semi-infinite $\fL_{\on{Ran}}N_P$-orbit (see §\ref{s:tildeploc} for a precise definition). Then $\sIC$ is supported on $\widetilde{S}^0_{P,\on{Ran}}$.

\end{itemize}

\subsubsection{} Define
\[
\ot{\Gr}_{P,\Ran} \coloneqq \Gr_{M,\Ran}\underset{\mathbb{B}\mathfrak{L}^+_{\Ran}M}{\times} \mathfrak{L}^+_{\Ran}M\backslash  \ot{S}^0_{P,\Ran};
\]
\[
\Gr_{P,\Ran} \coloneqq \Gr_{M,\Ran}\underset{\mathbb{B}\mathfrak{L}^+_{\Ran}M}{\times} \mathfrak{L}^+_{\Ran}M\backslash  S^0_{P,\Ran}.
\]

\noindent We have an open embedding:
\[
j_{\on{Ran}}: \Gr_{P,\Ran}\into \ot{\Gr}_{P,\Ran}.
\]

\noindent For convenience, write $j_{\on{Ran},!}:=j_{\on{Ran},!}(\omega_{\on{Gr}_{P,\on{Ran}}})$.

Consider the diagram:
\[\begin{tikzcd}
	& {\widetilde{\mathrm{Gr}}_{P,\mathrm{Ran}}} & {\mathrm{Gr}_{G,\mathrm{Ran}}} \\
	{\Gr_{M,\Ran}} && { \mathfrak{L}^+_{\Ran}M\backslash  \ot{S}^0_{P,\Ran}.}
	\arrow["{\widetilde{p}_{\mathrm{Ran}}}", from=1-2, to=1-3]
	\arrow["{\widetilde{q}_{\mathrm{Ran}}}"', from=1-2, to=2-1]
	\arrow["\pi", from=1-2, to=2-3]
\end{tikzcd}\]

\noindent From here, we may consider the functors:
\[
\on{Jac}_!^M: D(\on{Gr}_{G,\on{Ran}})\to D(\on{Gr}_{M,\on{Ran}}),\;\; \cF \mapsto  \widetilde{q}_{\on{Ran},*}(j_{\on{Ran},!}\overset{!}{\otimes} \widetilde{p}_{\on{Ran}}^!(\cF));
\]
\[
\on{Jac}_{!*}^M: D(\on{Gr}_{G,\on{Ran}})\to D(\on{Gr}_{M,\on{Ran}}),\;\; \cF\mapsto \widetilde{q}_{\on{Ran},*}(\pi^!(\sIC)\overset{!}{\otimes} \widetilde{p}_{\on{Ran}}^!(\cF)).
\]

\noindent These functors descend to functors on the Whittaker categories that we similarly denote by $\on{Jac}_!^M$ and $\on{Jac}_{!*}^M$:
\[
\on{Jac}_!^M: \on{Whit}(D(\on{Gr}_{G,\on{Ran}})\to \on{Whit}(D(\on{Gr}_{M,\on{Ran}}));
\]
\[
\on{Jac}_{!*}^M: \on{Whit}(D(\on{Gr}_{G,\on{Ran}}))\to \on{Whit}(D(\on{Gr}_{M,\on{Ran}})).
\]

\noindent Here, the Whittaker conditions for $G$ and $M$ are understood with respect to a non-degenerate character of $\fL_{\on{Ran}}N$ and its restriction to $\fL_{\on{Ran}}N_M$, respectively, see Section \ref{S:Whittakerness}.

\subsubsection{} Recall that the (factorizable) geometric Casselman-Shalika formula provides an equivalence
\[
\on{CS}_G: \on{Whit}(D(\on{Gr}_{G,\on{Ran}}))\simeq \on{Rep}(\check{G})_{\on{Ran}}
\]

\noindent of factorization categories. On the Langlands dual side, we may consider the functors
\[
C^{\bullet}(\check{\fn}_P,-): \on{Rep}(\check{G})_{\on{Ran}}\to \on{Rep}(\check{M})_{\on{Ran}};
\]
\[
\on{Res}^{\check{G}}_{\check{M}}: \on{Rep}(\check{G})_{\on{Ran}}\to \on{Rep}(\check{M})_{\on{Ran}}
\]

\noindent induced by taking Lie algebra cohomology along $\check{\fn}_P=\on{Lie}(\check{N}_P)$ and restricting along $\check{M}\to \check{G}$, respectively.

Our main result is the following:
\begin{thm}[Corollary \ref{cor: Koszul dual functor}, Theorem \ref{thm: restriction and !* Jacquet are the same}]\label{t:main}
    We have commutative diagrams:

    \[\begin{tikzcd}
	{\on{Whit}(D(\Gr_{G,\Ran}))} & {\on{Whit}(D(\Gr_{M,\Ran})} \\
	{\on{Rep}(\check{G})_{\Ran}} & {\on{Rep}(\check{M})_{\Ran},}
	\arrow["{\on{Jac}^M_!}", from=1-1, to=1-2]
	\arrow["{\on{CS}_G}"', from=1-1, to=2-1]
	\arrow["{\on{CS}_M}", from=1-2, to=2-2]
	\arrow["{C^{\bullet}(\check{\mathfrak{n}}_P,-)}"', from=2-1, to=2-2]
\end{tikzcd}\]

    \[\begin{tikzcd}
	{\on{Whit}(D(\Gr_{G,\Ran}))} & {\on{Whit}(D(\Gr_{M,\Ran}))} \\
	{\on{Rep}(\check{G})_{\Ran}} & {\on{Rep}(\check{M})_{\Ran}.}
	\arrow["{\on{Jac}^M_{!*}}", from=1-1, to=1-2]
	\arrow["{\on{CS}_G}"', from=1-1, to=2-1]
	\arrow["{\on{CS}_M}", from=1-2, to=2-2]
	\arrow["{\on{Res}^{\check{G}}_{\check{M}}}"', from=2-1, to=2-2]
\end{tikzcd}\]

\end{thm}

\subsubsection{} The commutativity of the first diagram was proved when $P=B$ is a Borel subgroup by Raskin in \cite{raskin2021chiral}.\footnote{In truth, Raskin proves that the functor $\on{Jac}_*^M: \on{Whit}(D(\on{Gr}_{G,\on{Ran}}))\to \on{Whit}(D(\on{Gr}_{M,\on{Ran}}))$ defined by replacing $j_{\on{Ran},!}$ with $j_{\on{Ran},*}$ in the definition of $\on{Jac}_!^M$ corresponds to taking Lie algebra \emph{ho}mology on the Langlands dual side.} The proof in \emph{loc.cit} relies on results of \cite{braverman2006deformations}. Our proof of Theorem \ref{t:main} is self-contained except for using a certain vanishing result of \cite[§3]{raskin2021chiral}, the latter result being independent from the rest of \cite{raskin2021chiral}. We also sketch an argument for how to circumvent the vanishing result of Raskin, see §\ref{s:!*jac}.

\subsubsection{$\bL_G$ interchanges Eisenstein series}\label{s:interchangeEIS}

Finally, let us illustrate the importance of Theorem \ref{t:main} by explaining its relevance to the proof in \cite{campbell2024proof} that the geometric Langlands functor
\[
\bL_G: D(\on{Bun}_G)\to \on{IndCoh}_{\cN}(\on{LocSys}_{\check{G}})
\]

\noindent interchanges Eisenstein series.\footnote{We remark that \emph{loc.cit} also establishes that $\bL_G$ commutes with constant term functors, an assertion we do not address at all in this paper.} We do this to orient the reader about how the overall logic of the proof in \cite{campbell2024proof} goes and in particular to highlight that Theorem \ref{t:main} is \emph{the} place one has to deal with Langlands duality.

For simplicity, we ignore issues about temperedness, twists, translations, Cartan involution and shifts in what follows.

There is a natural localization functor
\[\on{Loc}_{\check{G}}: \on{Rep}(\check{G})_{\Ran} \longrightarrow \on{QCoh}(\on{LocSys}_{\check{G}})\]

\noindent with a fully faithful right adjoint $\on{coLoc}_{\check{G}}$. By the vanishing result of \cite{gaitsgory2010generalized}, the action of $\on{Rep}(\check{G})_{\Ran}$ on $D(\Bun_G)$ by Hecke functors factors through an action of $\on{QCoh}(\on{LocSys}_{\check{G}})$. By acting on the Whittaker sheaf, we obtain a functor:
\[\bL_G^{\on{spec}}: \on{QCoh}(\on{LocSys}_{\check{G}}) \longrightarrow D(\Bun_G).\]

\noindent Similarly, we have a Poincaré functor
\[
\on{Whit}(D(\on{Gr}_{G,\on{Ran}}))\xrightarrow{\on{Poinc}_{G,!}} D(\on{Bun}_G)
\]

\noindent that fits into a commutative diagram:
\[\begin{tikzcd}
	{\on{Rep}(\check{G})_{\mathrm{Ran}}} && {\on{Whit}(D(\mathrm{Gr}_{G,\mathrm{Ran}}))} \\
	\\
	{\on{QCoh}(\mathrm{LocSys}_{\check{G}})} && {D(\mathrm{Bun}_G).}
	\arrow["{\mathrm{CS}_G^{-1}}", from=1-1, to=1-3]
	\arrow["{\mathrm{Loc}_{\check{G}}}"', from=1-1, to=3-1]
	\arrow["{\mathrm{Poinc}_{G,!}}", from=1-3, to=3-3]
	\arrow["{\mathbb{L}_G^{\mathrm{spec}}}"', from=3-1, to=3-3]
\end{tikzcd}\]

It suffices to show that the diagram
\begin{equation}\label{d:diagram1}
\begin{tikzcd}
	{D(\mathrm{Bun}_M)} && {\on{QCoh}(\mathrm{LocSys}_{\check{M}})} \\
	\\
	{D(\mathrm{Bun}_G)} && {\on{QCoh}(\mathrm{LocSys}_{\check{G}})}
	\arrow["{\mathbb{L}_M^{\mathrm{spec}}}"', from=1-3, to=1-1]
	\arrow["{\mathrm{CT}_{P,!}}", from=3-1, to=1-1]
	\arrow["{\mathrm{CT}_P^{\mathrm{spec}}}"', from=3-3, to=1-3]
	\arrow["{\mathbb{L}_G^{\mathrm{spec}}}", from=3-3, to=3-1]
\end{tikzcd}
\end{equation}
\noindent commutes. Here, $\on{CT}_{P,!}$ is the functor of $*$-pull, $!$-push along:
\[
\on{Bun}_M\leftarrow \on{Bun}_P\to\on{Bun}_G,
\]

\noindent and $\on{CT}_P^{\on{spec}}$ is the functor of pull-push along
\[
\mathrm{LocSys}_{\check{M}}\leftarrow \mathrm{LocSys}_{\check{P}}\to \mathrm{LocSys}_{\check{G}}.
\]

\noindent The reader should think that the implication
\[
(\on{commutativity}\; \on{of}\; (\ref{d:diagram1}))\to (\bL_G\;\on{interchanges}\; \on{Eisenstein}\;\on{series})
\]

\noindent follows by taking dual functors in (\ref{d:diagram1}), where we identify $D(\on{Bun}_G)$ and $\on{QCoh}(\on{LocSys}_{\check{G}})$ with their own duals through miraculous duality for the former and in the standard way for the latter.\footnote{We have grossly simplified this implication step. As a first point, the results of \cite{lin2023poincare} are needed to make this precise.}

Let us expand (\ref{d:diagram1}) to the following diagram:
\begin{equation}\label{d:diagram2}
\begin{tikzcd}
	{D(\mathrm{Bun}_M)} && {\on{QCoh}(\mathrm{LocSys}_{\check{M}})} \\
	\\
	{D(\mathrm{Bun}_G)} && {\on{QCoh}(\mathrm{LocSys}_{\check{G}})} \\
	\\
	{\on{Whit}(D(\mathrm{Gr}_{G,\mathrm{Ran}}))} && {\on{Rep}(\check{G})_{\mathrm{Ran}}.}
	\arrow["{\mathbb{L}_M^{\mathrm{spec}}}"', from=1-3, to=1-1]
	\arrow["{\mathrm{CT}_{P,!}}", from=3-1, to=1-1]
	\arrow["{\mathrm{CT}_P^{\mathrm{spec}}}"', from=3-3, to=1-3]
	\arrow["{\mathbb{L}_G^{\mathrm{spec}}}", from=3-3, to=3-1]
	\arrow["{\mathrm{Poinc}_{G,!}}", from=5-1, to=3-1]
	\arrow["{\mathrm{Loc}_{\check{G}}}"', from=5-3, to=3-3]
	\arrow["{\mathrm{CS}_G^{-1}}", from=5-3, to=5-1]
\end{tikzcd}
\end{equation}

\noindent We already noted that the lower diagram commutes. Thus, it suffices to show that the outer diagram commutes.

We have commutative diagrams:
\[\begin{tikzcd}
	{\on{Whit}(D(\mathrm{Gr}_{G,\mathrm{Ran}}))} && {\on{Whit}(D(\mathrm{Gr}_{M,\mathrm{Ran}}))} & {\on{Rep}(\check{G})_{\mathrm{Ran}}} && {\on{Rep}(\check{M})_{\mathrm{Ran}}} \\
	\\
	{D(\mathrm{Bun}_G)} && {D(\mathrm{Bun}_M),} & {\on{QCoh}(\mathrm{LocSys}_{\check{G}})} && {\on{QCoh}(\mathrm{LocSys}_{\check{M}}).}
	\arrow["{\mathrm{Jac}_!^M}", from=1-1, to=1-3]
	\arrow["{\mathrm{Poinc}_{G,!}}"', from=1-1, to=3-1]
	\arrow["{\mathrm{Poinc}_{M,!}}", from=1-3, to=3-3]
	\arrow["{C^{\bullet}(\check{\mathfrak{n}}_P,-)}", from=1-4, to=1-6]
	\arrow["{\mathrm{Loc}_{\check{G}}}"', from=1-4, to=3-4]
	\arrow["{\mathrm{Loc}_{\check{M}}}", from=1-6, to=3-6]
	\arrow["{\mathrm{CT}_{P,!}}"', from=3-1, to=3-3]
	\arrow["{\mathrm{CT}_P^{\mathrm{spec}}}"', from=3-4, to=3-6]
\end{tikzcd}\]

\noindent The commutativity of the left diagram was established in \cite{lin2023poincare}, and the commutativity of the right diagram was established in \cite[Thm. 12.2.10]{campbell2024proof}.\footnote{Really one has to precompose $\on{Jac}_!^M$ and $C^{\bullet}(\check{\fn}_P,-)$ with the functor of inserting the unit (cf. \cite[§C.11]{arinkin2024proof}), but we ignore this here.}\label{f:footnote10} Putting it all together, the commutativity of the outer diagram of (\ref{d:diagram2}) follows from our Theorem \ref{t:main}.

\begin{rem}
    We remark here that the commutativity of the right diagram can be proved roughly as follows. The diagram
    \[\begin{tikzcd}
	{\on{Rep}(\check{G})_{\Ran}} & {\on{Rep}(\check{M})_{\Ran}} \\
	{\on{QCoh}(\on{LocSys}_{\check{G}})} & {\on{QCoh}(\on{LocSys}_{\check{M}})}
	\arrow["{\on{Res}^{\check{G}}_{\check{M}}}", from=1-1, to=1-2]
	\arrow["{\on{Loc}_{\check{G}}}"', from=1-1, to=2-1]
	\arrow["{\on{Loc}_{\check{M}}}", from=1-2, to=2-2]
	\arrow["{\iota^*}"', from=2-1, to=2-2]
\end{tikzcd}\]

\noindent commutes, where $\iota: \on{LS}_{\check{M}} \to \on{LS}_{\check{G}}$ denotes the map of inducing local systems along $\check{M}\to\check{G}$, and pullback along this map is obviously compatible with the evident $\on{Rep}(\check{G})_{\Ran}$-actions. Here we view $\on{Rep}(\check{G})_{\Ran}$ as a monoidal category under \emph{external} convolution. Both functors carry actions of $\cO(\check{N}_P)_{\on{Ran}}$ (see §\ref{S:HeckeandKD} below for the definition of the latter). Taking invariants for $\cO(\check{N}_P)_{\on{Ran}}$, we obtain a commutative diagram with target category $\Omega(\check{\fn}_P)_{\on{Ran}}\on{-mod}(\on{QCoh}(\on{LocSys}_{\check{M}}))$.

Finally, one can show that
\[\Omega(\check{\fn}_P)_{\Ran}\text{-mod}(\on{QCoh}(\on{LocSys}_{\check{M}})) \simeq \on{QCoh}(\on{LocSys}_{\check{P}}),\]

\noindent and that under this equivalence, the functor $\on{CT}^{\on{spec}}_P$ corresponds to the composition
\[\on{QCoh}(\on{LocSys}_{\check{G}}) \longrightarrow \Omega(\check{\fn}_P) \text{-mod}(\on{QCoh}(\on{LocSys}_{\check{M}})) \longrightarrow \on{QCoh}(\on{LocSys}_{\check{M}}),\]

\noindent where the second arrow is the forgetful functor.\footnote{We have ignored one subtlety: taking invariants for $\cO(\check{N}_P)_{\on{Ran}}$ as an algebra object in $\on{Rep}(\check{M})_{\on{Ran}}$ with \emph{external} convolution does not recover the functor $C^{\bullet}(\check{\fn}_P,-)$. Rather, this functor is given by taking invariants when $\cO(\check{N}_P)_{\on{Ran}}$ is considered an algebra object in $\on{Rep}(\check{M})_{\on{Ran}}$ with \emph{pointwise} convolution. However, these differ exactly by the functor of inserting the unit, cf. the previous footnote.}
\end{rem}

\subsection{Hecke structures and Koszul duality}\label{S:HeckeandKD}

The second main result of this paper is the construction of Hecke and Drinfeld-Pl\"{u}cker structures on $\sIC$ and related sheaves. We refer to Section \ref{S:PSIC} for more details.

\subsubsection{Koszul duality} Consider $\cO(\check{N}_P)$ as a coalgebra object of $\on{Rep}(\check{M})$. Similarly, consider $C^{\bullet}(\check{\fn}_P)$, the cohomological Chevalley complex of $\check{\fn}_P$, as an algebra object of $\on{Rep}(\check{M})$. We define (co)algebra objects
\[
\cO(\check{N}_P)_{\on{Ran}},\Omega(\check{\fn}_P)_{\on{Ran}}\in\on{Rep}(\check{M})_{\on{Ran}}
\]

\noindent whose $!$-fibers at some $x\in X$ recover $\cO(\check{N}_P),C^{\bullet}(\check{\fn}_P)$, respectively.

\subsubsection{} Consider the semi-infinite category 
\[
\on{SI}_{P,\on{Ran}}=D(\fL_{\on{Ran}}N_P\fL^+_{\on{Ran}}M\backslash \on{Gr}_{G,\on{Ran}}).
\]

\noindent Note that $\on{SI}_{P,\on{Ran}}$ is naturally a module category for $\on{Rep}(\check{M}\times \check{G})_{\on{Ran}}$. As mentioned in §\ref{s:psic}, the semi-infinite IC-sheaf $\sIC$ defines an object of $\on{SI}_{P,\on{Ran}}$.

Koszul duality provides an equivalence of categories
\[
\on{Inv}_{\cO(\check{N}_P)_{\on{Ran}}}: \cO(\check{N}_P)_{\on{Ran}}\on{-comod}(\on{SI}_{P,\on{Ran}})\to \Omega(\check{\fn}_P)_{\on{Ran}}\on{-mod}(\on{SI}_{P,\on{Ran}})
\]

\noindent given by taking invariants for $\cO(\check{N}_P)_{\on{Ran}}$ (see Lemma \ref{l:KD} for a general statement of this type).

\subsubsection{}Denote by $\mathbf{j}_!$ the $!$-pushforward of the dualizing sheaf along the open embedding $S^0_{P,\on{Ran}}\into \widetilde{S}^0_{P,\on{Ran}}$. We have:
\begin{thm}[Proposition \ref{p:restostrata!}]\label{t:KD}

The sheaf $\sIC$ is equipped with a canonical comodule structure for $\cO(\check{N}_P)_{\on{Ran}}$. Moreover, we have:
\[
\on{Inv}_{\cO(\check{N}_P)_{\on{Ran}}}(\sIC)\simeq \mathbf{j}_!.
\]

\noindent In particular, $\mathbf{j}_!$ is equipped with a module structure for $\Omega(\check{\fn}_P)_{\on{Ran}}$.
    
\end{thm}

\subsubsection{Hecke structures}\label{ss:heckestructures} Let $\overline{\check{N}_P\backslash \check{G}}$ be the affinization of $\check{N}_P\backslash \check{G}$. 

For a scheme $X$, we denote by $\cO(X)$ the derived global section of the structure sheaf of $X$. We consider the schemes
\[
\check{G},\; \check{N_P}\backslash \check{G}, \; \overline{\check{N}_P\backslash \check{G}}
\]

\noindent as acted on by $\check{M}\times \check{G}$ in the obvious ways. We get algebras
\[
\cO(\check{G})_{\on{Ran}},\; \cO(\check{N}_P\backslash \check{G})_{\on{Ran}},\; \cO(\overline{\check{N}_P\backslash \check{G}})_{\on{Ran}}\in \on{Rep}(\check{M}\times\check{G})_{\on{Ran}}
\]

\noindent whose fiber at a point $x\in X$ recovers $\cO(\check{G}), \cO(\check{N}_P\backslash \check{G}), \cO(\overline{\check{N}_P\backslash \check{G}})\in \on{Rep}(\check{M}\times \check{G})$, respectively.

\subsubsection{}\label{s:3way} We define three categories consisting of objects of $\on{SI}_{P,\on{Ran}}$ equipped with extra structure related to the action of $\on{Rep}(\check{M}\times \check{G})_{\on{Ran}}$:
\begin{enumerate}
    \item $\on{Hecke}_{\check{M},\check{G}}(\on{SI}_{P,\on{Ran}}):=\cO(\check{G})_{\on{Ran}}\on{-mod}(\on{SI}_{P,\on{Ran}})$.

    For $\cF\in \on{SI}_{P,\on{Ran}}$, we refer to a lift of $\cF$ to $\on{Hecke}_{\check{M},\check{G}}(\on{SI}_{P,\on{Ran}})$ as a \emph{Hecke structure} on $\cF$. Concretely, this amounts to a family of isomorphisms:
    \[
    \cF\star V\simeq \on{Res}^{\check{G}}_{\check{M}}(V)\star \cF, \; V\in\on{Rep}(\check{G})_{\on{Ran}}
    \]

    \noindent satisfying natural higher compatibilities.

    \item $\on{DrPl}_{\check{M},\check{G}}(\on{SI}_{P,\on{Ran}}):=\cO(\overline{\check{N}_P\backslash \check{G}})_{\on{Ran}}\on{-mod}(\on{SI}_{P,\on{Ran}})$.

    For $\cF\in \on{SI}_{P,\on{Ran}}$, we refer to a lift of $\cF$ to $\on{EnhDrPl}_{\check{M},\check{G}}(\on{SI}_{P,\on{Ran}})$ as a \emph{Drinfeld-Plücker structure} on $\cF$. At a point $x\in X$, this amounts to a family of maps:
    \[
    V^{\check{N}_P}\star \cF\to \cF\star V, \; V\in\on{Rep}(\check{G})
    \]

    \noindent satisfying natural higher compatibilities. Here $V^{\check{N}_P}$ is the \emph{underived} invariants of $V$ along $\check{N}_P$.

    \item $\on{EnhDrPl}_{\check{M},\check{G}}(\on{SI}_{P,\on{Ran}}):=\cO(\check{N}_P\backslash \check{G})_{\on{Ran}}\on{-mod}(\on{SI}_{P,\on{Ran}})$.

    For $\cF\in \on{SI}_{P,\on{Ran}}$, we refer to a lift of $\cF$ to $\on{EnhDrPl}_{\check{M},\check{G}}(\on{SI}_{P,\on{Ran}})$ as an \emph{enhanced Drinfeld-Plücker structure} on $\cF$. Concretely, this amounts to a family of isomorphisms:
    \[
    \cF\star V\simeq C^{\bullet}(\check{\fn}_P,V)\underset{\Omega(\check{\fn}_P)_{\on{Ran}}}{\star} \cF, \; V\in\on{Rep}(\check{G})_{\on{Ran}}
    \]

    \noindent satisfying natural higher compatibilities.
\end{enumerate}

\begin{thm}\label{t:loc_Hecke}
$\sIC$ is canonically equipped with a Hecke structure. Moreover, $\mathbf{j}_!$ is canonically equipped with an enhanced Drinfeld-Plücker structure.
\end{thm}

\begin{rem}
In fact, contrary to \cite{gaitsgory2018semi}\cite{gaitsgory2021semi}, we \emph{define} $\sIC$ by the requirement that it possesses a Hecke structure. Since we have avoided defining $\sIC$ in this introduction, we keep the first part of the above result as a theorem for simplicity. On the other hand, that $\mathbf{j}_!$ possesses an enhanced Drinfeld-Plücker structure is quite non-trivial and follows from Theorem \ref{t:KD}.

This illustrates one of the main points of this paper rather well: instead of trying to define the enhanced Drinfeld-Plücker structure on $\mathbf{j}_!$ directly, we define another sheaf\footnote{In this case $\on{Inv}_{\cO(\check{N}_P)_{\on{Ran}}}(\sIC)$.} with an enhanced Drinfeld-Plücker structure and identify it with $\mathbf{j}_!$.\footnote{Which is in principle straightforward: to give an isomorphism $\cF\simeq\mathbf{j}_!$, one has to identify $\cF$ with the dualizing sheaf upon restriction to $S_{P,\on{Ran}}^0$ and show that the $*$-restriction of $\cF$ to $\widetilde{S}^0_{P,\on{Ran}}\setminus S_{P,\on{Ran}}^0$ vanishes.} If, for example, one tried to construct isomorphisms
\[
 \mathbf{j}_!\star V\simeq C^{\bullet}(\check{\fn}_P,V)\underset{\Omega(\check{\fn}_P)_{\on{Ran}}}{\star} \mathbf{j}_!, \; V\in\on{Rep}(\check{G})_{\on{Ran}}
\]

\noindent directly, this would be very subtle.

\subsubsection{Where do these structures come from?}

We have highlighted the advantage of our approach to constructing, say, Hecke structures on sheaves of interest compared to the approach taken previously in the literature, as this required constructing essentially all the structures 'by hand'. This begs the question of whether we have completely removed the necessity of doing things by hand.

This is not quite true. After all, our statements involve Langlands duality,\footnote{Langlands duality that does not come from geometric Satake, that is.} and so combinatorics has to appear at some point. In our case, this appears when having to construct a Drinfeld-Plücker structure on the delta sheaf along the unit section $\on{Ran}\to \fL^+_{\on{Ran}}M\backslash \on{Gr}_{G,\on{Ran}}$ from which we bootstrap everything. However, constructing this Drinfeld-Plücker structure is very simple, and we regard it as the minimal 'combinatorical' input needed to obtain the results of this paper.

\end{rem}

\subsection{Local-to-global comparisons}

In this subsection, we describe how to relate the semi-infinite IC-sheaf $\sIC$ to the IC-sheaf of $\widetilde{\on{Bun}}_P$. We refer to Section \ref{S:LOCTOGLOB} for more details.

\subsubsection{} We let
\[
 \ot{\Gr}_{P,\Bun_M}\to \on{Bun}_M
\]

\noindent denote a relative version of $\widetilde{S}^0_{P,\on{Ran}}$ living over $\on{Bun}_M$, see §\ref{s:relsemiinf}. By definition, it comes with a canonical map:
\begin{equation}\label{eq:projectionmap}
\ot{\Gr}_{P,\Bun_M}\to \fL^+_{\on{Ran}}M\backslash \tildePloc.
\end{equation}

We let ${}_{\on{Bun}_M}\sIC, {}_{\on{Bun}_M}\mathbf{j}_!$ denote the $!$-pullbacks of $\sIC, \mathbf{j}_!$, respectively.

\subsubsection{} We have a natural map:
\[
\pi_P: \ot{\Gr}_{P,\Bun_M}\to \widetilde{\on{Bun}}_P.
\]

\begin{thm}[Theorem \ref{t:IC-to-hecke}]\label{t:LOCTOGLOB} 
We have a canonical isomorphism:
\[
\pi_P^!(\on{IC}_{\widetilde{\on{Bun}}_P})[\on{dim}(\on{Bun}_P)]\simeq \sIC.
\]

\noindent Moreover, the counit map
\[
\pi_{P,!}\circ \pi_P^!(\on{IC}_{\widetilde{\on{Bun}}_P})\to \on{IC}_{\widetilde{\on{Bun}}_P}
\]

\noindent is an isomorphism.
\end{thm}

\begin{rem}
Here, $\on{dim}(\on{Bun}_P)$ denotes the locally constant function on $\on{Bun}_P$ that gives the dimension of a given connected component.
\end{rem}

\subsubsection{} Consider the prestack $\widetilde{\on{Bun}}_{P,\on{pol}}$ over $\on{Ran}$ that parametrizes a point $x_I\in\on{Ran}$ and a generalized $P$-reduction on $X$ that is non-degenerate away from $x_I$ (see §\ref{s:tildepol} for a precise definition).

The main property of $\widetilde{\on{Bun}}_{P,\on{pol}}$ is that it carries Hecke modifications for both $M$ and $G$. That is, the category $D(\widetilde{\on{Bun}}_{P,\on{pol}})$ is a module category for $\on{Rep}(\check{M}\times \check{G})_{\on{Ran}}$. This allows us to talk about Hecke structures and (enhanced) Drinfeld-Plücker structures of objects of $D(\widetilde{\on{Bun}}_{P,\on{pol}})$ as in §\ref{s:3way}.

\subsubsection{}
We have an open embedding:
\[
j_{P,\on{pol}}: \on{Bun}_P\times\on{Ran}\to \widetilde{\on{Bun}}_{P,\on{pol}}.
\]

\noindent We define $\mathbf{j}_!^{\on{glob}}:=j_{P,\on{pol},!}(\omega_{\on{Bun}_P\times\on{Ran}})\in D(\widetilde{\on{Bun}}_{P,\on{pol}})$.

\subsubsection{} We have a closed substack
\[
i_{P,\on{pol}}: \widetilde{\on{Bun}}_{P,\on{zer}}\into \widetilde{\on{Bun}}_{P,\on{pol}}
\]

\noindent defined by requiring the Drinfeld-Plücker maps be regular. It comes with a tautological map forgetting the point of $\on{Ran}$:
\[
\on{oblv}_{\on{zer}}: \widetilde{\on{Bun}}_{P,\on{zer}}\to \widetilde{\on{Bun}}_{P}.
\]

We define
\[
\on{IC}_{\widetilde{\on{Bun}}_{P,\on{pol}}}:=i_{P,\on{pol},*}\circ \on{oblv}_{\on{zer}}^!(\on{IC}_{\widetilde{\on{Bun}}_P})\in D(\widetilde{\on{Bun}}_{P,\on{pol}}).
\]

\noindent As a consequence of Theorem \ref{t:LOCTOGLOB} and the results in §\ref{S:HeckeandKD}, we obtain:
\begin{thm}\label{t:glob_KD}
    
$\on{IC}_{\widetilde{\on{Bun}}_{P,\on{pol}}}$ is a equipped with a canonical comodule structure for $\cO(\check{N}_P)_{\on{Ran}}$. Moreover, we have:
\[
\on{Inv}_{\cO(\check{N}_P)_{\on{Ran}}}(\on{IC}_{\widetilde{\on{Bun}}_{P,\on{pol}}})\simeq \mathbf{j}_!^{\on{glob}}.
\]
\noindent In particular, $\mathbf{j}_!$ is equipped with a module structure for $\Omega(\check{\fn}_P)_{\on{Ran}}$.
\end{thm}

\begin{thm}\label{t:glob_Hecke}
$\on{IC}_{\widetilde{\on{Bun}}_{P,\on{pol}}}$ is canonically equipped with a Hecke structure. Moreover, $\mathbf{j}_!^{\on{glob}}$ is canonically equipped with an enhanced Drinfeld-Plücker structure.
\end{thm}

\subsection{Organization of the paper}

In Section \ref{S:Notation}, we recall notations and conventions.

In Section 3, we introduce the main geometric players that appear in this paper.

In Section 4, we introduce and study the parabolic semi-infinite IC-sheaf.

In Section 5, we prove the local-to-global comparison between the semi-infinite IC-sheaf and the IC-sheaf on $\widetilde{\on{Bun}}_P$.

In Section 6, we prove the compatibility stated in §\ref{S:JACquet} between restriction of representations and the Jacquet functor under the geometric Casselman-Shalika equivalence.

\subsection{Acknowledgements}

We thank Justin Campbell, Gurbir Dhillon, Linus Hamann, David Hansen and Sergey Lysenko for many stimulating and insightful discussions.

We are especially grateful to Dennis Gaitsgory, Ivan Mirkovi\'{c} and Sam Raskin, conversations with whom had a significant influence on the development of this paper.

For the second author: the research project is implemented in the framework of H.F.R.I call "Basic research Financing (Horizontal support of all Sciences)" under the National Recovery and Resilience Plan "Greece 2.0" funded by the European Union - NextGenerationEU (H.F.R.I Project Number: 16785).

\section{Notation}\label{S:Notation}
In this section, we establish notation and conventions used throughout the paper.

\subsection{Categorical conventions and base field} 

\subsubsection{} Throughout, we work over an algebraically closed field $k$ of characteristic zero. We freely use the language of higher category theory and higher algebra in the sense of \cite{lurie2009higher}, \cite{lurie2017higher}, \cite{gaitsgory2019study}. Throughout, by a (DG) category, we mean a $k$-linear presentable stable $(\infty,1)$-category. We denote by $\on{DGCat}_{\on{cont}}$ the category of DG-categories in which the morphisms are colimit-preserving functors. Henceforth, we refer to DG-categories simply as \emph{categories}.

\subsubsection{}For a category $\sC$ equipped with a t-structure, we let $\sC^{\leq 0}$ and $\sC^{\geq 0}$ the subcategories of connective and coconnective objects, respectively. We denote by $\sC^{\heartsuit}=\sC^{\leq 0}\cap \sC^{\geq 0}$ the heart of the t-structure.

\subsection{D-modules and functoriality}

\subsubsection{} Let $\cY$ be a prestack locally almost of finite type in the sense of \cite{gaitsgory2019study}. Following \cite{gaitsgory2017study}, we denote by $D(\cY)$ the (DG-)category of D-modules on $\cY$.

\subsubsection{Functoriality} Let $f:\cX\to \cY$ be a map of prestacks locally almost of finite type. We have a pullback functor
\[
f^!: D(\cY)\to D(\cX).
\]

\noindent Whenever its left adjoint is defined, we denote it by $f_!$.

If $f$ is ind-representable, we similarly have a (continuous) pushforward functor
\[
f_{*}: D(\cX)\to D(\cY).
\]

\noindent Whenever its left adjoint is defined, we denote it by $f^{*}$. Recall that both $f_!$ and $f^{*}$ are defined on holonomic D-modules.

\subsection{Lie theory notation}

\subsubsection{} 
Let $G$ be a connected reductive complex algebraic group over $k$. We assume throughout that $G$ has a simply-connected derived subgroup. This is a standard assumption that simplifies the definition of Drinfeld's compactifications and Zastava spaces. To remove this hypothesis, consult \cite[§4.1]{arkhipov2005modules} or \cite[§7]{schieder2015harder}.

We choose a splitting $T\subset B$ and an opposing Borel $B^{-}$ so that $B^{-}\cap B=T$. We let $N, N^{-}$ denote the unipotent radicals of $B,B^-$. Denote by $\fg,\fb^{-},\fn,\fn^{-},\ft$ the corresponding Lie algebras. We let $\check{G}$ be the Langlands dual group of $G$.

\subsubsection{} We let $\check{\Lambda}, \Lambda$ (resp. $\check{\Lambda}^+, \Lambda^+$) be the lattice of characters and cocharacters of $T$ (resp. dominant characters and dominant cocharacters). The reason for choosing this notation is that cocharacters appear more in this paper than characters, making it more convenient to reserve the 'checks' for the characters. Let $W$ be the finite Weyl group of $G$.

We denote the coroots of $G$ by $\Phi$, and we write $\sJ=\sJ_G$ for the set of positive simple coroots of $G$ and by $\sJ^{\on{neg}}$ the set of negative simple coroots. For each $i\in \sJ$, we have a corresponding simple coroot $\alpha_i$ and simple root $\check{\alpha}_i$.

\subsubsection{Parabolic notation}
Let $P\subset G$ be a standard parabolic subgroup with unipotent radical $N_P$ and Levi quotient $M$. We let $\fp,\fn_P,\fm$ be the corresponding Lie algebras. Write $\sJ_M\subset \sJ$ for the subset of the Dynkin diagram corresponding to $M$. We let $\check{\Lambda}_{G,P}, \Lambda_{G,P}$ be the lattices of characters and cocharacters, respectively, for the torus $M^{\on{ab}}:=M/[M,M]$. There is a natural map
\begin{equation}\label{eq:latticequotient}
\Lambda\to \Lambda_{G,P}
\end{equation}

\noindent induced by $T\to M^{\on{ab}}$. By simply-connectedness of $[G,G]$, the kernel of the above map is spanned by $\alpha_i$, where $i\in \sJ_M\subset \sJ$ is a simple coroot corresponding to a vertex in the Dynkin diagram for $M$. We let $\Lambda_{G,P}^{\on{pos}}\subset \Lambda_{G,P}$ be the monoid spanned by linear combinations with non-negative coefficients of the images of $\alpha_i$ for $i\in\sJ\setminus\sJ_M$ under the map (\ref{eq:latticequotient}). Write $\Lambda_{G,P}^{\on{neg}}:=-\Lambda_{G,P}^{\on{pos}}$.

\section{The main geometric players}

Fix a smooth projective curve $X$ over $k$.

\subsection{Affine Grassmannians and configuration spaces.} 
In this subsection, we review the various versions of the factorizable affine grassmannian needed in the paper. We refer to \cite{braverman2002intersection} for a more detailed discussion of the spaces introduced. We freely use the language of chiral/factorization theory as in \cite{raskin2015chiral}.

\subsubsection{} Associated to the lattice $\Lambda_{G,P}$ we have the scheme $\on{Conf}_{G,P}$ parameterizing divisors on $X$ with coefficients in $\Lambda^{\on{neg}}_{G,P}$. It splits as a disjoint union
\[
\on{Conf}_{G,P}=\underset{\theta\in\Lambda_{G,P}^{\on{neg}}}{\coprod} X^{\theta},
\]

\noindent where $X^{\theta}$ parameterizes $\Lambda^{\on{neg}}_{G,P}$-valued divisors on $X$ of total degree $\theta$. Note that $\on{Conf}_{G,P}$ is naturally a monoid under addition of divisors.

\subsubsection{}\label{s:+gr} Denote by $\on{Gr}_{M,x}$ the affine Grassmannian for $M$ at a point $x\in X$. For an $M$-bundle $\sP_M$ and an $M$-representation $W$, we let $W_{\sP_M}$ be the induced vector bundle on $X$. Mostly, we will consider representations of the form $W=V^{N_P}$ for a $G$-representation $V$.

\subsubsection{} We let $\on{Gr}_{M,x}^+$ be the subscheme of $\on{Gr}_{M,x}$ parameterizing $(\sP_M,\phi)\in\on{Gr}_{M,x}$, where:
\begin{itemize}

    \item $\sP_M$ is an $M$-bundle on $X$.

    \item $\phi$ is a trivialization away from $x$
\[
\phi: \sP_{M\vert X-x}\xrightarrow{\simeq} \sP^0_{M\vert X-x}
\]

\noindent such that for every $G$-representation $V$, the induced map of vector bundles
\[
V^{N_P}_{\sP_{M\vert X-x}}\xrightarrow{\simeq} V^{N_P}_{\sP^0_{M\vert X-x}}
\]

\noindent is regular on $D_x$.

\end{itemize}

\subsubsection{} The connected components of $\on{Gr}_M$ are indexed by $\pi_1^{\on{alg}}(M):=\Lambda/\bZ \sJ_M$, the algebraic fundamental group of $M$ (i.e., the coweight lattice of $M$ modulo its coroot lattice). Since $[G,G]$ is simply-connected, we have $\pi_1^{\on{alg}}(M)=\Lambda_{G,P}$. For $\theta\in\Lambda_{G,P}$, we write $\on{Gr}_M^{\theta}$ for the corresponding connected component.

Moreover, let $\on{Gr}_M^{+,\theta}=\on{Gr}_M^{+}\cap \on{Gr}_M^{\theta}$. By construction (see e.g. \cite[Prop. 6.2.3]{braverman1999geometric}), we have
\[
\on{Gr}_M^{+,\theta}\neq \emptyset
\]

\noindent if and only if $\theta\in \Lambda_{G,P}^{\on{neg}}$.

\subsubsection{}\label{s:Gr_M+} Denote by $\on{Gr}_{M,\on{Ran}}$ the Beilinson-Drinfeld affine Grassmannian living over $\on{Ran}=\on{Ran}_X$, the Ran space of $X$. We may similarly define a factorization space $\on{Gr}_{M,\on{Ran}}^{+}$ over $\on{Ran}$ with fiber at $x\in X$ given by $\on{Gr}_{M,x}^+$.

\subsubsection{}\label{sec: arcs and loops} For an algebraic group $H$, we will also consider the group prestacks
\[\mathfrak{L}^+_{\Ran}H \,\,\, \text{and} \,\,\, \mathfrak{L}_{\Ran}H\]

\noindent of \emph{arcs} and \emph{loops} into $H$ over $\Ran$, respectively. More precisely, the fiber of $\mathfrak{L}^+_{\Ran}H$ over a point $x_I$ of $\Ran$ is a map from the formal completion of $x_I$ along $X$ to $H$, and the fiber of $\mathfrak{L}_{\Ran}H$ over a point $x_I$ is a map from the complement of $x_I$ in the \emph{affinization} of the formal completion of $x_I$ along $X$ into $H$. 

Note we have an equivalence
\[\Gr_{H,\Ran} \overset{\sim}{\longrightarrow} \mathfrak{L}_{\on{Ran}}(H)/\mathfrak{L}^+_{\on{Ran}}(H)\]

\noindent as spaces over $\Ran$. Throughout this paper, we sheafify quotients of prestack by algebraic groups in the fppf topology.

\subsubsection{Version living over configuration space}\label{s: Gr_M+Conf}
Denote by $\on{Gr}_{M,\on{Conf}}$ the factorization space over $\on{Conf}_{G,P}$ parameterizing triples $(D,\sP_M,\phi)$, where $D\in\on{Conf}_{G,P}$ is a $\Ln$-valued divisor on $X$, $\sP_M$ is an $M$-bundle, and $\phi$ is a trivialization of $\sP_M$ away from the support of $D$.\footnote{We write $\on{Gr}_{M,\on{Conf}}$ instead of the more cumbersome notation $\on{Gr}_{M,\on{Conf}_{G,P}}$.}

\subsubsection{}
    For an algebraic group $H$ we have versions $\mathfrak{L}^+_{\on{Conf}}H$ and $\mathfrak{L}_{\on{Conf}}H$ of the arc and loop spaces, living over $\on{Conf}_{G,P}$. These split as disjoint unions:
    \[\mathfrak{L}^+_{\on{Conf}}H = \coprod_{\theta \in \Lambda_{G,P}^{\on{neg}}} \mathfrak{L}^+_{X^{\theta}}H \,\,\, \text{and} \,\,\, \mathfrak{L}_{\on{Conf}}H = \coprod_{\theta \in \Lambda_{G,P}^{\on{neg}}} \mathfrak{L}_{X^{\theta}}H.\]

    \noindent Here the index $X^{\theta}$ is used to denote the connected component living over the component $X^{\theta}$ of $\on{Conf}_{G,P}$.
    
    Concretely, $\mathfrak{L}_{\on{Conf}}H$ parameterizes a divisor $D\in\on{Conf}_{G,P}$ and a map from the punctured disk around the support of $D$ to $H$. The arc group $\mathfrak{L}^+_{\on{Conf}}H$ is defined similarly. 
    
\subsubsection{}\label{s:weirdplus} Take $H=M$. We define the group prestack $\mathfrak{L}_{\on{Conf}}M^+\subset \mathfrak{L}_{\on{Conf}}M$ by the requirement that the diagram
\[\begin{tikzcd}
	{\mathfrak{L}_{\on{Conf}}M^+} && {\mathfrak{L}_{\on{Conf}}M} \\
	\\
	{\mathrm{Gr}^+_{M,\mathrm{Conf}}} && {\mathrm{Gr}_{M,\mathrm{Conf}}}
	\arrow[from=1-1, to=1-3]
	\arrow[from=1-3, to=3-3]
	\arrow[from=1-1, to=3-1]
	\arrow[from=3-1, to=3-3]
\end{tikzcd}\]

\noindent be Cartesian. For $\theta\in \Lambda_{G,P}^{\on{neg}}$, we let $\mathfrak{L}_{X^{\theta}}M^{+}$ be the pullback of $\mathfrak{L}_{\on{Conf}}M^+$ to $X^{\theta}$.

\subsubsection{} We have a natural factorization space $\on{Gr}_{M,\on{Conf}}^+$ over $\on{Conf}_{G,P}$ with fiber at $\theta\cdot x$ given by $\on{Gr}_{M,x}^{+,\theta}$.\footnote{There is another space which deserves to be called $\on{Gr}_{M,\on{Conf}}^+$; namely the factorization space over $\on{Conf}_{G,P}$ with fiber at $\theta\cdot x$ given by $\on{Gr}_{M,x}^{+}$. This space will not appear in the paper, however, and so no confusion is likely to occur.} For example, if $P=B$ is a Borel subgroup, then
\[
(\on{Gr}_{M,\on{Conf}}^+)^{\on{red}}\simeq\on{Conf}_{G,B}.
\]

We let:
\[
\on{Gr}^+_{M,X^{\theta}}=\on{Gr}^+_{M,\on{Conf}}\underset{\on{Conf}_{G,P}}{\times} X^{\theta}.
\]

\subsubsection{} Observe that we have a natural action
\[
\on{Ran}\curvearrowright \on{Gr}_{M,\on{Ran}}
\]

\noindent given by $\underline{y}\cdot (\underline{x},\sP_M,\phi)=(\underline{y}\cup\underline{x},\sP,\phi)$, where $\underline{x},\underline{y}\in\on{Ran}$. Moreover, the action preserves $\on{Gr}_{M,\on{Ran}}^+$. In other words, $\on{Gr}_{M,\on{Ran}}$ and $\on{Gr}_{M,\on{Ran}}^+$ are \emph{unital} factorization spaces.

We similarly have an action:
\[
\on{Conf}_{G,P}\curvearrowright \on{Gr}_{M,\on{Conf}}.
\]

\subsubsection{Warning}
Note that this action does not preserve $\on{Gr}_{M,\on{Conf}}^+$.

\subsubsection{} We have the following basic lemma:
\begin{lem}\label{l:fppf}
The prestacks
\[
\on{Gr}_{M,\on{Ran}}/\on{Ran},\;\;\on{Gr}_{M,\on{Conf}}/\on{Conf}_{G,P}
\]

\noindent become isomorphic after sheafification in the fppf topology. Under this isomorphism, the composition
\[
\on{Gr}_{M,\on{Conf}}^+\to \on{Gr}_{M,\on{Conf}}/\on{Conf}_{G,P}\simeq \on{Gr}_{M,\on{Ran}}/\on{Ran}
\]

\noindent maps isomorphically onto $\on{Gr}_{M,\on{Ran}}^+/\on{Ran}$.
\end{lem}

\begin{proof}
Consider the prestack $\on{Gr}_{M,\on{gen}}$ whose $S$-points for an affine scheme $S$ parameterize an $M$-bundle $\sP_M$ on $S\times X$ together with a trivialization $\phi$ on \emph{some} domain $U\subset S\times X$. We remind that a domain is an open subscheme such that for every $k$-point $s\in S$, its restriction to $X$ is dense (equivalently, non-empty). We have natural maps:
\[
\on{Gr}_{M,\on{Ran}}/\on{Ran}\to\on{Gr}_{M,\on{gen}}\leftarrow \on{Gr}_{M,\on{Conf}}/\on{Conf}_{G,P}.
\]

We claim that these maps become isomorphism after sheafification in the fppf topology. It suffices to check this for $S$-points where $S$ is an affine scheme of finite type. By \cite[Lemma 5.5.1]{barlev2012d}, every domain $U\subset S\times X$ is fppf locally the complement of a union of graphs. Similarly, by Lemma 3.2.7 in \emph{loc.cit}, every domain is Zariski locally the complement of a divisor.

To prove the second part of the lemma, define $\on{Gr}_{M,\on{gen}}^+$ in the obvious way. We have maps
\begin{equation}\label{eq:corrplus}
\on{Gr}_{M,\on{Ran}}^+/\on{Ran}\to\on{Gr}_{M,\on{gen}}^+\leftarrow \on{Gr}_{M,\on{Conf}}^+.
\end{equation}

\noindent By construction, the left-most map is an isomorphism after fppf sheafification. Let us construct an explicit inverse to the right-most map. Thus, let $(U\subset S\times X, \sP_M,\phi)$ be an $S$-point of $\on{Gr}_{M,\on{gen}}^+$. We get an associated $\Lp$-valued divisor $D$ on $X$. Concretely, for every character $\check{\nu}$ of $M$ of the form $\check{\nu}=V^{N_P}$ for some $G$-representation $V$, $D$ is the divisor such that
\[
\sL_{\sP_M}^{\check{\nu}}\simeq \cO_{S\times X}(-\langle D,\check{\nu}\rangle)
\]

\noindent compatibly with the embedding into $\cO_{S\times X}\simeq \sL^{\check{\nu}}_{\sP_M^0}$. We claim that for every $G$-representation $V$, the embedding of coherent sheaves
\[
\phi_V: V^{N_P}_{\sP_M}\into V^{N_P}_{\sP_M^0}
\]

\noindent is an isomorphism away from the support of $D$. Indeed, it suffices to check this after taking determinants of both vector bundles, where the assertion in turn follows from the definition of $D$.

This provides a map
\[
\on{Gr}_{M,\on{gen}}^+\to \on{Gr}_{M,\on{Conf}}^+,
\]

\noindent which is easily seen to be the inverse of the right-most map of (\ref{eq:corrplus}).

\end{proof}

\subsection{Factorization via twisted arrows: categories}
In the next subsections, we review the construction of factorization algebras (resp. factorization categories) from a commutative algebra (resp. symmetric monoidal category). Many of the constructions follow \cite[§2]{gaitsgory2021semi}.

\subsubsection{Twisted arrows} Denote by $\on{fSet}$ the category of non-empty finite sets under surjection. Note that any surjection $\phi: I\onto J$ induces a closed embedding
\[
\Delta_{\phi}: X^J\into X^I
\]

\noindent in the natural way.

\subsubsection{} We let $\on{TwArr}$ be the category whose objects are surjections
\[
\phi: I\onto J,
\]

\noindent where $I,J\in\on{fSet}$. Morphisms between $(I_1\onto J_1)$ and $(I_2\onto J_2)$ are given by commutative diagrams
\begin{equation}\label{d:twarrmor}
\begin{tikzcd}
	{I_1} && {J_1} \\
	\\
	{I_2} && {J_2.}
	\arrow["{\phi_1}", two heads, from=1-1, to=1-3]
	\arrow["{\psi_I}"', two heads, from=1-1, to=3-1]
	\arrow["{\phi_2}", two heads, from=3-1, to=3-3]
	\arrow["{\psi_J}"', two heads, from=3-3, to=1-3]
\end{tikzcd}
\end{equation}

\noindent Both $\psi_I$ and $\psi_J$ are required to be surjective.

\subsubsection{}\label{s:unitality} Let $\sC$ be a unital symmetric monoidal category. We may construct a unital factorization category $\on{Fact}(\sC)_{\on{Ran}}$ over $\on{Ran}$ whose fiber at $x\in X$ is $\sC$ itself. Namely, let:
\begin{equation}\label{eq:factcatdef}
\on{Fact}(\sC)_{\on{Ran}}:=\underset{(I\onto J)\in\on{TwArr}}{\on{colim}} \sC^{\otimes I}\otimes D(X^J),
\end{equation}

\noindent where for any commutative diagram (\ref{d:twarrmor}), the corresponding functors between the categories are induced by:
\begin{itemize}
    \item $\Delta_{\psi_J},*: D(X^{J_1})\to D(X^{J_2})$.

    \item The natural map $\sC^{\otimes I_1}\to \sC^{\otimes I_2}$ coming from $\psi_I$ and the monoidal structure on $\sC$.
\end{itemize}

\noindent It is easy to see that the unital structure on $\sC$ as a symmetric monoidal category defines a natural unital structure on $\on{Fact}(\sC)_{\on{Ran}}$ as a factorization category. Concretely, for $K\in\on{fSet}$, the unital structure is induced by the maps
\[
D(X^K)\otimes \sC^{\otimes I}\otimes D(X^J)\to \sC^{\otimes I\sqcup K}\otimes D(X^{J\sqcup K}),
\]

\noindent where $\sC^{\otimes I}\to \sC^{\otimes I\sqcup K}$ is given by inserting the unit and $D(X^K)\otimes D(X^J)\to D(X^{J\sqcup K})$ is exterior product. These combine to give an action
\[
D(\on{Ran})\curvearrowright \on{Fact}(\sC)_{\on{Ran}},
\]

\noindent where $D(\on{Ran})$ is considered as a symmetric monoidal category under convolution (i.e., pushforward along the map $\on{Ran}\times \on{Ran}\to\on{Ran},\;\; (\underline{x},\underline{y})\mapsto \underline{x}\cup\underline{y}$).

\subsubsection{Variant: Fixing I}

Let $I\in\on{fSet}$. We define the category $\on{TwArr}_{I/}$ whose objects are maps
\[
I\onto J\onto K,
\]

\noindent and where morphisms are commutative diagrams:
\[\begin{tikzcd}
	I && {J_1} && {K_1} \\
	\\
	I && {J_2} && {K_2.}
	\arrow[two heads, from=1-1, to=1-3]
	\arrow[two heads, from=1-3, to=1-5]
	\arrow["{\mathrm{id}}"', from=1-1, to=3-1]
	\arrow[two heads, from=3-1, to=3-3]
	\arrow[two heads, from=3-3, to=3-5]
	\arrow[two heads, from=1-3, to=3-3]
	\arrow[two heads, from=3-5, to=1-5]
\end{tikzcd}\]

It is not difficult to see that we have an isomorphism:
\begin{equation}\label{eq:twarri}
\on{Fact}(\sC)_I:=\on{Fact}(\sC)_{\on{Ran}} \underset{D(\on{Ran})}{\otimes}D(X^I)\simeq \underset{(I\onto J\onto K)\in\on{TwArr}_{I/}}{\on{colim}} \sC^{\otimes J}\otimes D(X^K).
\end{equation}

\subsubsection{Pointwise convolution}\label{s:pointwise}
Let $I\onto J\onto K$. Observe that each term
\[
\sC^{\otimes J}\otimes D(X^K)
\]

\noindent in the colimit (\ref{eq:twarri}) is symmetric monoidal category in $D(X^I)\on{-mod}$. Since the transition functors are moreover symmetric monoidal, we see that $\on{Fact}(\sC)_I$ is equipped with a symmetric monoidal structure over $X^I$.

Taking the limit over $I\in\on{fSet}$, we see that $\on{Fact}(\sC)_{\on{Ran}}$ is equipped with a symmetric monoidal structure over $\on{Ran}$. We refer to this monoidal structure
\[
\on{Fact}(\sC)_{\on{Ran}}\underset{D(\on{Ran}_X)}{\otimes} \on{Fact}(\sC)_{\on{Ran}}\to \on{Fact}(\sC)_{\on{Ran}}
\]

\noindent as \emph{pointwise convolution}. Note in particular that the pointwise convolution of two elements in $\on{Fact}(\sC)_{\on{Ran}}$ supported on $x,y\in X$, respectively, with $x\neq y$ is zero.

Unless stated otherwise, we will only consider $\on{Fact}(\sC)_{\on{Ran}}$ with its monoidal structure given by poinstwise convolution.

\subsubsection{} We have the following basic lemma:
\begin{lem}
The functor 
\[
\sC\mapsto \on{Fact}(\sC)_{\on{Ran}}
\]

\noindent is symmetric monoidal. That is, we have a canonical symmetric monoidal equivalence:
\[
\on{Fact}(\sC)_{\on{Ran}}\underset{D(\on{Ran}_X)}{\otimes} \on{Fact}(\sD)_{\on{Ran}}\xrightarrow{\simeq} \on{Fact}(\sC\otimes \sD)_{\on{Ran}}.
\]
\end{lem}

\begin{proof}
It is easy to see that we have a canonically defined symmetric monoidal functor
\[
\on{Fact}(\sC)_{\on{Ran}}\underset{D(\on{Ran}_X)}{\otimes} \on{Fact}(\sD)_{\on{Ran}}\to \on{Fact}(\sC\otimes \sD)_{\on{Ran}},
\]

\noindent and we need to check it is an equivalence. Since the functor is a functor of factorization categories, we may check the map is an isomorphism over $X\subset \on{Ran}$. But 
\[
D(X)\underset{D(\on{Ran}_X)}{\otimes}\on{Fact}(\sC)_{\on{Ran}}\simeq \sC\otimes D(X)
\]

\noindent by (\ref{eq:twarri}), and similarly for $\sD$, so this is clear.
\end{proof}

\begin{rem}
The above lemma also formally gives the pointwise symmetric monoidal structure on $\on{Fact}(\sC)_{\on{Ran}}$.

\end{rem}

\subsection{Factorization via twisted arrows: algebras} 

\subsubsection{} \label{s:factalg} Let $\sA\in\on{CommAlg}^{\on{un}}(\sC)$ be a unital commutative algebra object in $\sC$. We may associate a unital factorization algebra $\on{Fact}^{\on{alg}}(\sA)$ in $\on{Fact}(\sC)_{\on{Ran}}$ whose $!$-fiber at $x\in X$ is $\sA$ itself. By $!$-fiber of $\on{Fact}^{\on{alg}}(\sA)$ at $x$, we mean the functor
\[
\on{Fact}(\sC)_{\on{Ran}}\simeq D(\on{Ran}_X)\underset{D(\on{Ran}_X)}{\otimes} \on{Fact}(\sC)_{\on{Ran}}\xrightarrow{i_x^!\otimes \on{id}}D(\lbrace x\rbrace)\underset{D(\on{Ran}_X)}{\otimes} \on{Fact}(\sC)_{\on{Ran}}\simeq \sC
\]

\noindent evaluated on $\on{Fact}^{\on{alg}}(\sA)$, where $i_x: \lbrace x\rbrace\to \on{Ran}$ is the inclusion.

Namely, we set:
\[
\on{Fact}^{\on{alg}}(\sA)_{\on{Ran}}:=\underset{(I\onto J)\in\on{TwArr}}{\on{colim}} \sA^{\boxtimes I}\boxtimes \omega_{X^J}.
\]

\noindent Here, $\sA^{\boxtimes I}\boxtimes \omega_{X^J}\in \sC^{\otimes I}\otimes D(X^J)$. For a commutative diagram (\ref{d:twarrmor}), the colimit is formed with respect to the induced maps:
\begin{itemize}
    \item $\Delta_{\psi_J,*}(\omega_{X^{J_1}})\to \omega_{X^{J_2}}$.

    \item $m_{I_1\onto I_2}(\sA^{\boxtimes I_1})\to \sA^{\boxtimes I_2}$ coming from the algebra structure on $\sA$, where $m_{I_1\onto I_2}$ is the multiplication map $\sC^{\otimes I_1}\to \sC^{\otimes I_2}$.
\end{itemize}

\noindent The unital structure on $\sA$ naturally defines a unital structure on $\on{Fact}^{\on{alg}}(\sA)_{\on{Ran}}$ as a factorization algebra.

\subsubsection{} It is easy to see that the functor
\[
\sA\mapsto \on{Fact}^{\on{alg}}(\sA)_{\on{Ran}}
\]

\noindent is right-lax symmetric monoidal. In particular, if $\sA,\sB\in\on{CommAlg}^{\on{un}}(\sC)$, we have a canonical map of commutative factorization algebras:
\begin{equation}\label{eq:laxfactalg}
\on{Fact}^{\on{alg}}(\sA)_{\on{Ran}}\otimes \on{Fact}^{\on{alg}}(\sB)_{\on{Ran}}\to \on{Fact}^{\on{alg}}(\sA\otimes \sB)_{\on{Ran}}
\end{equation}

\noindent taking place in $\on{Fact}(\sC)_{\on{Ran}}$.
\begin{lem}
The map (\ref{eq:laxfactalg}) is an isomorphism. In particular, the functor
\[
\sA\mapsto \on{Fact}^{\on{alg}}(\sA)_{\on{Ran}}
\]

\noindent is symmetric monoidal.
\end{lem}

\begin{proof}
We need to check that the map (\ref{eq:laxfactalg}) is an isomorphism. Since (\ref{eq:laxfactalg}) is a map of factorization algebras, it suffices to check that it is an isomorphism when restricted to $X\subset \on{Ran}$. But there it follows from the fact that the restriction of $\on{Fact}^{\on{alg}}(\sA)$ to $X$ is given by
\[
\sA\boxtimes \omega_X\in \sC\otimes D(X)\simeq D(X)\underset{D(\on{Ran}_X)}{\otimes}\on{Fact}(\sC)_{\on{Ran}},
\]

\noindent and similarly for $\on{Fact}(\sB)^{\on{alg}}$.
\end{proof}

\subsubsection{}\label{s:bialg} From the above lemma, we see that $\on{Fact}^{\on{alg}}(\sA)$ is a commutative algebra object in $\on{Fact}(\sC)_{\on{Ran}}$. Moreover, if $\sA\in\on{CommBiAlg}(\sC):=\on{CommAlg}^{\on{un}}(\on{CoAlg}(\sC))$ is a bialgebra object in $\sC$ that is commutative as a unital algebra object, then $\on{Fact}^{\on{alg}}(\sA)$ is a bialgebra object in $\on{Fact}(\sC)_{\on{Ran}}$.

\subsubsection{Coalgebra analogue} If $\sA\in\on{CoCommCoAlg}^{\on{un}}(\sC)$ is a unital cocommutataive coalgebra object in $\sC$, we may similarly define a unital factorization coalgebra $\on{Fact}^{\on{coalg}}(\sA)_{\on{Ran}}$ in $\on{Fact}(\sC)_{\on{Ran}}$ whose cofiber (or $*$-fiber) at $x$ is $\sA$ itself.\footnote{Similar to before, $*$-fiber at $x$ means the functor $\on{Fact}(\sC)_{\on{Ran}}\simeq D(\on{Ran}_X)\underset{D(\on{Ran}_X)}{\otimes}\on{Fact}(\sC)_{\on{Ran}}\xrightarrow{i_x^{*}} D(\lbrace x\rbrace)\underset{D(\on{Ran}_X)}{\otimes}\on{Fact}(\sC)_{\on{Ran}}\simeq \sC$ evaluated on $\on{Fact}^{\on{coalg}}(\sA)$. In general, this functor takes values in the pro-category $\on{Pro}(\sC)$. However, when evaluated on $\on{Fact}^{\on{coalg}}(\sA)$, it is easily seen to factor through $\sC\subset \on{Pro}(\sC)$.} Namely:
\[
\on{Fact}^{\on{coalg}}(\sA):=\underset{(I\onto J)\in\on{TwArr}^{\on{op}}}{\on{colim}} \sA^{\boxtimes I}\boxtimes \underline{k}_{X^J}.
\]

\noindent Here, $\underline{k}$ denotes the constant sheaf. The maps between the terms in the colimits are dual to those of §\ref{s:factalg}. Similar to above, $\on{Fact}^{\on{coalg}}(\sA)$ has a natural structure of a cocommutative coalgebra object in $\on{Fact}(\sC)_{\on{Ran}}$. Moreover, if $\sA$ is a bialgebra which is cocommutative as a coalgebra, then $\on{Fact}(\sA)^{\on{coalg}}_{\on{Ran}}$ is a cocommutative bialgebra object in $\on{Fact}(\sC)_{\on{Ran}}$.

\subsubsection{}\label{s:indep} Finally, we note the following. Let
\[
\on{Fact}(\sC)_{\on{Ran},\on{indep}}:=\on{Fact}(\sC)_{\on{Ran}}\underset{D(\on{Ran})}{\otimes} \on{Vect}
\]

\noindent be the \emph{independent} category of $\on{Fact}(\sC)_{\on{Ran}}$. Here, we consider $D(\on{Ran})$ with its convolution monoidal structure. The action $D(\on{Ran})\curvearrowright \on{Vect}$ is given by pushforward along $p_{\on{Ran}}:\on{Ran}\to\on{pt}$, which is symmetric monoidal.

$!$-pullback along $p_{\on{Ran}}$ defines a fully faithful embedding
\[
\on{Fact}(\sC)_{\on{Ran},\on{indep}}\into \on{Fact}(\sC)_{\on{Ran}}.
\]

\noindent It is clear that $\on{Fact}^{\on{alg}}(\sA)$ (resp. $\on{Fact}^{\on{coalg}}(\sA)$) lies in the image of this functor.

\subsection{Geometric Satake and factorization algebras associated to nilpotent radicals.}\label{S:geometric satake and factorization algebras}

\subsubsection{} Fix a parabolic subgroup $\check{P}$ of $\check{G}$ with Levi $\check{M}$. We set
\[
\on{Rep}(\check{M})_{\on{Ran}}:=\on{Fact}(\on{Rep}(\check{M}))_{\on{Ran}}.
\]

\subsubsection{}\label{s:OmegaRan} Denote by $\check{\fn}_P$ the nilradical of $\check{\fp}=\on{Lie}(\chP)$, and let
\[
C^{\bullet}(\cnp)\in\on{Rep}(\check{M})
\]

\noindent be the cohomological Chevalley complex of $\cnp$ considered as a representation of $\chM$. Note that $C^{\bullet}(\cnp)$ is a commutative unital algebra object in $\on{Rep}(\chM)$. Let
\[
\Omega(\cnp)_{\on{Ran}}=\on{Fact}^{\on{alg}}(C^{\bullet}(\cnp))_{\on{Ran}}\in \Mran
\]

\noindent be the corresponding commutative factorization algebra.

\subsubsection{} Dually, let
\[
C_{\bullet}(\cnp)_{\on{Ran}}\in\on{Rep}(\check{M})
\]

\noindent be the homological Chevalley complex of $\cnp$. Let
\[
\Upsilon(\cnp)_{\on{Ran}}:=\on{Fact}^{\on{coalg}}(C_{\bullet}(\cnp))\in\Mran
\]

\noindent be the corresponding factorization algebra. It has the structure of a cocommutative coalgebra.

\subsubsection{}\label{s:U&O_Ran} Let
\[
U(\cnp)\in\on{Rep}(\check{M})
\]

\noindent be the universal enveloping algebra of $\cnp$ and define
\[
\fU(\cnp)_{\on{Ran}}:=\on{Fact}^{\on{coalg}}(U(\cnp))\in\Mran.
\]

\noindent It has the structure of a cocommutative bialgebra, cf. §\ref{s:bialg}.

Similarly, let
\[
\cO(\check{N}_P)\in \on{Rep}(\check{M})
\]

\noindent be the commutative algebra given by functions on $\check{N}_P$. We let
\[
\cO(\check{N}_P)_{\on{Ran}}:=\on{Fact}^{\on{alg}}(\cO(\check{N}_P))\in \Mran.
\]

\noindent It has the structure of a commutative bialgebra.

\subsubsection{} Let
\[
\on{Sph}_{M,\on{Ran}}
\]

\noindent be the factorization category associated to the spherical category of $M$. That is
\[
\on{Sph}_{M,\on{Ran}}=\underset{I\in\on{fSet}}{\on{colim}}\on{Sph}_{M,I},
\]

\noindent where
\[
\on{Sph}_{M,I}=D(\mathfrak{L}^+_{X^I}M\backslash \mathfrak{L}_{X^I}M/\mathfrak{L}^+_{X^I}M).
\]

\noindent Here, $\mathfrak{L}^+_{X^I}M=\mathfrak{L}^+_{\on{Ran}}M\underset{\on{Ran}}{\times}X^I$ (resp. $\mathfrak{L}_{X^I}M$) parameterizes an $I$-tuple $x_I\in X^I$ and a map $D_{x_I}\to M$ (resp. $\overset{\circ}{D}_{x_I}\to M$).

\subsubsection{} We have a full subcategory
\[
\on{Sph}_{M,\on{Ran}}^+\subset \on{Sph}_{M,\on{Ran}}
\]

\noindent consisting of D-modules supported on $\on{Gr}_{M,\on{Ran}}^+$. Note that the (pointwise) monoidal structure on $\on{Sph}_{M,\on{Ran}}$ restricts to one on $\on{Sph}_{M,\on{Ran}}^+$.

\subsubsection{} We also have versions living over the configuration space:
\[
\on{Sph}_{M,\on{Conf}}^+\subset \on{Sph}_{M,\on{Conf.}}
\]

\noindent Here, the $!$-fiber of $\on{Sph}_{M,\on{Conf}}^+$ at some $\theta\cdot x\in\on{Conf}_{G,P}$ is
\[
D(\mathfrak{L}^+_{x}M\backslash \on{Gr}_{M,x}^{+,\theta}).
\]

Let us give a concrete description of $\on{Sph}_{M,\on{Conf}}^+$, which also makes it evident that the category is equipped with a natural external convolution monoidal structure.

\subsubsection{} Recall the group prestack $\mathfrak{L}_{\on{Conf}}M^+$ over $\on{Conf}_{G,P}$ from §\ref{s:weirdplus}. Observe that we have:
\[
\mathfrak{L}_{\on{Conf}}M^+/\mathfrak{L}^+_{\on{Conf}}M=\on{Gr}^+_{M,\mathrm{Conf}}.
\]

\noindent Moreover, we have:
\[
\on{Sph}_{M,\on{Conf}}^+=D(\mathfrak{L}^+_{\on{Conf}}M\backslash \mathfrak{L}_{\on{Conf}}M^+/\mathfrak{L}^+_{\on{Conf}}M).
\]

\noindent Note that we have a natural decomposition
\[
\on{Sph}_{M,\on{Conf}}^+=\underset{\theta\in \Lambda_{G,P}^{\on{neg}}}{\bigoplus} \on{Sph}_{M,X^{\theta},}^{+}
\]

\noindent where $\on{Sph}_{M,X^{\theta}}^{+}:=\on{Sph}_{M,\on{Conf}}^+\underset{D(\on{Conf}_{G,P})}{\otimes} D(X^{\theta})$.

\subsubsection{} There is a natural action (not relative to $\on{Conf}_{G,P}$):
\[
\mathfrak{L}_{\on{Conf}}M^+\times \mathrm{Gr}^+_{M,\mathrm{Conf}}\to \mathrm{Gr}^+_{M,\mathrm{Conf}}
\]

\noindent changing the trivialization at the punctured disk. Note that the underlying divisors of $\on{Conf}_{G,P}$ get added under this action.

\subsubsection{} The external convolution structure on $\on{Sph}_{M,\on{Conf}}^+$ is defined by $!$-pull, $*$-push along the correspondence:
\[\begin{tikzcd}
	{\mathfrak{L}^+_{\on{Conf}}M\backslash \mathfrak{L}_{\on{Conf}}M^+\overset{\mathfrak{L}^+_{\on{Conf}}M}{\times}\mathrm{Gr}^+_{M,\mathrm{Conf}}} & {\mathfrak{L}^+_{\on{Conf}}M\backslash \mathfrak{L}_{\on{Conf}}M^+/\mathfrak{L}^+_{\on{Conf}}M} \\
	\\
	{\mathfrak{L}^+_{\on{Conf}}M\backslash \mathfrak{L}_{\on{Conf}}M^+/\mathfrak{L}^+_{\on{Conf}}M\times \mathfrak{L}^+_{\on{Conf}}M\backslash \mathfrak{L}_{\on{Conf}}M^+/\mathfrak{L}^+_{\on{Conf}}M.}
	\arrow[from=1-1, to=3-1]
	\arrow[from=1-1, to=1-2]
\end{tikzcd}\]

\noindent We note that the usual argument shows that the horizontal map of the above correspondence is stratified semismall. In particular, convolution is t-exact for the natural perverse t-structure on $\on{Sph}_{M,\on{Conf}}^+$ given in §\ref{s:tstructure}.

\subsubsection{The factorizable naive geometric Satake functor} By \cite[§6]{raskin2021chiral}, we have a monoidal functor
\[
\on{Sat}^{\on{nv}}_{\on{Ran}}: \Mran\to \on{Sph}_{M,\on{Ran}}
\]

\noindent that restricts to the usual (naive) geometric Satake functor $\on{Sat}^{\on{nv}}$ on fibers and that is t-exact when restricted to $X^I$. We will denote by $\on{Sat}^{\on{nv}}_{X^I}$ the restriction of $\on{Sat}^{\on{nv}}_{\on{Ran}}$ to $D(X^I)$.

By abuse of notation, we also denote by 
\[
\Omega(\cnp)_{\on{Ran}},\;\; \Upsilon(\cnp)_{\on{Ran}},\;\; \fU(\cnp)_{\on{Ran}},\;\; \cO(\check{N}_P)_{\on{Ran}}
\]

\noindent the corresponding factorization algebras from §\ref{S:geometric satake and factorization algebras} under $\on{Sat}^{\on{nv}}_{\on{Ran}}$.
\begin{lem}
The factorization algebras
\[
\Omega(\cnp)_{\on{Ran}},\;\; \Upsilon(\cnp)_{\on{Ran}},\;\; \fU(\cnp)_{\on{Ran}},\;\; \cO(\check{N}_P)_{\on{Ran}}
\]

\noindent are supported on 
\[
\on{Sph}^+_{M,\on{Ran}}\subset\on{Sph}_{M,\on{Ran.}}
\]

\end{lem}

\begin{proof}
Recall that we may also regard $\Lambda_{G,P}$ as the character lattice of $Z(\check{M})^{\circ}$, the connected component of the identity of the center of $\check{M}$. Here, $\Lp$ becomes the positive span of the simple positive roots not contained in $\fm$. We will prove the assertion for $\Omega(\cnp)_{\on{Ran}}$, the proof for the remaining factorization algebras is similar.

Let $w_0^M$ be the longest element in the Weyl group of $M$. By \cite[Prop. 6.2.3]{braverman1999geometric}, it suffices to check that if $\nu$ is an $\check{M}$-dominant root occurring in $\check{\fn}_P$, then $w_0^M(\nu)$ is a sum of positive roots in $\check{G}$. However, this simply follows from the fact that $\check{M}$ stabilizes $\check{\mathfrak{n}}_P$; in particular, $w_0^M(\nu)$ will be a sum of positive roots in $\check{\mathfrak{n}}_P$.
\end{proof}

\subsubsection{}\label{s:fromrantoconf} Let
\[
\on{Sph}_{M,\on{Ran},\on{indep}}^+:=\on{Sph}_{M,\on{Ran}}^+\underset{D(\on{Ran})}{\otimes} \on{Vect,}
\]

\noindent as in §\ref{s:indep}. By Lemma \ref{l:fppf}, we have an equivalence:
\begin{equation}\label{eq:ran=conf}
\on{Sph}_{M,\on{Ran},\on{indep}}^+\simeq \on{Sph}_{M,\on{Conf.}}^+
\end{equation}

\noindent By unitality of the factorization algebras, we may view $\Omega(\cnp)_{\on{Ran}}, \Upsilon(\cnp)_{\on{Ran}}, \fU(\cnp)_{\on{Ran}}, \cO(\check{N}_P)_{\on{Ran}}$ as factorization algebras in $\on{Sph}_{M,\on{Conf}}^+$. To avoid confusion, we write
\[
\Omega(\cnp)_{\on{Conf}}, \Upsilon(\cnp)_{\on{Conf}}, \fU(\cnp)_{\on{Conf}}, \cO(\check{N}_P)_{\on{Conf}}
\]

\noindent for the corresponding factorization algebras in $\on{Sph}_{M,\on{Conf}}^+$.

\subsubsection{} By construction, the $!$-fiber of $\Omega(\cnp)_{\on{Conf}}$ and $\cO(\check{N}_P)_{\on{Conf}}$ at $\theta\cdot x\in\on{Conf}_{G,P}$ is
\[
C^{\bullet}(\cnp)^{\theta}, \cO(\check{N}_P)^{\theta}\in\on{Sph}_{M,x}^+;
\]

\noindent that is, the image under geometric Satake of the $\theta$-graded piece of $C^{\bullet}(\cnp)$ and $\cO(\check{N}_P)$, respectively. We have analogous assertions for $\Upsilon(\cnp)_{\on{Conf}}, \fU(\cnp)_{\on{Cof}}$ with respect to $*$-fibers.

We let
\[
\Omega(\cnp)_{X^{\theta}}, \Upsilon(\cnp)_{X^{\theta}}, \fU(\cnp)_{X^{\theta}},\cO(\check{N}_P)_{X^{\theta}}\in\on{Sph}_{M,X^{\theta}}^{+}
\]

\noindent be the respective restrictions of the factorization algebras to $X^{\theta}\subset \on{Conf}_{G,P}$.

\subsubsection{} Note that the pointwise convolution structure on $\on{Sph}_{M,\on{Ran}}$ restricts to one on $\on{Sph}_{M,\on{Ran}}^+$.
\begin{lem}\label{l:point=ext}
Under the equivalence (\ref{eq:ran=conf}), pointwise convolution structure goes to external convolution.
\end{lem}

\begin{proof}
The pointwise convolution structure on $\on{Sph}_{M,\on{Ran}}^+$ is given by pull-push along the correspondence
\[\begin{tikzcd}
	\mathfrak{L}^+_{\on{Ran}}M\backslash \mathfrak{L}_{\on{Ran}}M^+\overset{\mathfrak{L}^+_{\on{Ran}}M}{\times} \mathrm{Gr}_{M,\mathrm{Ran}}^+ & {\mathfrak{L}^+_{\on{Ran}}M\backslash \mathfrak{L}_{\on{Ran}}M^+/\mathfrak{L}^+_{\on{Ran}}M} \\
	\\
	\mathfrak{L}^+_{\on{Ran}}M\backslash \mathfrak{L}_{\on{Ran}}M^+/\mathfrak{L}^+_{\on{Ran}}M\underset{\mathrm{Ran}}{\times} \mathfrak{L}^+_{\on{Ran}}M\backslash \mathfrak{L}_{\on{Ran}}M^+/\mathfrak{L}^+_{\on{Ran}}M.
	\arrow[from=1-1, to=3-1]
	\arrow[from=1-1, to=1-2]
\end{tikzcd}\]

\noindent Note that the above is a correspondence of unital factorization spaces over $\on{Ran}$. Taking the quotient by $\on{Ran}$, we obtain the convolution diagram for $\on{Sph}_{M,\on{Ran},\on{indep}}^+$. However, by Lemma \ref{l:fppf}, this is precisely the convolution diagram for $\on{Sph}_{M,\on{Conf}}^+$.
\end{proof}

\subsubsection{} By the above lemma, $\Omega(\cnp)_{\on{Conf}}$ is equipped with a commutative algebra structure in $\on{Sph}_{M,\on{Conf}}^+$ with its external convolution structure. Similar assertions hold for $\Upsilon(\cnp)_{\on{Conf}}$ and $\fU(\cnp)_{\on{Conf}}$.

\subsubsection{} By construction, $\Omega(\cnp)_{\on{Conf}}$ is a locally compact object of $\on{Sph}_{M,\on{Conf}}^+$, and we have
\begin{equation}\label{eq:verdier}
\bD(\Omega(\cnp)_{\on{Conf}})\simeq \Upsilon(\cnp)_{\on{Conf}}.
\end{equation}

\noindent Similarly, we have:
\begin{equation}\label{eq:VerdierdualityFU}
\bD(\cO(\check{N}_P)_{\on{Conf}})\simeq \fU(\cnp)_{\on{Conf}}.
\end{equation}

\subsubsection{t-structure}\label{s:tstructure} We define a t-structre on $\on{Sph}_{M,\on{Conf}}^+$ by the requirement that the forgetful functor
\begin{equation}\label{eq:forgetGO}
\on{Sph}_{M,\on{Conf}}^+\to D(\mathrm{Gr}^+_{M,\mathrm{Conf}})
\end{equation}

\noindent is right t-exact. The argument in \cite[Lemma 2.1.15]{beraldo2021geometric} shows that (\ref{eq:forgetGO}) is t-exact.
\begin{lem}\label{l:perverse}
The factorization algebras $\Omega(\cnp)_{\on{Conf}}$ and $\Upsilon(\cnp)_{\on{Conf}}$ are perverse.
\end{lem}

\begin{proof}
It suffices to show that $\Omega(\cnp)_{\on{Conf}}$ is perverse by (\ref{eq:verdier}).

Let $F^0\cnp=\cnp, F^{-1}\cnp=[\cnp,\cnp], F^{-2}\cnp=[F^{-1}\cnp,\cnp]$ etc. Let $\cnp^i=F^{-i}\cnp/F^{-i-1}\cnp$. Denoting by $\star$ the convolution structure on $\on{Sph}_{M,\on{Conf}}^+$. Then $\Omega(\cnp)_{\on{Conf}}$ carries a filtration with associated graded
\[
\underset{i}{\star} \Omega(\cnp^i)_{\on{Conf}}.
\]

\noindent Since convolution is t-exact, we may assume that $\cnp$ is abelian. Decompose $\cnp=\underset{\alpha\in \Lambda_{G,P}^{\on{pos}}}{\check{\fn}_{P,\alpha}}$ into its weight spaces for the action of $Z(\check{M})$. Then we need to show that each
\[
\Omega(\check{\fn}_{P,\alpha})_{\on{Conf}}
\]

\noindent is perverse, where we consider $\check{\fn}_{P,\alpha}$ as an abelian Lie algebra. Recall that $\Omega(\check{\fn}_{P,\alpha})_{X^{\theta}}$ denotes the restriction of $\Omega(\check{\fn}_{P,\alpha})_{\on{Conf}}$ to $X^{\theta}\subset \on{Conf}_{G,P}$. Then $\Omega(\check{\fn}_{P,\alpha})_{X^{\theta}}$ is non-zero only if $\theta=n\cdot \alpha$, in which case
\[
\Omega(\check{\fn}_{P,\alpha})_{X^{\theta}}=(\Omega(\check{\fn}_{P,\alpha})_{X^{\alpha}})^{\star n}.
\]

\noindent Thus, it suffices to show that $\Omega(\check{\fn}_{P,\alpha})_{X^{\alpha}}$ is perverse. By construction, it is supported over the main diagonal $X\into X^{\alpha},\; x\mapsto \alpha\cdot x$. Moreover, it is easy to see that $\Omega(\check{\fn}_{P,\alpha})_{X^{\alpha}}$ is ULA over $X$. Thus, we simply need to check that the $!$-fibers of $\Omega(\check{\fn}_{P,\alpha})_{X^{\alpha}}$ at a point $\alpha\cdot x$ is concentrated in perverse degree $1$. But this fiber is the image of $C^{\bullet}(\check{\fn}_{P,\alpha})^{\alpha}=\check{\fn}_{P,\alpha}[-1]$ under geometric Satake.
\end{proof}

\subsubsection{} We record the following lemmas for later use:
\begin{lem}\label{l:ugeq1}
For $\theta\neq 0$, the sheaf $\fU(\cnp)_{X^{\theta}}$ is concentrated in perverse degrees $\geq 1$.
\end{lem}

\begin{proof}
Since $\fU(\cnp)_{\on{Conf}}$ is a factorization algebra, it suffices to show that its $*$-pullback along the diagonal
\[
X\to X^{\theta},\;\; x\mapsto \theta\cdot x
\]

\noindent is concentrated in perverse degrees $\geq 1$. Denote by $\fU(\cnp)_{\Delta^{\theta}}$ the resulting sheaf over $X$. As in the proof of Lemma \ref{l:perverse}, the sheaf $\fU(\cnp)_{\Delta^{\theta}}$ is ULA over $X$. Moreover, its $*$-fiber at some $x$ is given by the image of
\[
U(\check{\fn}_P)^{\theta}\in\on{Rep}(\check{M})^{\heartsuit}
\]

\noindent under geometric Satake. Since the latter is perverse, this forces $\fU(\cnp)_{\Delta^{\theta}}$ to live in strictly positive cohomological perverse degrees.
\end{proof}

\begin{lem}\label{l:oleq-1} For $\theta \neq 0$, the sheaf $\mathcal{O}(\cnp)_{X^{\theta}}$ is concentrated in perverse degrees $\leq -1$.
\end{lem}
\begin{proof}
This follows from a similar argument to that of Lemma \eqref{l:ugeq1}, or alternatively by applying \emph{loc.cit} and Verdier duality. 
\end{proof}

\subsection{Drinfeld's compactification.}\label{S:DRINF}\subsubsection{}

Associated to our smooth projective curve $X$, we have the moduli stack $\Bun_G$ of principal $G$-bundles on $X$. Additionally, we have the moduli stacks $\Bun_P$ and $\Bun_M$ of principal $P$ and $M$-bundles, respectively. The quotient map $P \to M$ and the inclusion $P \hookrightarrow G$ induce morphisms: 
\[
q:\Bun_P \to \Bun_M;
\]
\[
p:\Bun_P \to \Bun_G.\]

The stack $\Bun_M$ has connected components $\Bun_M^{\eta}$ indexed by $\eta \in \pi_1^{\on{alg}}(M)=\Lambda_{G,P}$. We denote by $\Bun^{\eta}_P$ the preimage of $\Bun^{\eta}_M$ in $\Bun_P$. 

\subsubsection{} The map $p:\on{Bun}_P\to \on{Bun}_G$ has a relative compactification 
\[\ot{p}:\ot{\Bun}_P \longrightarrow \Bun_G,\] 

\noindent namely Drinfeld's compactification, see \cite{braverman1999geometric}. The stack $\ot{\Bun}_P$ is also equipped with a map
\[\ot{q}: \ot{\Bun}_P \longrightarrow \Bun_M \]

\noindent and splits as a disjoint union \[\ot{\Bun}_P = \coprod_{\eta \in \Lambda_{G,P}} {\ot{\Bun}^{\eta}_P}\] of connected components.

\subsubsection{Stratification}\label{sec: stratification of Drinfeld's comp} For $\theta\in\Lambda_{G,P}^{\on{neg}}$, we let $\sH_{M,\on{Conf}}$ be the Hecke stack over $\on{Conf}_{G,P}$ parameterizing $(D,\sP_M^1,\sP_M^2,\phi)$, where:
\begin{itemize}
    \item $D\in \on{Conf}_{G,P}$.

    \item Each $\sP_M^i$ is an $M$-bundle on $X$.

    \item $\phi$ is an isomorphism:
    \[
    (\sP_M^1)_{\vert X-D}\simeq (\sP_M^2)_{\vert X-D}.
    \]
\end{itemize}

\noindent We may similarly define the Hecke stack $\sH_{\on{Conf}}^{+}\subset \sH_{\on{Conf}}$, where we require that the isomorphism $\phi$ extends to a regular embedding whenever twisting by $V^{N_P}$ for every $V\in \on{Rep}(G)$, analogous to the definition given in §\ref{s:+gr}.

We have decompositions:
\[
\sH_{M,\on{Conf}}=\underset{\theta\in\Lambda_{G,P}^{\on{neg}}}{\coprod} \sH_{M,X^{\theta}}, \;\; \sH_{M,\on{Conf}}^+=\underset{\theta\in\Lambda_{G,P}^{\on{neg}}}{\coprod} \sH_{M,X^{\theta}}^{+}.
\]

\subsubsection{}\label{s:strata} We have a natural map
\begin{equation}\label{eq:locclosed}
\on{act}: \sH_M^{+}\underset{\on{Bun}_M}{\times} \tildeP\to \ot{\on{Bun}}_P
\end{equation}

\noindent by modifying the $M$-bundle at $D$. For $\theta\in \Lambda_{G,P}^{\on{neg}}$, we let $\iota_{\theta}$ be the composition:
\[
\sH_{M,X^{\theta}}^{+}\underset{\on{Bun}_M}{\times} \on{Bun}_P\to \sH_{M,\on{Conf}}^{+}\underset{\on{Bun}_M}{\times} \tildeP\to \ot{\on{Bun}}_P.
\]
\begin{prop}[\cite{braverman2002intersection}, Prop. 1.9]\label{p:BFGM1}
Each map $\iota_{\theta}$ is a locally closed embedding and these stratify $\ot{\on{Bun}}_P$.
\end{prop}

\subsubsection{} Henceforth, we write
\[
{}_{\theta}\ot{\Bun}_P:=\sH_{M,X^{\theta}}^{+}\underset{\on{Bun}_M}{\times} \on{Bun}_{P.}
\]

\noindent In particular, ${}_{0}\ot{\Bun}_P$ identifies with $\Bun_P$ itself, and in this case we write $j=\iota_0$ for the open embedding:
\[
j: \on{Bun}_P\into\tildeP.
\]

We will also consider a finer stratification over each connected component of $\ot{\Bun}_P$. Namely, ${\ot{\Bun}^{\eta}_P}$ is stratified by the substacks
\[{}_{\theta}\ot{\Bun}^{\eta}_P \simeq \sH_{M,X^{\theta}}^+ \underset{\on{Bun}_M}{\times} \Bun_P^{\theta+\eta},\]

\noindent whose embedding into ${\ot{\Bun}^{\eta}_P}$ we also denote by $\iota_{\theta}$. 

\section{Parabolic semi-infinite IC-sheaf}\label{S:PSIC}
In this section, we introduce the factorizable parabolic semi-infinite IC-sheaf. We define it in terms of Drinfeld-Plücker and Hecke structures, and so we start this section by discussing the latter notions.

\subsection{Pointwise Hecke- and Drinfeld-Plücker structures}\label{S:pointwiseHeckeandDrPl}
In this subsection, we give an overview of (parabolic) Hecke and Drinfeld-Plücker structures. We follow closely that of \cite[§5.4]{gaitsgory2021semi} where the case of a principal parabolic is considered.

\subsubsection{} For a scheme $X$, we write $\cO(X)$ for the derived global sections of the structure sheaf of $X$.

\subsubsection{Hecke structures} Denote by $\on{Res}_{\check{M}}^{\check{G}}$ the restriction functor
\[
\on{Rep}(\check{G})\to \on{Rep}(\check{M}).
\]

\noindent Consider the regular representation
\begin{equation}\label{eq:regrep}
\cO(\check{G})\in \on{Rep}(\check{G})\otimes \on{Rep}(\check{G})
\end{equation}

\noindent considered as a $\check{G}\times \check{G}$-representation. By slight abuse of notation, we continue to $\cO(\check{G})$ for the image of (\ref{eq:regrep}) under the functor
\[
\on{Res}_{\check{M}}^{\check{G}}\otimes\on{id}: \on{Rep}(\check{G})\otimes \on{Rep}(\check{G})\to \on{Rep}(\check{M})\otimes \on{Rep}(\check{G}).
\]

\subsubsection{}\label{s:bimod} Let $\sC$ be a $(\on{Rep}(\check{M}),\on{Rep}(\check{G}))$-bimodule category. Since $\on{Rep}(\check{G})$ is symmetric monoidal, a right $\on{Rep}(\check{G})$-action gives a left action. As such, we consider $\sC$ as acted on by $\on{Rep}(\check{M})\otimes\on{Rep}(\check{G})$ on the left.

Consider the category
\[
\on{Hecke}_{\check{M},\check{G}}(\sC):=\cO(\check{G})\on{-mod}(\sC)
\]

\noindent of modules for $\cO(\check{G})$ in $\sC$. For $c\in \sC$, we refer to a lift of $c$ to an object of $\cO(\check{G})\on{-mod}(\sC)$ as a \emph{Hecke structure} on $c$.

\subsubsection{Drinfeld-Plücker structures}\label{s:abinvs}
Let $\check{N}_P$ denote the unipotent radical of $\check{P}$. Consider the functor
\[
(-)^{\check{N}_P}: \on{Rep}(\check{G})\to \on{Rep}(\check{M}),\;\; V\mapsto V^{\check{N}_P}
\]

\noindent of taking invariants against $\check{N}_P$. Here we take \emph{non-derived} invariants.\footnote{So for example the trivial representation maps to the trivial representation.}

Let
\[
\overline{\check{N}_P\backslash \check{G}}:=\on{Spec}(H^0(\cO(\check{N}_P\backslash \check{G})))
\]

\noindent be the (parabolic) basic affine space. We have:
\[
\cO(\overline{\check{N}_P\backslash \check{G}})\simeq\underset{\lambda}{\bigoplus}\; (V^{\lambda})^{\check{N}_P}\otimes (V^{\lambda})^*
\]

\noindent as $\check{M}\times\check{G}$-representations, where the sum is over all dominant coweights of $G$. We remind that $(V^{\lambda})^{\check{N}_P}$ identifies with the irreducible representation $\check{M}$ of highest weight $\lambda$.

\subsubsection{}\label{s:whatdrplactuallymeans}
Consider the category
\[
\on{DrPl}_{\check{M},\check{G}}(\sC):=\cO(\overline{\check{N}_P\backslash \check{G}})\on{-mod}(\sC).
\]

\noindent Concretely, for $c\in \sC$, the datum of a lift of $c$ to $\on{DrPl}_{\check{M},\check{G}}(\sC)$ consists of a family of maps
\[
V^{\check{N}_P}\star c\to c\star V,\;\; V\in\on{Rep}(\check{G})
\]

\noindent satisfying Drinfeld-Plücker identities as in \cite[§5.3]{gaitsgory2021semi}. We refer to such a lift as a \emph{Drinfeld-Plücker} structure on $c$.

\subsubsection{} 
The map
\[
\check{G}\to \overline{\check{N}_P\backslash \check{G}}
\]

\noindent induces a homomorphism
\[
\cO(\overline{\check{N}_P\backslash \check{G}})\to \cO(\check{G})
\]

\noindent and hence a forgetful functor
\begin{equation}\label{eq:hecketodrpl}
\on{Hecke}_{\check{M},\check{G}}(\sC)\to \on{DrPl}_{\check{M},\check{G}}(\sC).
\end{equation}

\noindent The functor (\ref{eq:hecketodrpl}) admits a left adjoint:
\[
\on{Ind}_{\on{DrPl}_{\check{M},\check{G}}}^{\Hecke_{\check{M},\check{G}}}: \on{DrPl}_{\check{M},\check{G}}(\sC)\to \Hecke_{\check{M},\check{G}}(\sC),\;\; c\mapsto \cO(\check{G})\underset{\cO(\overline{\check{N}_P\backslash \check{G}})}{\otimes} c.
\]

\subsection{Factorizable Hecke- and Drinfeld-Plücker structures}\label{S:factheckedrinf}

\subsubsection{}\label{s:algsinquestion} From the algebras
\[
\cO(\check{G}),\cO(\overline{\check{N}_P\backslash \check{G}})\in\on{Rep}(\check{M})\otimes \on{Rep}(\check{G}),
\]

\noindent we get the algebras (cf. §\ref{s:factalg}):
\[
\on{Fact}^{\on{alg}}(\cO(\overline{\check{N}_P\backslash \check{G}}))_{\on{Ran}}, \on{Fact}^{\on{alg}}(\cO(\check{G}))_{\on{Ran}}\in (\on{Rep}(\check{M})\otimes \on{Rep}(\check{G}))_{\on{Ran}}\simeq \on{Rep}(\check{M})_{\on{Ran}}\underset{D(\on{Ran})}{\otimes}\on{Rep}(\check{G})_{\on{Ran}}.
\]

\noindent We use the shorthand notation:
\[
\cO(\overline{\check{N}_P\backslash \check{G}})_{\on{Ran}}:=\on{Fact}^{\on{alg}}(\cO(\overline{\check{N}_P\backslash \check{G}}))_{\on{Ran}},\;\; \cO(\check{G})_{\on{Ran}}:=\on{Fact}^{\on{alg}}(\cO(\check{G}))_{\on{Ran}}.
\]

\noindent Moreover, recall the coalgebra
\[
\cO(\check{N}_P)_{\on{Ran}}\in\on{Rep}(\check{M})_{\on{Ran}},
\]

\noindent cf. §\ref{s:U&O_Ran}.

\subsubsection{}
Suppose that $\sC$ is a category over $\on{Ran}$ equipped with a $(\on{Rep}(\check{M})_{\on{Ran}},\on{Rep}(\check{G})_{\on{Ran}})$-bimodule structure. As in §\ref{s:bimod}, we consider it as a category acted on by $\on{Rep}(\check{M})_{\on{Ran}}\underset{D(\on{Ran})}{\otimes}\on{Rep}(\check{M})_{\on{Ran}}$ on the left.

We write
\[
\on{Hecke}_{\check{M},\check{G}}(\sC):=\cO(\check{G})_{\on{Ran}}\on{-mod}(\sC),\;\; \on{DrPl}_{\check{M},\check{G}}(\sC):=\cO(\overline{\check{N}_P\backslash \check{G}})_{\on{Ran}}\on{-mod}(\sC).
\]

\noindent As in §\ref{S:pointwiseHeckeandDrPl}, we have a map of algebras
\[
\cO(\overline{\check{N}_P\backslash \check{G}})_{\on{Ran}}\to \cO(\check{G})_{\on{Ran}},
\]

\noindent inducing a forgetful functor
\begin{equation}\label{eq:hecketodrplran}
\on{Oblv}_{\on{DrPl}_{\check{M},\check{G}}}^{\on{Hecke}_{\check{M},\check{G}}}:\on{Hecke}_{\check{M},\check{G}}(\sC)\to \on{DrPl}_{\check{M},\check{G}}(\sC).
\end{equation}

\noindent The functor (\ref{eq:hecketodrplran}) admits a left adjoint:
\begin{equation}\label{eq:inddrplhecke}
\on{Ind}_{\on{DrPl}_{\check{M},\check{G}}}^{\Hecke_{\check{M},\check{G}}}: \on{DrPl}_{\check{M},\check{G}}(\sC)\to \Hecke_{\check{M},\check{G}}(\sC),\;\; c\mapsto \cO(\check{G})_{\on{Ran}}\underset{\cO(\overline{\check{N}_P\backslash \check{G}})_{\on{Ran}}}{\otimes} c.
\end{equation}

\subsubsection{} Recall that $\cO(\check{N}_P)_{\on{Ran}}$ has a natural structure of a coalgebra object in $\on{Rep}(\check{M})_{\on{Ran}}$. The action
\[
\check{N}_P\curvearrowright \check{G}
\]

\noindent induces a coaction of $\cO(\check{N}_P)_{\on{Ran}}$ on $\cO(\check{G})_{\on{Ran}}$. It follows that we have a coaction of $\cO(\check{N}_P)_{\on{Ran}}$ on the monad 
\[
\on{Oblv}_{\on{DrPl}_{\check{M},\check{G}}}^{\on{Hecke}_{\check{M},\check{G}}}\circ\on{Ind}_{\on{DrPl}_{\check{M},\check{G}}}^{\Hecke_{\check{M},\check{G}}}: \on{DrPl}_{\check{M},\check{G}}(\sC)\to \on{DrPl}_{\check{M},\check{G}}(\sC).
\]

\begin{cor}\label{c:actiononinduction}
Let $c\in \on{DrPl}_{\check{M},\check{G}}(\sC)$. Then $\cO(\check{N}_P)_{\on{Ran}}$ coacts on $\on{Oblv}_{\on{DrPl}_{\check{M},\check{G}}}^{\on{Hecke}_{\check{M},\check{G}}}\circ\on{Ind}_{\on{DrPl}_{\check{M},\check{G}}}^{\Hecke_{\check{M},\check{G}}}(c)$.
\end{cor}

\subsubsection{Enhanced Drinfeld-Plücker structures}\label{s:enhdrpl}
Next, we introduce an intermediate structure that sits between a Hecke strcture and a Drinfeld-Plücker structure.

From the algebra
\[
\cO(\check{N}_P\backslash \check{G})\in\on{Rep}(\check{M})\otimes \on{Rep}(\check{G}),
\]

\noindent we get the algebra
\[
\cO(\check{N}_P\backslash \check{G})_{\on{Ran}}:=\on{Fact}^{\on{alg}}(\cO(\check{N}_P\backslash \check{G}))_{\on{Ran}}\in \on{Rep}(\check{M})_{\on{Ran}}\underset{D(\on{Ran})}{\otimes} \on{Rep}(\check{G})_{\on{Ran}}.
\]

\noindent Define
\[
\on{EnhDrPl}_{\check{M},\check{G}}(\sC):=\cO(\check{N}_P\backslash \check{G})_{\on{Ran}}\on{-mod}(\sC).
\]

\noindent We refer to objects in $\on{EnhDrPl}_{\check{M},\check{G}}(\sC)$ as objects in $\sC$ equipped with an \emph{enhanced Drinfeld-Plücker} structure. The map $\check{N}_P\backslash \check{G}\to \overline{\check{N}_P\backslash \check{G}}$ induces a forgetful functor
\[
\on{Oblv}^{\on{EnhDrPl}_{\check{M},\check{G}}}_{\on{DrPl}_{\check{M},\check{G}}}: \on{EnhDrPl}_{\check{M},\check{G}}(\sC)\to \on{DrPl}_{\check{M},\check{G}}(\sC),
\]

\noindent which admits a left adjoint:
\[
\on{Ind}^{\on{EnhDrPl}_{\check{M},\check{G}}}_{\on{DrPl}_{\check{M},\check{G}}}:\on{DrPl}_{\check{M},\check{G}}(\sC)\to \on{EnhDrPl}_{\check{M},\check{G}}(\sC).
\]

\noindent Similarly, the map $\check{G}\to \check{N}_P\backslash \check{G}$ induces a forgetful functor
\[
\on{Oblv}_{\on{EnhDrPl}_{\check{M},\check{G}}}^{\on{Hecke}_{\check{M},\check{G}}}:\on{Hecke}_{\check{M},\check{G}}(\sC)\to \on{EnhDrPl}_{\check{M},\check{G}}(\sC)
\]

\noindent with a left adjoint:
\[
\on{Ind}_{\on{EnhDrPl}_{\check{M},\check{G}}}^{\on{Hecke}_{\check{M,\check{G}}}}\on{EnhDrPl}_{\check{M},\check{G}}(\sC)\to\on{Hecke}_{\check{M},\check{G}}(\sC).
\]

\noindent By construction, the composition
\[
\on{Hecke}_{\check{M},\check{G}}(\sC)\xrightarrow{\on{Oblv}_{\on{EnhDrPl}_{\check{M},\check{G}}}^{\on{Hecke}_{\check{M},\check{G}}}} \on{EnhDrPl}_{\check{M},\check{G}}(\sC)\xrightarrow{\on{Oblv}^{\on{EnhDrPl}_{\check{M},\check{G}}}_{\on{DrPl}_{\check{M},\check{G}}}}\on{DrPl}_{\check{M},\check{G}}(\sC)
\]

\noindent recovers the functor $\on{Oblv}_{\on{DrPl}_{\check{M},\check{G}}}^{\on{Hecke}_{\check{M},\check{G}}}$, and similarly for the left adjoints.

Note that because $\check{N}_P\backslash \check{G}\to \overline{\check{N}_P\backslash \check{G}}$ is an open embedding and the target is affine, it follows that the forgetful functor $\on{Oblv}^{\on{EnhDrPl}_{\check{M},\check{G}}}_{\on{DrPl}_{\check{M},\check{G}}}$ is fully faithful. That is, an upgrade of a Drinfeld-Plücker structure to an enhanced Drinfeld-Plücker structure is a property.

\subsubsection{} We now describe more explicitly what it means to have a Hecke structure and an enhanced Drinfeld-Plücker structure.

Tautologically, we have identifications
\[
\on{Hecke}_{\check{M},\check{G}}(\sC)\simeq \on{Rep}(\check{M})_{\on{Ran}}\underset{\on{Rep}(\check{M})_{\on{Ran}}\underset{D(\on{Ran})}{\otimes}\on{Rep}(\check{G})_{\on{Ran}}}{\otimes} \sC
\]
\[
\simeq \on{Rep}(\check{G})_{\on{Ran}}\underset{\on{Rep}(\check{G})_{\on{Ran}}\underset{D(\on{Ran})}{\otimes}\on{Rep}(\check{G})_{\on{Ran}}}{\otimes} \sC,
\]

\noindent where in the last term we consider $\sC$ as acted on by $\on{Rep}(\check{G})_{\on{Ran}}\underset{D(\on{Ran})}{\otimes}\on{Rep}(\check{G})_{\on{Ran}}$ via the monoidal functor
\[
\on{Rep}(\check{G})_{\on{Ran}}\underset{D(\on{Ran})}{\otimes}\on{Rep}(\check{G})_{\on{Ran}}\xrightarrow{\on{Res}_{\check{M}}^{\check{G}}\otimes \on{id}} \on{Rep}(\check{M})_{\on{Ran}}\underset{D(\on{Ran})}{\otimes}\on{Rep}(\check{G})_{\on{Ran}}.
\]

\noindent This shows that for $c\in\sC$, the datum of a lift to $\on{Hecke}_{\check{M},\check{G}}(\sC)$ is equivalent to the datum of an identification between the action of $\on{Rep}(\check{G})_{\on{Ran}}$ on $c$ and the action induced by the restriction functor $\on{Res}_{\check{M}}^{\check{G}}$. That is, we require a family of isomorphisms
\[
\on{Res}_{\check{M}}^{\check{G}}(V)\star c\simeq c\star V, \;\; V\in \on{Rep}(\check{G})_{\on{Ran}}
\]

\noindent satisfying natural compatibilities (see e.g. \cite[§5]{gaitsgory2021semi}).

\subsubsection{}\label{s:inv} For enhanced Drinfeld-Plücker structures, we have an equivalence:
\begin{equation}\label{eq:EnhDrPlReformulation}
\on{EnhDrPl}_{\check{M},\check{G}}(\sC)\simeq \on{Rep}(\check{P})_{\on{Ran}}\underset{\on{Rep}(\check{M})_{\on{Ran}}\underset{D(\on{Ran})}{\otimes}\on{Rep}(\check{G})_{\on{Ran}}}{\otimes} \sC.
\end{equation}

\noindent Moreover, consider the functor
\[
C^{\bullet}(\check{\fn}_P,-): \on{Rep}(\check{P})_{\on{Ran}}\to \on{Rep}(\check{M})_{\on{Ran}}
\]

\noindent induced by restricting along $\check{P}\to\check{G}$ and applying (derived) Lie algebra cohomology against $\check{\fn}_P$.\footnote{We remark that even though taking invariants against $\check{\mathfrak{n}}_P$ is not symmetric monoidal, it is lax symmetric monoidal, and so still gives rise to a functor on the level of twisted arrows.}

This functor induces a symmetric monoidal equivalence
\[
\on{Rep}(\check{P})_{\on{Ran}}\overset{\simeq}{\to} \Omega(\check{\fn}_P)_{\on{Ran}}\on{-mod}(\on{Rep}(\check{M})_{\on{Ran}}),
\]

\noindent where $\Omega(\check{\fn}_P)_{\on{Ran}}$ is as in §\ref{s:OmegaRan}. In particular, we have an action:
\[
\on{Rep}(\check{P})_{\on{Ran}}\curvearrowright \Omega(\check{\fn}_P)_{\on{Ran}}\on{-mod}(\sC).
\]

\noindent From (\ref{eq:EnhDrPlReformulation}), we obtain an equivalence
\begin{equation}\label{eq:EnhDrPlEquivalence}
\on{EnhDrPl}_{\check{M},\check{G}}(\sC)\simeq \on{Rep}(\check{G})_{\on{Ran}}\underset{\on{Rep}(\check{G})_{\on{Ran}}\underset{D(\on{Ran})}{\otimes}\on{Rep}(\check{G})_{\on{Ran}}}{\otimes} \Omega(\check{\fn}_P)_{\on{Ran}}\on{-mod}(\sC), 
\end{equation}

\noindent where we consider $\Omega(\check{\fn}_P)_{\on{Ran}}\on{-mod}(\sC)$ as acted on by $\on{Rep}(\check{G})_{\on{Ran}}\underset{D(\on{Ran})}{\otimes}\on{Rep}(\check{G})_{\on{Ran}}$ via the monoidal functor
\[
\on{Rep}(\check{G})_{\on{Ran}}\underset{D(\on{Ran})}{\otimes}\on{Rep}(\check{G})_{\on{Ran}}\xrightarrow{\on{Res}_{\check{P}}^{\check{G}}\otimes \on{id}}\on{Rep}(\check{P})_{\on{Ran}}\underset{D(\on{Ran})}{\otimes}\on{Rep}(\check{G})_{\on{Ran}}.
\]

\noindent From the map $\check{N}_P\backslash \check{G}\to \check{N}_P\backslash \on{pt}$, we obtain a forgetful functor
\[
\on{Oblv}: \on{EnhDrPl}_{\check{M},\check{G}}(\sC)\to \Omega(\check{\fn}_P)_{\on{Ran}}\on{-mod}(\sC).
\]

\noindent We remind that $\Omega(\check{\fn}_P)_{\on{Ran}}$ is an algebra object of $\on{Rep}(\check{M})_{\on{Ran}}$, and the latter acts on $\sC$ via the symmetric monoidal functor $\on{Rep}(\check{M})_{\on{Ran}}\to \on{Rep}(\check{M})_{\on{Ran}}\otimes \on{Rep}(\check{G})_{\on{Ran}}$, inserting the unit on the second factor.

In particular, the existence of an enhanced Drinfeld-Plücker structure implies the existence of a module structure for $\Omega(\check{\fn}_P)_{\on{Ran}}$. The equivalence (\ref{eq:EnhDrPlEquivalence}) shows that for $c\in \Omega(\check{\fn}_P)_{\on{Ran}}\on{-mod}(\sC)$, the datum of a lift to $\on{EnhDrPl}_{\check{M},\check{G}}(\sC)$ is equivalent to the datum of an identification between the action of $\on{Rep}(\check{G})_{\on{Ran}}$ on $c$ and the action induced by the monoidal functor $C^{\bullet}(\check{\fn}_P,-): \on{Rep}(\check{G})_{\on{Ran}}\to \Omega(\check{\fn}_P)_{\on{Ran}}\on{-mod}(\on{Rep}(\check{M})_{\on{Ran}})$. That is, we require a family of isomorphisms
\[
C^{\bullet}(\check{\fn}_P,V)\underset{\Omega(\check{\fn}_P)_{\on{Ran}}}{\star} c\simeq c\star V, \;\; V\in \on{Rep}(\check{G})_{\on{Ran}}
\]

\noindent satisfying natural higher compatibilities.

\subsubsection{}\label{s:comp1} We highlight one further compatibility between the constructions. The functor
\[
\on{Triv}_ {\cO(\check{N}_P)}: \sC\to \cO(\check{N}_P)_{\on{Ran}}\on{-comod}(\sC)
\]

\noindent endowing an object $c\in \sC$ with the trivial comodule structure admits a right adjoint:
\[
\on{Inv}_ {\cO(\check{N}_P)}: \cO(\check{N}_P)_{\on{Ran}}\on{-comod}(\sC)\to \sC
\]

\noindent given by taking invariants against $\cO(\check{N}_P)_{\on{Ran}}$. The objects in the image of this functor naturally admit a module structure for $\Omega(\mathfrak{\check{n}}_P)_{\on{Ran}}$, and as such, we get an enhanced functor:
\[
\on{Inv}_ {\cO(\check{N}_P),\on{enh}}: \cO(\check{N}_P)_{\on{Ran}}\on{-comod}(\sC)\to \Omega(\mathfrak{\check{n}}_P)_{\on{Ran}}\on{-mod}(\sC).
\]

\begin{lem}\label{l:KD}
Let $\sC$ be a category acted on by $\on{Rep}(\check{M})_{\on{Ran}}$. Suppose $\sC$ is dualizable as a module category for $\on{Rep}(\check{M})_{\on{Ran}}$. Then the functor
\[
\on{Inv}_ {\cO(\check{N}_P),\on{enh}}: \cO(\check{N}_P)_{\on{Ran}}\on{-comod}(\sC)\to \Omega(\mathfrak{\check{n}}_P)_{\on{Ran}}\on{-mod}(\sC)
\]

\noindent is an equivalence.
\end{lem}

\begin{proof}
\step The assertion is clear when $\sC=\on{Rep}(\check{M})_{\on{Ran}}$ by the usual Koszul duality between $\cO(\check{N}_P)$ and $\Omega(\check{\fn}_P)$. Thus, it remains to prove that the canonical functors:
\[
\cO(\check{N}_P)_{\on{Ran}}\on{-comod}(\on{Rep}(\check{M})_{\on{Ran}})\underset{\on{Rep}(\check{M})_{\on{Ran}}}{\otimes}\sC\to \cO(\check{N}_P)_{\on{Ran}}\on{-comod}(\sC);
\]
\[\Omega(\check{\fn}_P)_{\on{Ran}}\on{-mod}(\on{Rep}(\check{M})_{\on{Ran}})\underset{\on{Rep}(\check{M})_{\on{Ran}}}{\otimes}\sC\to \Omega(\check{\fn}_P)_{\on{Ran}}\on{-mod}(\sC)
\]

\noindent are equivalences.

\step Consider the following setup. Let $\sA$ be a symmetric monoidal category and $\sC$ a module category for $\sA$. Let $A\in \sA$ be a cocommutative coalgebra object, and let $B\in \sA$ be a commutative algebra object. We may consider the functors:
\[
A\on{-comod}(\sA)\underset{\sA}{\otimes}\sC\to A\on{-comod}(\sC);
\]

\[
B\on{-mod}(\sA)\underset{\sA}{\otimes}\sC\to B\on{-mod}(\sC).
\]

The second functor is always an equivalence, see \cite[Corollary 8.5.7]{gaitsgory2019study}. We claim that the first functor is an equivalence under the assumption that $\sC$ is dualizable as an $\sA$-module category. Indeed, we need to check that the forgetful functor 
\[
A\on{-comod}(\sA)\underset{\sA}{\otimes}\sC\to \sC
\]

\noindent is comonadic. In turn it suffices to check that it is conservative. But this holds whenever $\sC$ is dualizable as an $\sA$-module category.  
\end{proof}

\subsubsection{} By construction, the composition
\[
\on{DrPl}_{\check{M},\check{G}}(\sC)\xrightarrow{\on{Ind}_{\on{DrPl}_{\check{M},\check{G}}}^{\Hecke_{\check{M},\check{G}}}} \on{Hecke}_{\check{M},\check{G}}(\sC)\xrightarrow{\on{Cor.} \ref{c:actiononinduction}} \cO(\check{N}_P)_{\on{Ran}}\on{-comod}(\sC)\xrightarrow{\on{Inv}_{\cO(\check{N}_P),\on{enh}}}\Omega(\mathfrak{\check{n}}_P)_{\on{Ran}}\on{-mod}(\sC)
\]

\noindent coincides with the functor
\[
\on{DrPl}_{\check{M},\check{G}}(\sC)\xrightarrow{\on{Ind}_{\on{DrPl}_{\check{M},\check{G}}}^{\on{EnhDrPl}_{\check{M},\check{G}}}} \on{EnhDrPl}_{\check{M},\check{G}}(\sC)\xrightarrow{\on{Oblv}} \Omega(\mathfrak{\check{n}}_P)_{\on{Ran}}\on{-mod}(\sC).
\]

\subsection{Semi-infinite category}

\subsubsection{} We now specialize to the case of interest, namely the semi-infinite category defined below. Before doing so, we need a suitable modification of the geometric Satake equivalence for $M$. 

For $\theta$ an element of $\in\Lambda_{G,P}$, which we remind identifies with the character lattice of the connected component of the center of $\check{M}$, denote by $e^{\theta}\in\on{Rep}(Z(\check{M})^{\on{\circ}})$ the corresponding character of $Z(\check{M})^{\on{\circ}}$. Let $2\rho_M$ denote the sum of positive coroots of $M$.

The category $\on{Rep}(\check{M})$ decomposes as a direct sum
\[
\on{Rep}(\check{M})\simeq \underset{\theta\in\Lambda_{G,P}}{\bigoplus} \on{Rep}(\check{M})^{\theta},
\]

\noindent where $\on{Rep}(\check{M})^{\theta}$ denotes the category of $\check{M}$-representations in which $Z(\check{M})^{\on{circ}}$ acts via $\theta$.

\subsubsection{}\label{s:shift2} Denote by $\sS^{\theta}$ the functor:
\begin{equation}\label{eq:shift2}
\sS^{\theta}: \on{Rep}(\check{M})^{\theta}\to \on{Rep}(\check{M})^{\theta},\;\; V\mapsto V[-\langle 2(\rho_G-\rho_M),\theta\rangle].
\end{equation}

\noindent These functors combine to a symmetric monoidal autoequivalence:
\[
\sS=\underset{\theta\in\Lambda_{G,P}}{\oplus}\sS^{\theta}:  \on{Rep}(\check{M})\to \on{Rep}(\check{M}).
\]

\noindent This functor makes sense factorizably as well. Namely, we have an evident symmetric monoidal factorizable functor
\[
\sS_{\on{Ran}}: \on{Rep}(\check{M})_{\on{Ran}}\to \on{Rep}(\check{M})_{\on{Ran}}
\]

\noindent that recovers $\sS$ on fibers. In terms of the presentation (\ref{eq:factcatdef}), $\sS_{\on{Ran}}$ is defined by taking the colimit of the functors
\[
\on{Rep}(\check{M})^{\otimes I}\otimes D(X^J)\xrightarrow{\sS^{\otimes I}\otimes \on{id}}\on{Rep}(\check{M})^{\otimes I}\otimes D(X^J)\to \on{Rep}(\check{M})_{\on{Ran}}.
\]

\subsubsection{} Let $\mathfrak{L}N_P$ denote the loop group of $N_P$ and consider its factorizable version $\mathfrak{L}_{\on{Ran}}N_P\to\on{Ran}$. We remind that it is a (factorizable) ind-group scheme. We define:
\[
\on{SI}_{P,\on{Ran}}:= D(\mathfrak{L}_{\on{Ran}}N_P\mathfrak{L}^+_{\on{Ran}}M\backslash \on{Gr}_{G,\on{Ran}}).
\]

\noindent We refer to the above category as the \emph{semi-infinite category} (associated to $P$). Note that it is a full subcategory of $D(\mathfrak{L}^+_{\on{Ran}}M\backslash \on{Gr}_{G,\on{Ran}})$.

\subsubsection{} The action
\[
\on{Sph}_{M,\on{Ran}}\curvearrowright D(\mathfrak{L}^+_{\on{Ran}}M\backslash \on{Gr}_{G,\on{Ran}})
\]

\noindent induces an action
\begin{equation}\label{eq:modifiedact}
\on{Rep}(\check{M})_{\on{Ran}}\xrightarrow{\sS_{\on{Ran}}} \on{Rep}(\check{M})_{\on{Ran}}\to \on{Sph}_{M,\on{Ran}}\curvearrowright \on{SI}_{P,\on{Ran}}.
\end{equation}

\noindent Here, the second functor is the (naive) geometric Satake functor. Henceforth, when we talk about the action of $\on{Rep}(\check{M})_{\on{Ran}}$ on $D(\mathfrak{L}^+_{\on{Ran}}M\backslash \on{Gr}_{G,\on{Ran}})$ or $\on{SI}_{P,\on{Ran}}$, we always mean (\ref{eq:modifiedact}).

\subsubsection{} We have a right action
\[
D(\mathfrak{L}^+_{\on{Ran}}M\backslash \on{Gr}_{G,\on{Ran}})\curvearrowleft \on{Sph}_{G,\on{Ran}}\leftarrow \on{Rep}(\check{G})_{\on{Ran}}.
\]

\noindent As such, the categories
\[
D(\mathfrak{L}^+_{\on{Ran}}M\backslash \on{Gr}_{G,\on{Ran}}),\;\; \on{SI}_{P,\on{Ran}}
\]

\noindent are equipped with a $(\on{Rep}(\check{M})_{\on{Ran}},\on{Rep}(\check{G})_{\on{Ran}})$-bimodule structure and hence fit the framework of §\ref{S:factheckedrinf}. 

\subsubsection{}\label{s:deltasheaf} The unit section
\[
\on{Ran}\to \on{Gr}_{G,\on{Ran}}
\]

\noindent induces a section
\[
s_{\on{Ran}}: \bB \mathfrak{L}^+_{\on{Ran}}M\to \mathfrak{L}^+_{\on{Ran}}M\backslash \on{Gr}_{G,\on{Ran.}}
\]

\noindent Let
\[
\delta_{\on{Gr}_{G,\on{Ran}}}:=(s_{\on{Ran}})_!(\omega_{\bB \mathfrak{L}^+_{\on{Ran}}M})\in D(\mathfrak{L}^+_{\on{Ran}}M\backslash \on{Gr}_{G,\on{Ran}}).
\]

\subsubsection{}\label{sec: def of semiinfinite ic} In Appendix \ref{S:DrPLstructure}, we construct a Drinfeld-Plücker structure on $\delta_{\on{Gr}_{G,\on{Ran}}}$. As such, we may define:
\[
\on{IC}_{P,\on{Ran}}^{\frac{\infty}{2}}:=\on{Ind}_{\on{DrPl}_{\check{M},\check{G}}}^{\Hecke_{\check{M},\check{G}}}(\delta_{\on{Gr}_{G,\on{Ran}}})\in \on{Hecke}_{\check{M},\check{G}}(D(\mathfrak{L}^+_{\on{Ran}}M\backslash \on{Gr}_{G,\on{Ran}})).
\]

\noindent We refer to the above sheaf as the (factorizable) semi-infinite sheaf associated to $P$.

\subsubsection{}\label{s:coaction} By Corollary \ref{c:actiononinduction}, we have a coaction:
\[
\on{IC}_{P,\on{Ran}}^{\frac{\infty}{2}}\to \cO(\check{N}_P)_{\on{Ran}}\star \on{IC}_{P,\on{Ran}}^{\frac{\infty}{2}}.
\]

\subsection{Stratification}
The purpose of this section is to calculate the restriction of the semi-infinite IC-sheaf to a natural stratification.

\subsubsection{} Let $S^0_{P,\on{Ran}}\subset \on{Gr}_{G,\on{Ran}}$ be the $\mathfrak{L}_{\on{Ran}}N_P$-orbit along the unit section $\on{Ran}\to \on{Gr}_{G,\on{Ran}}$. That is, $\Ploc$ is the factorizable ind-scheme defined as follows: for an affine test scheme $T$, a map $T\to \on{Gr}_{G,\on{Ran}}$ corresponding to a triple $(\underline{x},\sP_G,\phi)$ factors through $\Ploc$ if and only if for every $G$-representation $V$, the meromorphic map of vector bundles (regular on $T\times X\setminus \Gamma$, where $\Gamma$ is the union of graphs of the maps comprising $\underline{x}$)
\begin{equation}\label{eq:meromap1}
V_{\sP_G^0}^{N_P}\to V_{\sP_G^0}\to V_{\sP_G}
\end{equation}

\noindent extends to an injective map of vector bundles over $T\times X$ (that is, the meromorphic map extends to a regular map of coherent sheaves over $T\times X$ with cokernel flat over $T\times X$).

\subsubsection{}\label{s:tildeploc} Let $\tildePloc\subset \on{Gr}_{G,\on{Ran}}$ be the 'closure' of $\Ploc$ in $\on{Gr}_{G,\on{Ran}}$. That is, a map $T\to \on{Gr}_{G,\on{Ran}}$ factors through $\tildePloc$ if and only if the meromorphic map (\ref{eq:meromap1}) extends to an injective map of coherent sheaves (with cokernel flat over $T$).

\subsubsection{} The prestack $\tildePloc$ admits a stratification indexed by $\Lambda_{G,P}^{\on{neg}}$, by bounding the zeroes of (\ref{eq:meromap1}). More precisely, for $\theta\in \Lambda_{G,P}^{\on{neg}}$, let $S^{\theta}_{P,\on{Ran}}$ be the prestack defined as follows: for an affine test scheme $T$, a map $T\to \tildePloc$ corresponding to a triple $(\underline{x}, \sP_G,\phi)$ factors through $S^{\theta}_{P,\on{Ran}}$ if and only if there exists a colored divisor $D\in X^{\theta}(T)$ such that for every $\lambda\in \Lambda_{G,P}^{\on{neg}}$, the map (\ref{eq:meromap1}) extends to an injective map of vector bundles:
\[
V_{\sP_G^0}^{N_P}(\lambda(D))\to V_{\sP_G}.
\]

Note that each stratum $S^{\theta}_{P,\on{Ran}}$ is $\mathfrak{L}_{\on{Ran}}N_P$-stable.

\subsubsection{}We denote by $j^{\theta}$ the corresponding locally closed embedding:
\[
j^{\theta}: S^{\theta}_{P,\on{Ran}}\into \tildePloc.
\]

By taking the zeroes of (\ref{eq:meromap1}), we obtain a map
\begin{equation}\label{eq:semiinftodiv0}
S^{\theta}_{P,\on{Ran}}\to X^{\theta}.
\end{equation}

\subsubsection{} For $\theta\in \Lambda_{G,P}^{\on{neg}}$, let
\[
(X^{\theta}\times \on{Ran})^{\subset}\subset X^{\theta}\times\on{Ran}
\]

\noindent be the subprestack whose $T$-points parameterize pairs $(D,\underline{x})$ of a colored divisor on $T\times X$ of total degree $\theta$ together with a map $\underline{x}: T\to\on{Ran}$ such that $D$ is set-theoretically supported on the union of the graphs of $\underline{x}$. 

\subsubsection{} Recall the factorization space $\on{Gr}_{M,\on{Ran}}^{+}$ defined in $\S$\ref{s:Gr_M+}. Recall similarly the space $\on{Gr}^+_{M,\on{Conf}}\to \on{Conf}_{G,P}$ defined in §\ref{s: Gr_M+Conf}. By Lemma \ref{l:fppf}, we have a natural map $\on{Gr}_{M,\on{Ran}}^{+}\to \on{Gr}^+_{M,\on{Conf}}$.

We let $\on{Gr}^{+}_{M,X^{\theta}}$ and $\on{Gr}^{+}_{M,(X^{\theta}\times \on{Ran})^{\subset}}$ denote the respective pullbacks of $\on{Gr}_{M,\on{Conf}}^{+}$ and $\on{Gr}_{M,\on{Ran}}^{+}$ along $X^{\theta}\to \on{Conf}_{G,P}$. Note that we have a natural projection map:
\[
\on{Gr}^{+}_{M,(X^{\theta}\times \on{Ran})^{\subset}}\to (X^{\theta}\times \on{Ran})^{\subset}.
\]

\subsubsection{} By construction, the map (\ref{eq:semiinftodiv0}) factors through a map:
\[
'p^{\theta}: S_{P,\on{Ran}}^{\theta}\to \mathfrak{L}_{\on{Ran}}N_P\backslash S_{P,\on{Ran}}^{\theta}\to \on{Gr}^{+}_{M,(X^{\theta}\times \on{Ran})^{\subset}}.
\]

\noindent Moreover:

\begin{lem}
The functor
\[
'p^{\theta,!}: D(\on{Gr}^{+}_{M,(X^{\theta}\times \on{Ran})^{\subset}})\to D(\mathfrak{L}_{\on{Ran}}N_P\backslash S^{\theta}_{P,\on{Ran}})
\]

\noindent is an equivalence.
\end{lem}

\begin{proof}
It is easy to see that the map $\mathfrak{L}_{\on{Ran}}N_P\backslash S_{P,\on{Ran}}^{\theta}\to \on{Gr}^{+}_{M,(X^{\theta}\times \on{Ran})^{\subset}}$ realizes the source as a unipotent gerbe over the target.
\end{proof}

\subsubsection{} Note that the action of $\mathfrak{L}^+_{\on{Ran}}M$ on $\on{Gr}_{G,\on{Ran}}$ stabilizes $\tildePloc$. Define
\[
\on{SI}_{P,\on{Ran}}^{\leq 0}:= D(\mathfrak{L}_{\on{Ran}}N_P\mathfrak{L}^+_{\on{Ran}}M\backslash \tildePloc)
\]

\noindent to be the full subcategory consisting of $\on{SI}_{M,\on{Ran}}$ of D-modules supported on $\tildePloc$. The following lemma is proved in Appendix \ref{S:APPB}:
\begin{lem}\label{l:ICisSI}
The sheaf $\sIC$ defines an object in $\on{SI}_{P,\on{Ran}}^{\leq 0}$. That is, $\on{IC}_{P,\on{Ran}}^{\frac{\infty}{2}}$ is $\mathfrak{L}_{\on{Ran}}N_P$-equivariant and is supported on $\tildePloc$.
\end{lem}

\subsubsection{} Next, we describe the restriction of $\sIC$ to the stratum $S^{\theta}_{\on{Ran}}$. 

Note that the map $'p^{\theta}$ is $\mathfrak{L}^+_{\on{Ran}}M$-equivariant. We similarly denote by $'p^{\theta}$ the induced map:
\[
'p^{\theta}: \mathfrak{L}^+_{\on{Ran}}M\backslash S^{\theta}_{P,\on{Ran}}\to \mathfrak{L}^+_{\on{Ran}}M\backslash \on{Gr}^{+}_{M,(X^{\theta}\times \on{Ran})^{\subset}}.
\]

We let $p^{\theta}$ denote the composition:
\[
p^{\theta}:\mathfrak{L}^+_{\on{Ran}}M\backslash S^{\theta}_{P,\on{Ran}}\xrightarrow{'p^{\theta}}\mathfrak{L}^+_{\on{Ran}}M\backslash \on{Gr}^{+}_{M,(X^{\theta}\times \on{Ran})^{\subset}}\to \mathfrak{L}^+_{\on{Ran}}M\backslash \on{Gr}^{+}_{M,X^{\theta}}.
\]

\noindent The following proposition is proved in Section \ref{s:stratcomp} below.

\begin{prop}\label{p:strata}
We have a canonical isomorphism:
\[
j^{\theta,*}(\on{IC}_{P,\on{Ran}}^{\frac{\infty}{2}})\simeq p^{\theta,!}(\cO(\check{N}_P)_{X^{\theta}})[-\langle 2(\rho_G-\rho_M),\theta\rangle].
\]
\end{prop}

\subsection{Stratification computation}\label{s:stratcomp}

In this section, we prove Proposition \ref{p:strata}.

\subsubsection{}  Denote by $j=j^0$ the open embedding:
\[
S^{0}_{P,\on{Ran}}\into \tildePloc.
\]

\noindent For convenience, write:
\[
\mathbf{j}_!:=j_!(\omega_{S^{0}_{P,\on{Ran}}})\in D(\tildePloc).
\]

\noindent Note that $\mathbf{j}_!$ naturally defines an object of $\on{SI}^{\leq 0}_{P,\on{Ran}}$.

\subsubsection{}\label{s:DrPlstructureondelta} Consider the delta sheaf $\delta_{\on{Gr}_{G,\on{Ran}}}\in D(\mathfrak{L}^+_{\on{Ran}}M\backslash \on{Gr}_{G,\on{Ran}})$. From its Drinfeld-Plücker structure, we may define the object:
\[
\on{Ind}_{\on{DrPl}_{\check{M},\check{G}}}^{\on{EnhDrPl}_{\check{M},\check{G}}}(\delta_{\on{Gr}_G,\on{Ran}})\in \on{EnhDrPl}_{\check{M},\check{G}}(D(\mathfrak{L}^+_{\on{Ran}}M\backslash \on{Gr}_{G,\on{Ran}})).
\]

\noindent By \S \ref{s:comp1}, we have an isomorphism
\[
\on{Ind}_{\on{DrPl}_{\check{M},\check{G}}}^{\on{EnhDrPl}_{\check{M},\check{G}}}(\delta_{\on{Gr}_G,\on{Ran}})\simeq \on{Inv}_ {\cO(\check{N}_P)}(\sIC)
\]

\noindent of $\Omega(\mathfrak{\check{n}}_P)_{\on{Ran}}$-modules. By Lemma \ref{l:ICisSI}, the sheaf $\on{Inv}_ {\cO(\check{N}_P)}(\sIC)$ defines an object of the category $\on{SI}^{\leq 0}_{P,\on{Ran}}$.

\subsubsection{} Consider the diagram:
\[\begin{tikzcd}
	{S_{P,\on{Ran}}^{\theta}} && \tildePloc \\
	\\
	{\on{Gr}^{+}_{M,(X^{\theta}\times\on{Ran})^{\subset}}.}
	\arrow["{j^{\theta}}", from=1-1, to=1-3]
	\arrow["{'p^{\theta}}"', from=1-1, to=3-1]
\end{tikzcd}\]

\noindent Let $\on{pres}^{\theta}:= 'p^{\theta}_!\circ j^{\theta,*}[\langle 2(\rho_G-\rho_M),\theta\rangle]: D(\tildePloc)\to D(\on{Gr}^{+}_{M,(X^{\theta}\times\on{Ran})^{\subset}})$ denote the corresponding parabolic restriction functor. Let $\on{pres}:=\underset{\theta}{\bigoplus} \on{pres}^{\theta}$.

We denote by $\on{pres}^{\theta,x}: D(\ot{S}^0_{P,x})\to D(\on{Gr}_{M,x}^{+,\theta})$ the corresponding functor where we replace $\on{Ran}$ by a point $x\in X$. Similarly, let $\on{pres}^x:=\underset{\theta}{\bigoplus} \on{pres}^{\theta,x}$.

\subsubsection{} We will prove:

\begin{prop}\label{p:restostrata!}
We have a canonical identification:
\[
\on{Inv}_ {\cO(\check{N}_P)}(\sIC)\simeq \mathbf{j}_!.
\]

\noindent In particular, the sheaf $\mathbf{j}_!$ is equipped with a canonical module structure for $\Omega(\mathfrak{\check{n}}_P)_{\on{Ran}}$ and carries a canonical enhanced Drinfeld-Plücker structure.
\end{prop}

\begin{proof}
Since both sheaves in question are objects of $\on{SI}^{\leq 0}_{P,\on{Ran}}$, by Lemma \ref{l:ICisSI}, it suffices to show that
\begin{equation}\label{eq:prestheta}
\on{pres}^{\theta}(\on{Inv}_ {\cO(\check{N}_P)}(\sIC))=0
\end{equation}

\noindent for all $\theta\neq 0$ and that
\begin{equation}\label{eq:pres0}
\on{pres}^{0}(\on{Inv}_ {\cO(\check{N}_P)}(\sIC))\simeq \delta_{\on{Gr}_{M,\on{Ran}}}.
\end{equation}

\noindent Here $\delta_{\on{Gr}_{M,\on{Ran}}}$ is the image of the dualizing sheaf under the unit section $\on{Ran}\to \on{Gr}_{M,\on{Ran}}$. Note that $\on{pres}^{0}(\on{Inv}_ {\cO(\check{N}_P)}(\sIC))$ is a unital sheaf; that is, defines an object of $D(\on{Gr}_{M,\on{Ran}}^{+,0}/\on{Ran})\simeq \on{Vect}$. Moreover, $\on{Inv}_ {\cO(\check{N}_P)}(\sIC)$ is a factorization algebra. As such, we may check the identities (\ref{eq:prestheta}) and (\ref{eq:pres0}) after taking the $!$-fiber at every $x\in X\subset \on{Ran}$.\footnote{We remark that the functor $\on{pres}$ commutes with taking $!$-fibers by hyperbolic localization: the functor $\on{pres}$ coincides with the analogous parabolic restriction functor for the opposite parabolic $P^-$, replacing $*$-pull, $!$-push with $!$-pull, $*$-push, and the latter functor clearly commutes with $!$-fibers by base change.}

We write $\cO(\check{N}_P\backslash \check{G})_x$ (resp. $\cO(\overline{\check{N}_P\backslash \check{G}})_x$) for the $!$-restriction of $\cO(\check{N}_P\backslash \check{G})_{\on{Ran}}$ (resp. $\cO(\overline{\check{N}_P\backslash \check{G}})_{\on{Ran}}$) to a point $x$, i.e., as objects of $\on{Rep}(\check{M})\otimes \on{Rep}(\check{G})$.

By \S \ref{s:comp1}, we have the identity:
\[
\on{Inv}_ {\cO(\check{N}_P)}(\sIC)\simeq \on{Ind}_{\on{DrPl}_{\check{M},\check{G}}}^{\on{EnhDrPl}_{\check{M},\check{G}}}(\delta_{\on{Gr}_{G,\on{Ran}}})=\cO(\check{N}_P\backslash \check{G})_{\on{Ran}}\underset{\cO(\overline{\check{N}_P\backslash \check{G}})_{\on{Ran}}}{\otimes} \delta_{\on{Gr}_{G,\on{Ran}}}. 
\]

\noindent As such, it suffices to establish the identities:
\begin{equation}\label{eq:presthetax}
\on{pres}^{\theta,x}(\cO(\check{N}_P\backslash \check{G})_x\underset{\cO(\overline{\check{N}_P\backslash \check{G}})_x}{\otimes} \delta_{\on{Gr}_{G,x}})=0;
\end{equation}

\begin{equation}\label{eq:pres0x}
\on{pres}^{0,x}(\cO(\check{N}_P\backslash G)_x\underset{\cO(\overline{\check{N}_P\backslash G})_x}{\otimes} \delta_{\on{Gr}_{G,x}})=\delta_{\on{Gr}_{M,x}}.
\end{equation}

\step Write the tensor product in question as a Bar complex:
\[
\cO(\check{N}_P\backslash \check{G})_x\underset{\cO(\overline{\check{N}_P\backslash \check{G}})_x}{\otimes} \delta_{\on{Gr}_{G,x}}=\underset{n}{\on{colim}} \;\cO(\check{N}_P\backslash \check{G})_x\otimes \cO(\overline{\check{N}_P\backslash \check{G}})_x^{\otimes n}\otimes \delta_{\on{Gr}_{G,x}}.
\]

Note that $\on{pres}^x(\cO(\check{N}_P\backslash \check{G})_x\otimes \cO(\overline{\check{N}_P\backslash \check{G}})_x^{\otimes n}\otimes \delta_{\on{Gr}_{G,x}})$ is by definition the image of $\cO(\check{N}_P\backslash \check{G})_x\otimes \cO(\overline{\check{N}_P\backslash \check{G}})_x^{\otimes n}$ under the functor:
\begin{equation}\label{eq:comp2}
\on{Rep}(\check{M})\otimes\on{Rep}(\check{G})\to \on{Sph}_{M,x}\otimes \on{Sph}_{G,x}\xrightarrow{\on{id}\otimes \on{pres}^x}\on{Sph}_{M,x}\otimes \on{Sph}_{M,x}\xrightarrow{-\star -} \on{Sph}_{M,x}.
\end{equation}

\noindent Here, the first functor is the geometric Satake functor, the third is convolution, and we have abused notation by also denoting $\on{pres}^x: \on{Sph}_{G,x}\to \on{Sph}_{M,x}$ the (shifted) direct sum over all $\theta$ of the functor of $*$-pull and $!$-push along:\\
\begin{tikzcd}
	{\mathfrak{L}^+_{x}M\backslash S^{\theta}_x} && {\mathfrak{L}^+_{x}G\backslash \mathfrak{L}_{x}G/\mathfrak{L}^+_{x}G} \\
	\\
	{\mathfrak{L}^+_{x}M\backslash \mathfrak{L}_{x}M/\mathfrak{L}^+_{x}M.}
	\arrow[from=1-1, to=1-3]
	\arrow[from=1-1, to=3-1]
\end{tikzcd}

Recall that by construction of the geometric Satake functor $\on{Sat}^{\on{nv}}: \on{Rep}(\check{G})\to \on{Sph}_{G,x}$, the composition
\[
\on{Rep}(\check{G})\xrightarrow{\on{Sat}^{\on{nv}}} \on{Sph}_{G,x}\xrightarrow{\on{pres}^x} \on{Sph}_{M,x}
\]

\noindent coincides with the functor
\begin{equation}\label{eq:res+sat}
\on{Rep}(\check{G})\xrightarrow{\on{Res}_{\check{M}}^{\check{G}}}\on{Rep}(\check{M})\xrightarrow{\underset{\theta}{\bigoplus}\;\on{Sat}_M^{\on{nv},\theta}[-\langle 2(\rho_G-\rho_M),\theta\rangle]} \on{Sph}_{M,x},
\end{equation}

\noindent see e.g. \cite[§3]{mirkovic2007geometric} or \cite[Thm. 2.2 (3)]{braverman2001crystals}. As such, we get:
\[
\on{pres}^x(\cO(\check{N}_P\backslash \check{G})_x\otimes \cO(\overline{\check{N}_P\backslash \check{G}})_x^{\otimes n}\otimes \delta_{\on{Gr}_{G,x}})\simeq \cO(\check{N}_P\backslash \check{G})_x\otimes \cO(\overline{\check{N}_P\backslash \check{G}})_x^{\otimes n}\otimes \delta_{\on{Gr}_{M,x}},
\]

\noindent where we now consider $\on{Sph}_{M,x}$ as module category for $\on{Rep}(\check{M})\otimes \on{Rep}(\check{G})$ via the restriction functor $\on{Rep}(\check{M})\otimes \on{Rep}(\check{G})\xrightarrow{\on{id}\otimes \on{Res}_{\check{M}}^{\check{G}}} \on{Rep}(\check{M})\otimes \on{Rep}(\check{M})$ and the usual Satake action of the latter on $\on{Sph}_{M,x}$.\footnote{We remark that the shift by $-\langle 2(\rho_G-\rho_M,\theta)\rangle$ in the functor (\ref{eq:res+sat}) exactly cancels the shift appearing in \S\ref{s:shift2} for the action of $\on{Rep}(\check{M})$ on $D(\mathfrak{L}_x^+M\backslash \on{Gr}_{G,x})$.} We conclude that:
\begin{equation}\label{eq:colimit1}
\on{pres}^x(\cO(\check{N}_P\backslash \check{G})_x\underset{\cO(\overline{\check{N}_P\backslash \check{G}})_x}{\otimes} \delta_{\on{Gr}_{G,x}})\simeq \underset{n}{\on{colim}} \;\cO(\check{N}_P\backslash \check{G})_x\otimes \cO(\overline{\check{N}_P\backslash \check{G}})_x^{\otimes n}\otimes \delta_{\on{Gr}_{M,x}}.
\end{equation}

\step Consider $\delta_{\on{Gr}_{M,x}}$ as a module for $\cO(\overline{\check{N}_P\backslash \check{G}})_x\in \on{Rep}(\check{M})\otimes \on{Rep}(\check{G})$ via the augmentation. That is, the action of $\cO(\overline{\check{N}_P\backslash \check{G}})_x$ on $\delta_{\on{Gr}_{M,x}}$ is given by $\cO(\overline{\check{N}_P\backslash \check{G}})_x\to k$ induced by the point $1\in \overline{\check{N}_P\backslash \check{G}}$, and where we consider $\cO(\overline{\check{N}_P\backslash \check{G}})_x$ as an $\check{M}$-representation via the diagonal map $\check{M}\to \check{M}\times \check{G}$.

This endows $\delta_{\on{Gr}_{M,x}}$ with a Drinfeld-Plücker structure at $x$. This tautologically comes from an enhanced Drinfeld-Plücker structure (which we remind is a property, not a structure, cf. \S \ref{s:enhdrpl}). That is:
\begin{equation}\label{eq:colimit2}
\delta_{\on{Gr}_{M,x}}\simeq \cO(\check{N}_P\backslash \check{G})_x\underset{\cO(\overline{\check{N}_P\backslash \check{G}})_x}{\otimes}\delta_{\on{Gr}_{M,x}}=\underset{n}{\on{colim}} \; \cO(\check{N}_P\backslash \check{G})_x\otimes \cO(\overline{\check{N}_P\backslash \check{G}})_x^{\otimes n}\otimes \delta_{\on{Gr}_{M,x}}.
\end{equation}

\noindent If we can show that the transition maps in the colimits (\ref{eq:colimit1}) and (\ref{eq:colimit2}) agree in a coherent manner, then we are done.

In the transition maps in (\ref{eq:colimit1}), we consider maps of two types:
\begin{itemize}
    \item The maps $\cO(\overline{\check{N}_P\backslash \check{G}})_x\otimes \delta_{\on{Gr}_{M,x}}\to \delta_{\on{Gr}_{M,x}}$ induced by applying $\on{pres}^x$ to the map $\cO(\overline{\check{N}_P\backslash \check{G}})_x\otimes \delta_{\on{Gr}_{G,x}}\to \delta_{\on{Gr}_G,x}$ in $D(\mathfrak{L}^+_{x}M\backslash \on{Gr}_{G,x})$ induced by the Drinfeld-Plücker structure on $\delta_{\on{Gr}_{G,x}}$. This evidently coincides with the corresponding map in the colimit (\ref{eq:colimit2}). Moreover, since both $\cO(\overline{\check{N}_P\backslash \check{G}})_x\otimes \delta_{\on{Gr}_{M,x}}$ and $\delta_{\on{Gr}_{M,x}}$ are in the heart of the perverse t-structure on $\on{Sph}_{M,x}$, this automatically provides higher coherence.

    \item The maps $\cO(\check{N}_P\backslash \check{G})_x\otimes \cO(\overline{\check{N}_P\backslash \check{G}})_x\to \cO(\check{N}_P\backslash \check{G})_x$. By construction, these are obtained by applying the functor (\ref{eq:comp2}) to the natural action $\cO(\overline{\check{N}_P\backslash \check{G}})_x$ on $\cO(\check{N}_P\backslash \check{G})_x$. These coincide with the similar maps appearing in the colimit (\ref{eq:colimit2}).
\end{itemize}

\end{proof}

\subsubsection{}

We may now prove Proposition \ref{p:strata}:

\begin{proof}[Proof of Proposition \ref{p:strata}]

By $\mathfrak{L}_{\on{Ran}}N_P$-equivariance of $\sIC$, we need to show that $\on{pres}(\sIC)\simeq \cO(\check{N}_P)_{\on{Ran}}$. Indeed, by unitality, the $!$-restriction of $\cO(\check{N}_P)_{\on{Ran}}$ along $\on{Gr}_{M,(X^{\theta}\times\on{Ran})^{\subset}}^+\to \on{Gr}_{M,\on{Ran}}^+$ descends to $\on{Gr}_{M,X^{\theta}}^+$.

Since $\on{pres}$ is $\on{Sph}_{M,\on{Ran}}$-linear, we obtain a coaction of $\cO(\check{N}_P)_{\on{Ran}}$ on $\on{pres}(\sIC)$. Moreover, by (the proof of) Proposition \ref{p:restostrata!}, we have an isomorphism 
\[
\on{Inv}_ {\cO(\check{N}_P)}(\on{pres}(\sIC))\simeq \delta_{\on{Gr}_{M,\on{Ran}}}
\]

\noindent as $\Omega(\mathfrak{\check{n}}_P)_{\on{Ran}}$-modules, where the action of $\Omega(\mathfrak{\check{n}}_P)_{\on{Ran}}$ on $\delta_{\on{Gr}_{M,\on{Ran}}}$ is trivial. The proof now follows by Koszul duality: denote by $\delta_{\on{Gr}_{M,\on{Ran}}}\underset{\Omega(\mathfrak{\check{n}}_P)_{\on{Ran}}}{\star}-$ the functor of taking coinvariants for a $\Omega(\mathfrak{\check{n}}_P)_{\on{Ran}}$-module. That is, the inverse functor to the equivalence of Lemma \ref{l:KD}.\footnote{We apply Lemma \ref{l:KD} to the category $\sC=\on{Sph}_{M,\on{Ran}}$, which is easily seen to be dualizable as a module category for $\on{Rep}(\check{M})_{\on{Ran}}$.} Then:
\[
\on{pres}(\sIC)\simeq \delta_{\on{Gr}_{M,\on{Ran}}}\underset{\Omega(\mathfrak{\check{n}}_P)_{\on{Ran}}}{\star}\on{Inv}_ {\cO(\check{N}_P)}(\on{pres}(\sIC))\simeq  \delta_{\on{Gr}_{M,\on{Ran}}}\underset{\Omega(\mathfrak{\check{n}}_P)_{\on{Ran}}}{\star}\delta_{\on{Gr}_{M,\on{Ran}}}\simeq \cO(\check{N}_P)_{\on{Ran}}.
\]

\end{proof}

\subsubsection{}\label{s:!rescomp} Finally, let us record a lemma that provides a lower bound on the perverse cohomological degrees of the $!$-restriction of $\sIC$ to strata. This will be used in the local-to-global comparison in Section \ref{S:loctoglob}.

By $\mathfrak{L}_{\on{Ran}}N_P$-equivariance and unitality of $\sIC$, the sheaf $j^{\theta,!}(\sIC)$ descends along $p^{\theta}$. That is, we have 
\[
j^{\theta,!}(\sIC)\simeq p^{\theta,!}(\cF^{\theta})
\]

\noindent for some $\cF^{\theta}\in D(\on{Gr}_{M,X^{\theta}}^{+})$.

\subsubsection{} The following lemma is proved in Appendix \ref{s:APPB2}:
\begin{lem}\label{l:rescomp}
For $\theta\neq 0$, the sheaf $\cF^{\theta}$ lies in perverse cohomological degrees $\geq 1 + \langle 2(\rho_G-\rho_M),\theta\rangle$.
\end{lem}

\begin{rem}
It follows from Corollary \ref{c:globrestostrata} below that in fact:
\[
\cF^{\theta}\simeq \fU(\cnp)_{X^{\theta}}[-\langle 2(\rho_G-\rho_M),\theta\rangle].
\]

\end{rem}

\newpage

\section{Local-to-global comparisons}\label{S:LOCTOGLOB}

In this section, we will describe a local version of the geometric Eisenstein series functors and compare them with their global counterparts. 

\subsection{The relative semi-infinite space.}\label{s:relsemiinf}

Define a relative version of $\ot{S}^0_{P, \Ran}$ over $\on{Bun}_M$ as the fiber product:
\[(\Bun_M \times \Ran)\underset{\mathbb{B}\mathfrak{L}^+_{\Ran}M}{\times}\mathfrak{L}^+_{\Ran}M\backslash\ot{S}^0_{P,\Ran}=: \ot{\Gr}_{P,\Bun_M}.\]

\noindent Here $\mathbb{B} \mathfrak{L}^+_{\Ran}M $ denotes the factorizable prestack parameterizing a point of $\Ran$ together with an $M$-bundle on its formal neighborhood along $X$. Note every such bundle is trivial fppf locally on the base. Unwinding the definition of $\ot{\Gr}_{P,\Bun_M}$, we see by Beauville-Laszlo gluing that there is a canonical map: 
\[ \pi_P: \ot{\Gr}_{P,\Bun_M} \longrightarrow \ot{\Bun}_P.\]

\subsubsection{} Note that $\ot{\Gr}_{P,\Bun_M}$ is stratified by the spaces
\[(\Bun_M \times \Ran)\underset{\mathbb{B}\mathfrak{L}^+_{\Ran}M}{\times}\mathfrak{L}^+_{\Ran}M\backslash S_{P,\Ran}^{\theta}\]

\noindent in a way compatible with the stratification on $\ot{\Bun}_P$, cf. Section \ref{S:DRINF}.

Let $\on{pr_0}$ and $\on{pr_M}$ denote the projections from $\ot{\Gr}_{P,\Bun_M}$ to $\mathfrak{L}^+_{\Ran}M\backslash\ot{S}^0_{P,\Ran}$ and $\Bun_M$, respectively.

\subsection{Homological contractibility results} In this section, we will show that the map $\pi_P$ is universally homologically contractible.

\subsubsection{} Let $\mathcal{Y}$ be a prestack equipped with a map $\mathcal{Y} \to \bB\mathfrak{L}^+_{\Ran}M$. Define 
\[_{\mathcal{Y}}{\Gr_{G,\Ran}} \coloneqq \cY\underset{\bB\mathfrak{L}^+_{\Ran}M}{\times} \mathfrak{L}^+_{\Ran}M\backslash \Gr_{G,\Ran}\]

\noindent to be the corresponding twisted version of the Beilinson-Drinfeld affine Grassmannian. To simplify notation, we write: 
\[_{\Bun_M}{\Gr_{G,\Ran}} \coloneqq _{\Bun_M \times \Ran}{\Gr_{G,\Ran}}.\]

\subsubsection{} Note that $_{\Bun_M}{\Gr_{G,\Ran}}$ parameterizes a point $x_I$ of $\Ran$, a $G$-bundle $\sP_G$ on $X$, an $M$-bundle $\sP_M$ on $X$ and an identification of $\sP_G$ with $\sP_M\overset{M}{\times} G$ away from $x_I$.

Note that there is an action
\[\on{Sph}_{M,\Ran} \underset{D(\on{Ran})}{\otimes}\on{Sph}_{G,\Ran}\curvearrowright D({_{\Bun_M}{\Gr_{G,\Ran}}}),\]

\noindent and therefore an action
\[\on{Rep}(\check{M})_{\on{Ran}} \underset{D(\on{Ran})}{\otimes} \on{Rep}(\check{G})_{\on{Ran}}\curvearrowright D({_{\Bun_M}{\Gr_{G,\Ran}}})\]

\noindent by restriction along the naive geometric Satake functor associated to the reductive group $M\times G$. 

\subsubsection{}\label{s:tildepol} Define $\ot{\Bun}_{P,\on{pol}}\to \on{Ran}$ to be the moduli stack parameterizing a point $x_I$ of $\Ran$, a $G$-bundle $\sP_G$, an $M$-bundle $\sP_M$, and for every $G$-representation $V$, injective meromorphic maps  
\begin{equation}\label{eq: meromorphic maps} V^{N_P}_{\sP_M} \to V_{\sP_G}\end{equation}

\noindent of coherent sheaves satisfying the Plücker relations, and which are non-degenerate away from $x_I$. 

Inside $\ot{\Bun}_{P,\on{pol}}$ there is the subfunctor $\ot{\Bun}_{P,\on{zer}}$ consisting of points as above with the property that each map \eqref{eq: meromorphic maps} is regular. In this case, the zeroes of the maps (\ref{eq: meromorphic maps}) are automatically supported on the $x_I$. We have a natural map:
\[
\ot{\Bun}_{P,\on{zer}}\to \ot{\on{Bun}}_P\times\on{Ran}.
\]

\subsubsection{} Note that there is a canonical map
\[\pi_{P,\on{pol}}: {_{\Bun_M}{\Gr_{G,\Ran}}} \longrightarrow \ot{\Bun}_{P,\on{pol}}\]

\noindent taking a generic reduction of a $G$-bundle to an $M$-bundle to the resulting (meromorphic) Plücker data. The fiber of $\pi_{P,\on{pol}}$ over $\ot{\Bun}_{P,\on{zer}}$ is precisely $\ot{\Gr}_{P,\Bun_M}$. We let $\pi_{P,\on{zer}}$ denote the resulting map:
\[
\pi_{P,\on{zer}}: \ot{\Gr}_{P,\Bun_M}\to \pi_{P,\on{zer}}.
\]

\subsubsection{} For $\theta\in \Lambda_{G,P}^{\on{neg}}$, define:
\[
{}_{\theta}\ot{\Bun}_{P,\on{zer}}:=\ot{\Bun}_{P,\on{zer}} \underset{{\ot{\Bun}_P}}{\times} {}_{\theta}\ot{\Bun}_P.
\]

\subsubsection{} The main property of $\ot{\Bun}_{P,\on{pol}}$ is that we have a Hecke action:
\[\on{Rep}(\check{M})_{\on{Ran}} \underset{D(\on{Ran})}{\otimes}\on{Rep}(\check{G})_{\on{Ran}}\curvearrowright D(\ot{\Bun}_{P,\on{pol}}). \]

\noindent The morphism $\pi_{P,\on{pol}}$ clearly commutes with Hecke correspondences for $M$ and $G$, and hence the functor $\pi^!_{P,\on{pol}}$ is equivariant for the action of $\on{Rep}(\check{M})_{\on{Ran}} \underset{D(\on{Ran})}{\otimes}\on{Rep}(\check{G})_{\on{Ran}}$. 

\subsubsection{} For an algebraic group $H$, define $\Bun_H^{\on{gen}}$ to be the stack parameterizing a generically defined $H$ bundle on $X$, see \cite[§2]{barlev2012d}. For any homomorphism $H \to K$ of algebraic groups, there is an evident map 
\[ \on{Ind}^{\on{gen}}_{H \to K}: \Bun_H^{\on{gen}} \longrightarrow \Bun_K^{\on{gen}}\]

\noindent given by induction of torsors. In particular, we obtain a map $\on{Ind}^{\on{gen}}_{M \to P}:\Bun_M^{\on{gen}} \to \Bun_P^{\on{gen}}$. 

\subsubsection{} Before continuing, let us briefly recall the notion of universal homological contractibility from \cite{gaitsgory2011contractibility}. A morphism $\cY \to \cX$ is said to be \emph{universally homologically contractible} if for any affine scheme $S$ mapping to $\cX$, the $!$-pullback functor
\[D(S) \longrightarrow D(S \underset{\cX}{\times} \cY)\]

\noindent along the projection $S \underset{\cX}{\times} \cY \to S$ is fully faithful. Using that any prestack may be written as a colimit of affine schemes, we obtain that for \emph{any} prestack $\cX_0$ mapping to $\cX$, we have that $!$-pullback 
\[D(\cX_0) \longrightarrow D(\cX_0 \underset{\cX}{\times} \cY)\]

\noindent along the projection is fully faithful. 

\begin{lem}\label{lem: affine contractibility}
    The induction map $\on{Ind}^{\on{gen}}_{M \to P}:\Bun_M^{\on{gen}} \to \Bun_P^{\on{gen}}$ is universally homologically contractible.
\end{lem}
\begin{proof}
    Fix an $S$-point $\sP_P$ of $\Bun_P^{\on{gen}}$. We wish to show that pullback along the projection 
    \[S \underset{\Bun_P^{\on{gen}}}{\times}\Bun_M^{\on{gen}} \to S\]

    \noindent is fully faithful. To this end, if 
    \[p: \mathscr{Q} \to U \subseteq X \times S\]
    
    \noindent is a morphism from a prestack $\mathscr{Q}$ to a domain $U$ in $X \times S$, we can define the prestack 
    \[\underline{\on{GSect}}_{S}(X \times S, \mathscr{Q})\]
    
    \noindent of generically defined sections of $p$ (see \cite{barlev2012d} for a definition).\footnote{Although in \cite{barlev2012d}, the author only defines rational sections for $\mathscr{Q}$ living over the whole curve, the space still makes sense in this broader generality since, for example, the intersection of any two domains is a domain.}
    
    For a generically defined $S$-family of $P$-torsors $\sP_P$, we can consider the quotient 
    \[\sP_P \overset{P}{\times} P/M \coloneqq (\sP_P \times P/M)/P\]

    \noindent where $P$ acts diagonally. Note $\sP_P \overset{P}{\times} P/M$ is equipped with the structure of an fppf locally trivial fibration over $U$ with a universally homologically contractible fiber\footnote{Of course, in general there is no action of $N_P$ on $\sP_P \overset{P}{\times} P/M$ unless $\sP_P$ is trivial, and hence $\sP_P \overset{P}{\times} P/M$ is \emph{not} a torsor for $N_P$.} given by $P/M \simeq N_P$.

    Now it is easy to see that there is an isomorphism 
    \[\underline{\on{GSect}}_{S}(X \times S, \sP_P \overset{P}{\times} P/M) \overset{\sim}{\longrightarrow} S \underset{\Bun_P^{\on{gen}}}{\times}\Bun_M^{\on{gen}}\]

    \noindent commuting with the projections to $S$. The result now follows from \cite[Lemma 3.1.2]{gaitsgory2011contractibility} using the fact that $P/M 
    \simeq N_P$ is an affine space (see also \cite[Remark 6.2.12]{barlev2012d}). 
\end{proof}

\begin{prop}\label{prop: relative Grassmannian contractibility}
    The map $\pi_{P,\on{pol}}: {_{\Bun_M}}{\Gr_{G,\on{Ran}}} \to \ot{\Bun}_{P,\on{pol}}$ is universally homologically contractible.
\end{prop}\label{prop: full contractibility}
\begin{proof}
    Forgetting the point of $\on{Ran}$, we obtain morphisms
    \[{_{\Bun_M}}{\Gr_{G,\on{Ran}}} \longrightarrow \Bun_M^{\on{gen}} \underset{\Bun_G^{\on{gen}}}{\times} \Bun_G,\,\,\,  \,\,\, \ot{\Bun}_{P,\on{pol}} \longrightarrow  \Bun_P^{\on{gen}} \underset{\Bun_G^{\on{gen}}}{\times} \Bun_G \]

    \noindent such that the diagram 
    \[\begin{tikzcd}
	{{_{\Bun_M}}{\Gr_{G,\on{Ran}}}} & {\Bun_M^{\on{gen}} \underset{\Bun_G^{\on{gen}}}{\times} \Bun_G} \\
	{\ot{\Bun}_{P,\on{pol}}} & {\Bun_P^{\on{gen}} \underset{\Bun_G^{\on{gen}}}{\times} \Bun_G}
	\arrow[from=1-1, to=1-2]
	\arrow["{\pi_{P,\on{pol}}}"', from=1-1, to=2-1]
	\arrow[from=1-2, to=2-2]
	\arrow[from=2-1, to=2-2]
\end{tikzcd}\]

\noindent is Cartesian. Here the right vertical arrow is obtained via base change from $\on{Ind}^{\on{gen}}_{M \to P}$. Now the result follows from Lemma \ref{lem: affine contractibility}. 
\end{proof}

\begin{cor}\label{cor: positive contractibility}
    The map $\pi_P: \ot{\Gr}_{P,\Bun_M} \to \ot{\Bun}_P$ is universally homologically contractible. 
\end{cor}
\begin{proof}

We have a Cartesian diagram:

\[\begin{tikzcd}
	{\ot{\Gr}_{P,\Bun_M}} && {{_{\Bun_M}}{\Gr_{G,\on{Ran}}}} \\
	\\
	{\ot{\Bun}_{P,\on{zer}}} && {\ot{\Bun}_{P,\on{pol}}.}
	\arrow[from=1-1, to=1-3]
	\arrow["{\pi_{P,\mathrm{zer}}}"', from=1-1, to=3-1]
	\arrow["{\pi_{P,\mathrm{pol}}}", from=1-3, to=3-3]
	\arrow[from=3-1, to=3-3]
\end{tikzcd}\]

It is easy to see that the forgetful map
    \[\ot{\Bun}_{P,\on{zer}} \to \ot{\Bun}_P \]

    \noindent that forgets the point of $\Ran$ is \emph{pseudo-proper}, i.e. its fiber over an affine scheme $S$ can be written as a colimit of schemes proper over $S$ with closed embeddings as transition maps. 

Hence by \cite[Lemma A.2.5]{gaitsgory2021semi}, it suffices to show that its fibers over field-valued points are homologically contractible. This follows, for example, from \cite[Prop. A.2.7]{gaitsgory2021semi}.
\end{proof}

\subsection{Local-to-global compatibility of IC sheaves}\label{S:loctoglob} In this section we verify that $\pi_P^!(\on{IC}_{\ot{\Bun}_P})$ is equivalent to the relative semi-infinite IC sheaf, up to suitable shifts.
 
\subsubsection{} There is a natural map
\[j_{P,\on{pol}}: \Bun_P \times \Ran \longrightarrow \ot{\Bun}_{P,\on{pol}}\]

\noindent over $\on{Ran}$ taking a $P$-bundle with a point $x_I$ of $\Ran$ to the associated non-degenerate Plücker data for the induced $G$-bundle (see e.g. \cite[§4.1]{braverman1999geometric}).

\subsubsection{} For any $\theta\in \Lambda_{G,P}^{\on{neg}}$, define
\[
\on{Gr}_{P,\Bun_M}^{\theta}:=\Bun_M \times \Ran\underset{\mathbb{B}\mathfrak{L}^+_{\Ran}M}{\times}\mathfrak{L}^+_{\Ran}M\backslash S^{\theta}_{P,\Ran}.
\]

\noindent When $\theta=0$, we write:
\[
\on{Gr}_{P,\Bun_M}^{0}=:\on{Gr}_{P,\Bun_M}.
\]

We denote by $j_P^{\theta}$ the corresponding embedding:
\[
j_P^{\theta}: \on{Gr}_{P,\Bun_M}^{\theta}\into {_{\Bun_M}}{\Gr_{G,\Ran}}.
\]

\noindent When $\theta=0$, we write: $j_P^{0}=:j_P$.

\subsubsection{} We have a Cartesian square
\begin{equation}\label{origincartsquare}\begin{tikzcd}
	\on{Gr}_{P,\Bun_M} & {{_{\Bun_M}}{\Gr_{G,\Ran}}} \\
	{\Bun_P \times \Ran} & {\ot{\Bun}_{P,\on{pol}}.}
	\arrow["{j_P}", from=1-1, to=1-2]
	\arrow["{\pi^0_{P}}"', from=1-1, to=2-1]
	\arrow["{\pi_{P,\on{pol}}}", from=1-2, to=2-2]
	\arrow["{j_{P,\on{pol}}}"', from=2-1, to=2-2]
\end{tikzcd}
\end{equation}

\subsubsection{} Note that the partially defined left adjoint $\pi^0_{P,!}$ of the functor $\pi^{0,!}_{P}$ is defined on $\omega_{\on{Gr}_{P,\Bun_M}}$ by holonomicity of the dualizing sheaf. Moreover, by the universal homological contractibility statement of Proposition \ref{prop: relative Grassmannian contractibility}, the counit map
\begin{equation}\label{eq: counit for dualizing is an iso}\pi^0_{P,!}(\omega_{\on{Gr}_{P,\Bun_M}}) \to \omega_{\Bun_P \times \on{Ran}}\end{equation}

\noindent is an isomorphism. 

Since $\pi^!_{P,\on{pol}}$ is $\on{Rep}(\check{M} \times \check{G})_{\on{Ran}}$-equivariant and by rigidity of the latter, the functor $\pi_{P,\on{pol},!}$ is also $\on{Rep}(\check{M} \times \check{G})_{\on{Ran}}$-equivariant. We therefore get induced functors
\[\pi_{P,\on{pol},!}:\on{DrPl}_{\check{M},\check{G}}(D({_{\Bun_M}}{\Gr_{G,\Ran}})) \longleftrightarrow \on{DrPl}_{\check{M},\check{G}}(D(\ot{\Bun}_{P,\on{pol}})): \pi^{!}_{P,\on{pol}},\]

\noindent where the functor from left to right is only partially defined. We have similar assertions for enhanced Drinfeld-Plücker structures and Hecke structures.

\subsubsection{} Recall that we denote by $\mathbf{j}_!$ the $!$-extension of $\omega_{{S^0_{P,\Ran}}}$ to $\mathfrak{L}^+_{\on{Ran}}\backslash {\Gr_{G,\on{Ran}}}$. By Proposition \ref{p:restostrata!}, $\mathbf{j}_!$ is equipped with a canonical enhanced Drinfeld-Plücker structure. Since the projection
\[\on{pr}_{\Bun_M}: {_{\Bun_M}}{\Gr_{G,\Ran}} \longrightarrow \mathfrak{L}^+_{\Ran}M \backslash \Gr_{G,\Ran}\]

\noindent commutes with Hecke correspondences, the sheaf 
\[{_{\Bun_M}}\mathbf{j}_! \coloneqq \on{pr}^!_{\Bun_M}(\mathbf{j}_!)\]

\noindent is also equipped with a Drinfeld-Plücker structure. As a result, the sheaf
\[{_{\Bun_M}}{\IC^{\frac{\infty}{2}}_{P,\Ran}} \coloneqq \on{pr}_{\Bun_M}^!(\IC^{\frac{\infty}{2}}_{P,\Ran}) \]

\noindent comes with a canonical identification:
\[
{_{\Bun_M}}{\IC^{\frac{\infty}{2}}_{P,\Ran}}\simeq \on{Ind}_{\on{EnhDrPl}_{\check{M},\check{G}}}^{\on{Hecke}_{\check{M},\check{G}}}({_{\Bun_M}}\mathbf{j}_!).
\]

\subsubsection{} Let us also denote by $\mathbf{j}^{\on{glob}}_!$ the $!$-extension of $\omega_{\Bun_P \times \on{\Ran}}$ along $j_{P,\on{pol}}: \Bun_P \times \on{\Ran}\to \widetilde{\on{Bun}}_{P,\on{pol}}$. By (\ref{eq: counit for dualizing is an iso}), we have:
\begin{equation}\label{eq:glob}
\pi_{P,\on{pol},!}({_{\Bun_M}}\mathbf{j}_!)\simeq\mathbf{j}_!^{\on{glob}}.
\end{equation}

\noindent By Proposition \ref{p:restostrata!}, $\mathbf{j}^{\on{glob}}_!$ is equipped with an enhanced Drinfeld-Plücker structure.

The following lemma is immediate from universal homological contractibility of $\pi_{P,\on{pol}}$ and $\on{Rep}(\check{M} \times \check{G})_{\on{Ran}}$-equivariance:
\begin{lem}\label{hecke-to-hecke}
    There is a canonical isomorphism:
    \[{_{\Bun_M}}{\IC^{\frac{\infty}{2}}_{P,\Ran}}= \on{Ind}_{\on{EnhDrPl}_{\check{M},\check{G}}}^{\on{Hecke}_{\check{M},\check{G}}}({_{\Bun_M}}\mathbf{j}_!) \simeq \pi^!_{P,\on{pol}}(\on{Ind}_{\on{EnhDrPl}_{\check{M},\check{G}}}^{\on{Hecke}_{\check{M},\check{G}}}(\mathbf{j}^{\on{glob}}_!)).\]
\noindent Moreover, the counit for $\pi^!_{P,\on{pol}}$ applied to $\on{Ind}_{\on{DrPl}}^{\on{Hecke}}(\mathbf{j}^{\on{glob}}_!)$ gives an isomorphism:
\[
\pi_{P,\on{pol},!}\circ \pi_ {P,\on{pol}}^!(\on{Ind}_{\on{EnhDrPl}_{\check{M},\check{G}}}^{\on{Hecke}_{\check{M},\check{G}}}(\mathbf{j}^{\on{glob}}_!))\to \on{Ind}_{\on{EnhDrPl}_{\check{M},\check{G}}}^{\on{Hecke}_{\check{M},\check{G}}}(\mathbf{j}^{\on{glob}}_!).
\]
\end{lem}

\subsubsection{} Let $i_{P,\on{pol}}$ denote the closed embedding:
\[i_{P,\on{pol}}: \ot{\Bun}_{P,\on{zer}}\to \ot{\Bun}_{P,\on{pol}}.\]

\noindent Moreover, let $\on{oblv}_{\on{zer}}$ denote the forgetful map
\[
\on{oblv}_{\on{zer}}: \ot{\Bun}_{P,\on{zer}}\to \ot{\Bun}_{P}.
\]

From the IC-sheaf on $\ot{\Bun}_P$ with respect to the perverse $t$-structure, we define the sheaf:
\[\IC_{\ot{\Bun}_{P,\on{pol}}} \coloneqq i_{\on{pol},*}(\on{oblv}_{\on{zer}}^!(\IC_{\ot{\Bun}_P})).\]

\subsubsection{} In what follows, we will use the notation $\on{dim}(\Bun_P)$ to denote the locally constant function on $\Bun_P$ that takes the value 
\[\on{dim}(\Bun_M)+\on{dim}(\Bun_{N_P})+\langle 2(\rho_G-\rho_M),\eta \rangle\]

\noindent on the component of $\Bun_P$ living over the connected component $\Bun^{\eta}_M$ of $\Bun_M$. We have the following local-to-global result:

\begin{thm}\label{t:IC-to-hecke}
    There is a canonical isomorphism:
    \[{_{\Bun_M}}{\IC^{\frac{\infty}{2}}_{P,\Ran}} \simeq \pi^!_{P,\on{pol}}(\IC_{\ot{\Bun}_{P,\on{pol}}}[\on{dim}(\Bun_P)]).\]
\noindent Moreover, the counit morphism
\[\pi_{P,\on{pol},!}({_{\Bun_M}}{\IC^{\frac{\infty}{2}}_{P,\Ran}}) \longrightarrow \IC_{\ot{\Bun}_{P,\on{pol}}}[\on{dim}(\Bun_P)]\]
\noindent is an isomorphism. 
\end{thm}

\begin{proof}
By Lemma \ref{hecke-to-hecke}, it suffices to prove that there is a canonical identification:
\[\on{Ind}_{\on{EnhDrPl}_{\check{M},\check{G}}}^{\on{Hecke}_{\check{M},\check{G}}}(\mathbf{j}_!^{\on{glob}})[-\on{dim}(\Bun_P)] \simeq \IC_{\ot{\Bun}_{P,\on{pol}}}.\]

By Lemma \eqref{l:ICisSI} and (\ref{eq:glob}), we get that $\on{Ind}_{\on{EnhDrPl}_{\check{M},\check{G}}}^{\on{Hecke}_{\check{M},\check{G}}}(\mathbf{j}_!^{\on{glob}})[-\on{dim}(\Bun_P)]$ is supported on the image of the closed embedding $i_{P,\on{pol}}$. We have a commutative diagram:
\begin{equation}\label{eq: local to global comparison diagram}\begin{tikzcd}
	& {\on{Gr}_{P,\Bun_M}^{\theta}} & {\ot{\Gr}_{P,\Bun_M}} & {{_{\Bun_M}}{\Gr_{G,\Ran}}} \\
	& {{}_{\theta}\ot{\Bun}_{P,\on{zer}}} & {\ot{\Bun}_{P,\on{zer}}} & {\ot{\Bun}_{P,\on{pol}}} \\
	{\sH^{+}_{M, X^{\theta}}} & {{}_{\theta}\ot{\Bun}_P} & {\ot{\Bun}_P}
	\arrow[from=1-2, to=1-3]
	\arrow[from=1-2, to=2-2]
	\arrow[curve={height=12pt}, from=1-2, to=3-1]
	\arrow[from=1-3, to=1-4]
	\arrow[from=1-3, to=2-3]
	\arrow[from=1-4, to=2-4]
	\arrow[from=2-2, to=2-3]
	\arrow[from=2-2, to=3-2]
	\arrow[from=2-3, to=2-4]
	\arrow[from=2-3, to=3-3]
	\arrow[from=3-2, to=3-1]
	\arrow[from=3-2, to=3-3]
\end{tikzcd}\end{equation}

\noindent with all squares Cartesian and all vertical arrows universally homologically contractible. For the remainder of the proof, we will denote the $!$-pullback of $\on{Ind}_{\on{EnhDrPl}_{\check{M},\check{G}}}^{\on{Hecke}_{\check{M},\check{G}}}(\mathbf{j}_!^{\on{glob}})[-\on{dim}(\Bun_P)]$ to $\ot{\Bun}_{P,\on{zer}}$ by $\mathscr{I}$. We immediately note that $\mathscr{I}$ descends to a sheaf on $\ot{\Bun}_P$ by unitality, which we will also denote by $\mathscr{I}$. To complete the proof, it suffices to show that as a sheaf on $\ot{\Bun}_P$, the $!$ (resp. $*$) restrictions of $\mathscr{I}$ to the strata indexed by $\theta \neq 0$ lie in perverse cohomoloigcal degrees $\geq 1$ (resp. $\leq -1$), and that the restriction of $\mathscr{I}$ to $\Bun_P$ coincides with $\omega_{\on{Bun}_P}[-\on{dim}(\on{Bun}_P)]$.

\step Let us first compute the $*$-restriction of $\mathscr{I}$ to ${}_{\theta}\ot{\Bun}_P$. All sheaves in question are ind-holonomic, and hence by base change along the diagram \eqref{eq: local to global comparison diagram}, $\iota_{\theta}^*(\mathscr{I})$ is the $!$-pushforward of a sheaf $\cG^{\theta}$ on $\on{Gr}^{\theta}_{P,\Bun_M}$ to ${}_{\theta}\ot{\Bun}_P$. By Proposition \ref{p:strata}, $\cG^{\theta}$ has the property that it is the $!$-pullback of $\mathcal{O}(\check{N}_P)_{X^{\theta}}[-\langle 2(\rho_G - \rho_M), \theta \rangle -\on{dim}(\Bun_P)]$ along the composition
\[
\on{Gr}_{P,\Bun_M}^{\theta}\to \sH_{M,X^{\theta}}^+\to \mathfrak{L}^+_{X^{\theta}} M \backslash \Gr^+_{M,X^{\theta}}.
\]

\noindent By universal homological contractibility of the map $\on{Gr}^{\theta}_{P,\Bun_M} \to {}_{\theta}\ot{\Bun}_P$, it follows that $\iota^*_{\theta}(\mathscr{I})$ is the $!$-pullback of 
\[\mathcal{O}(\check{N})_{X^{\theta}}[-\langle 2(\rho_G - \rho_M),\theta \rangle - \on{dim}(\Bun_P) ]\]

\noindent along the projection
\begin{equation}\label{eq:estimation1}
{}_{\theta}\ot{\Bun}_P \to \mathfrak{L}^+_{X^{\theta}}M \backslash \Gr^+_{M,X^{\theta}}.
\end{equation}

\noindent We factor the above map as the composition:
\[
{}_{\theta}\ot{\Bun}_P\to \sH_{M,X^{\theta}}^+ \to\mathfrak{L}^+_{X^{\theta}}M \backslash \Gr^+_{M,X^{\theta}}.
\]

\noindent Recall the t-structure from §\ref{s:tstructure}. A similar argument as in the proof of \cite[Lemma 2.1.15]{beraldo2021geometric} shows that $!$-pullback along the map $\sH_{M,X^{\theta}}^+ \to\mathfrak{L}^+_{X^{\theta}}M \backslash \Gr^+_{M,X^{\theta}}$ is t-exact up to a cohomological shift by $\on{dim}(\on{Bun}_M)$.

Next, fix a component $\Bun_M^{\eta}$ of $\Bun_M$ and consider the stratum
\[\iota_{\theta}: {_{\theta}}{\ot{\Bun}_P^{\eta}} \to \ot{\Bun}_P^{\eta}.\]

\noindent By definition of the former, we have an isomorphism:
\[{_{\theta}}{\ot{\Bun}^{\eta}_P} \overset{\sim}{\longrightarrow} \sH^{+}_{M, X^{\theta}} \underset{\on{Bun}_M}{\times} \Bun_P^{\theta+\eta}.\]

\noindent The projection
\[{_{\theta}}{\ot{\Bun}^{\eta}_P} \simeq \sH^{+}_{M, X^{\theta}} \underset{\on{Bun}_M}{\times} \Bun^{\theta+\eta}_P \longrightarrow \sH^{+}_{M, X^{\theta}}\]

\noindent is smooth of relative dimension $\on{dim}(\Bun_{N_P}) + \langle 2(\rho_G-\rho_M),\theta+\eta \rangle$, by \cite[Corollary 2.2.9]{frenkel2001whittaker} and the discussion in \emph{loc. cit}. Moreover, for $\theta\neq 0$, the sheaf $\mathcal{O}(\check{N}_P)_{X^{\theta}}$ lives in perverse degrees $\leq -1$, cf. Lemma \ref{l:oleq-1}.

Putting the results together, it follows that for $\theta \neq 0$, the result of $!$-pulling the sheaf

\[\mathcal{O}(\check{N}_P)_{X^{\theta}}[\langle -2(\rho_G-\rho_M),\theta \rangle -\on{dim}(\Bun_P)]\]

\noindent back to ${_{\theta}}{\ot{\Bun}^{\eta}_P}$ lives in perverse degrees at most
\begin{multline*}
 -1 - \on{dim}(\Bun_M) - \on{dim}(\Bun_{N_P}) - \langle 2(\rho_G - \rho_M), \theta+\eta \rangle + \langle 2(\rho_G-\rho_M),\theta \rangle + \on{dim}(\Bun_P) = \\-1 - \on{dim}(\Bun_M) - \on{dim}(\Bun_{N_P}) - \langle 2(\rho_G - \rho_M), \theta+\eta \rangle + \langle 2(\rho_G-\rho_M),\theta \rangle \\+ \on{dim}(\Bun_M) +\on{dim}(\Bun_{N_P}) + \langle 2(\rho_G-\rho_M),\eta \rangle = -1.
 \end{multline*}

\step For $\theta = 0$, the $!$-pullback of $\cO(\check{N}_P)_{X^0}[-\on{dim}(\on{Bun}_P)]$ along (\ref{eq:estimation1}) evidently gives $\omega_{\on{Bun}_P}[-\on{dim}(\on{Bun}_P)]$.

\step Using Lemma \ref{l:rescomp}, a similar computation to Step 1 shows that the $!$-restriction of $\mathscr{I}$ along $\iota_{\theta}:{}_{\theta}\widetilde{\on{Bun}}_P\to \widetilde{\on{Bun}}_P$ lives in perverse degrees $\geq 1$. It follows that $\mathscr{I}$, when viewed as a sheaf on $\ot{\Bun}_P$, is canonically isomorphic to $\IC_{\ot{\Bun}_P}$, concluding the proof. 
\end{proof}

\begin{cor}\label{c:globrestostrata}
We have canonical isomorphisms:
\[
\iota_{\theta}^!(\on{IC}_{\widetilde{\on{Bun}}_P})\simeq \fU(\cnp)_{X^{\theta}}\widetilde{\boxtimes}\on{IC}_{\on{Bun}_P};
\]
\[
\iota_{\theta}^*(\on{IC}_{\widetilde{\on{Bun}}_P})\simeq \cO(\check{N}_P)_{X^{\theta}}\widetilde{\boxtimes}\on{IC}_{\on{Bun}_P}
\]

\noindent as sheaves on ${}_{\theta}\widetilde{\on{Bun}}_P\simeq \sH_{M,X^{\theta}}^+\underset{\on{Bun}_M}{\times} \on{Bun}_P$.
\end{cor}

\begin{proof}
Using the second assertion of Theorem \ref{t:IC-to-hecke}, we essentially proved that
\[
\iota_{\theta}^*(\on{IC}_{\widetilde{\on{Bun}}_P})\simeq \cO(\check{N}_P)_{X^{\theta}}\widetilde{\boxtimes}\on{IC}_{\on{Bun}_P}
\]

\noindent in the course of the proof of \emph{loc.cit}.\footnote{More precisely, we saw that we have an isomorphism $\iota_{\theta}^*(\on{IC}_{\widetilde{\on{Bun}}_P})\simeq \cO(\check{N}_P)_{X^{\theta}}\widetilde{\boxtimes}\on{IC}_{\on{Bun}_P}$, up to a cohomological shift by some integer $d$. However, tracing through the proof of Theorem \ref{t:IC-to-hecke} and keeping track of the shifts, we see that $d=0$.}

By Verdier self-duality of $\on{IC}_{\widetilde{\on{Bun}}_P}$ and (\ref{eq:VerdierdualityFU}), we get the assertion for $\iota_{\theta}^!(\on{IC}_{\widetilde{\on{Bun}}_P})$.

\end{proof}

\newpage

\section{Restriction of representations and the Casselman-Shalika formula}\label{S:Whittakerness}

\subsection{Whittaker categories} For a factorization category $\sC$ with an action of the loop group $\mathfrak{L}_{\Ran}G$, we consider its \emph{Whittaker category} $\on{Whit}(\sC)$. When $\sC$ is the category of D-modules on some prestack $\mathcal{Y}$, we will often use the notation $\on{Whit}(\mathcal{Y}):=\on{Whit}(D(\mathcal{Y}))$. In particular, we let $\on{Whit}(\Gr_{G,\Ran})$ denote the \emph{spherical Whittaker category}, i.e. the Whittaker category for D-modules on the affine Grassmannian. 

\subsubsection{} The category $\on{Whit}(\sC):=\sC^{\fL_{\on{Ran}} N^-,\psi_{N^-,\on{Ran}}}$ consists of objects in $\sC$ that are equivariant for $\mathfrak{L}_{\Ran}N^-$ against a non-degenerate character $\psi_{N^-,\on{Ran}}=\psi_{\fL_{\on{Ran}} N^-}$ of $\fL_{\on{Ran}}N^-$. In particular $\on{Whit}(\sC)$ is equipped with a fully-faithful forgetful functor
\[\Oblv^{\on{Whit}}_{\sC}: \on{Whit}(\sC) \longrightarrow \sC \]

\noindent admitting a (non-continuous) right adjoint. We will renormalize the right adjoint to be continuous in the case of $\sC = D(\Gr_{G,\Ran})$ as follows. There are tautological equivalences:
\[\on{Whit}(\mathfrak{L}_{\Ran}N^-)\simeq \on{Whit}(\on{Gr}_{N^-,\on{Ran}}) \simeq D(\Ran). \]

\noindent We write $\psi_{N^-,\on{Ran}}$ for the images of $\omega_{\on{Ran}}$ under these equivalences. We write $\psi_{G,\on{Ran}}$ for the $!$-extension of $\psi_{N^-,\on{Ran}}\in D(\on{Gr}_{N^-,\on{Ran}})$ along $\on{Gr}_{N^-,\on{Ran}}\to \on{Gr}_{G,\on{Ran}}$.

Moreover, we write $\psi_{N^-,x}$ for the restriction of $\psi_{N^-,\on{Ran}}$ to $x\in X$.

\subsubsection{}\label{s:whitaveraging} Now, convolution provides a pairing 
\[ -\underset{\mathfrak{L}_{\Ran}G}{\star}-: \on{Whit}(\mathfrak{L}_{\Ran}G) \otimes D(\Gr_{G,\Ran}) \to \on{Whit}(\Gr_{G,\Ran})\]

\noindent and for an object $\cF$ of $D(\Gr_{G,\Ran})$, we write:
\[\Av^{\mathfrak{L}_{\Ran}N^-,\psi}_*(\cF) \coloneqq \psi_{G,\on{Ran}} \underset{\mathfrak{L}_{\Ran}G}{\star} \cF.\]

\noindent We call the functor $\Av^{\mathfrak{L}_{\Ran}(N^-),\psi}_*: D(\Gr_{G,\Ran})\to \on{Whit}(\Gr_{G,\Ran})$ \emph{renormalized Whittaker averaging}. 

\subsubsection{} Denote by $\on{Whit}(\Gr_{G,x})$ the fiber of $\on{Whit}(\Gr_{G,\Ran})$ at a $k$-point $x$ of $X$. Then $\on{Whit}(\Gr_{G,x})$ has a set of compact generators $\psi_{\mu}$ where $\mu$ ranges over anti-dominant coweights of $G$. Specifically, for some anti-dominant $\mu$, the sheaf $\psi_{\mu}$ is given by $!$-extending the object corresponding to the unit under the evident equivalence 
\[\on{Whit}(S^{-,\mu}_x)=D(S^{-,\mu}_x)^{\fL_x N^-,\psi_{N^-,x}} \simeq \Vect,\]

\noindent where $S^{-,\mu}_x$ denotes the $\mu$-th semi-infinite orbit for $N^-$ in $\Gr_{G,x}$. By the cleanness theorem of \cite{frenkel2001whittaker}, $\psi_{\mu}$ is equivalently the $*$-extension from $S^{-,\mu}_x$.

\subsubsection{} Define a t-structure on $\on{Whit}(\Gr_{G,x})$ by declaring an object $\cF$ to lie in $\on{Whit}(\Gr_{G,x})^{\leq 0}$ if an only if 
\[H^0\Hom_{\on{Whit}(\Gr_{G,x})}(\cF,\psi_{\mu}[-k])=0\]

\noindent for all anti-dominant $\mu$ for all integers $k > 0$.

We have the following geometric Casselman-Shalika formula (see \cite{frenkel2001whittaker} and \cite{raskin2018chiralII}):

\begin{thm}\label{Casselman-Shalika} There is a canonical equivalence of $\on{Rep}(\check{G})_{\Ran}$ factorization module categories 
\[\on{CS}_G: \on{Whit}(\Gr_{G,\Ran}) \overset{\sim}{\longrightarrow} \on{Rep}(\check{G})_{\Ran}\]

\noindent The $!$-fiber at any $k$-point $x \in X$ gives a t-exact equivalence 
\[\on{CS}_{G,x}: \on{Whit}(\Gr_{G,x}) \overset{\sim}{\longrightarrow} \on{Rep}(\check{G})\]

\noindent that sends $\psi_{\mu}$ to the irreducible representation of highest weight $w_0(\mu)$. 
\end{thm}



%

%

%




\subsection{The Jacquet functors} In this section we will construct several functors
\[\on{Whit}(\Gr_{G,\Ran}) \longrightarrow \on{Whit}(\Gr_{M,\Ran}),\]

\noindent where the Whittaker condition on $D(\Gr_{M,\Ran})$ is understood to be relative to $\mathfrak{L}_{\Ran}N^-_M$ via the non-degenerate character $\psi_{N^-,\on{Ran}}$ obtained by restricting the same named character along the inclusion $\mathfrak{L}_{\Ran}N^-_M \hookrightarrow \mathfrak{L}_{\Ran}N^-$.

\subsubsection{} Consider the prestack
\[{_{\Gr_{M,\Ran}}}{\Gr_{G,\Ran}} \coloneqq 
 \Gr_{M,\Ran} \underset{\bB \mathfrak{L}^+_{\Ran}M}{\times}\mathfrak{L}^+_{\Ran}M \backslash \Gr_{G,\Ran}.\]

\noindent Unwinding the definitions, we see there is a canonical map 
\begin{equation}\label{eq: product to fiber product}{_{\Gr_{M,\Ran}}}{\Gr_{G,\Ran}} \longrightarrow \Gr_{G,\Ran} \underset{\on{Ran}}{\times} \Gr_{M,\Ran}\end{equation}

\noindent which is easily seen to be an isomorphism. Hence we have an equivalence of categories:
\begin{equation}\label{eq: D-mods to functors}D({_{\Gr_{M,\Ran}}}{\Gr_{G,\Ran}}) \overset{\sim}{\longrightarrow} \Hom_{D(\Ran)}(D(\Gr_{G,\Ran}),D(\Gr_{M,\Ran})).\end{equation}

\noindent That is, we can view D-modules on ${_{\Gr_{M,\Ran}}}{\Gr_{G,\Ran}}$ as kernels of functors $D(\Gr_{G,\Ran}) \to D(\Gr_{M,\Ran})$ relative to the action of $D(\Ran)$ on both sides. 

\subsubsection{} Note we have a diagonal action of $\mathfrak{L}_{\Ran}P$ on $\Gr_{G,\Ran} \underset{\on{Ran}}{\times} \Gr_{M,\Ran}$ and we have an isomorphism
\[\mathfrak{L}^+_{\Ran}N_P \mathfrak{L}_{\Ran}M \backslash \Gr_{G,\Ran} \overset{\sim}{\longrightarrow} \mathfrak{L}_{\Ran}P\backslash (\Gr_{G,\Ran} \underset{\on{Ran}}{\times} \Gr_{M,\Ran})\]

\noindent such that the diagram
\[\begin{tikzcd}
	{{_{\Gr_{M,\Ran}}}{\Gr_{G,\Ran}}} & {\Gr_{G,\Ran} \underset{\on{Ran}}{\times} \Gr_{M,\Ran}} \\
	{\mathfrak{L}^+_{\Ran}M\mathfrak{L}_{\Ran}N_P \backslash \Gr_{G,\Ran}} & {\mathfrak{L}_{\Ran}P\backslash (\Gr_{G,\Ran} \underset{\on{Ran}}{\times} \Gr_{M,\Ran})}
	\arrow[from=1-1, to=1-2]
	\arrow[from=1-1, to=2-1]
	\arrow[from=1-2, to=2-2]
	\arrow[from=2-1, to=2-2]
\end{tikzcd}\]

 \noindent commutes. Here the left vertical map is the projection to $\mathfrak{L}^+_{\Ran}M \backslash \Gr_{G,\Ran}$ followed by the canonical map to $\mathfrak{L}^+_{\on{Ran}}M \mathfrak{L}_{\Ran}N_P \backslash \Gr_{G,\Ran}$, and the right vertical map is the quotient map.

 It follows that we obtain a canonical equivalence:
 \[
 \on{SI}_{P,\on{Ran}}:=D(\mathfrak{L}^+_{\Ran}N_P \mathfrak{L}_{\Ran}M \backslash \Gr_{G,\Ran})\simeq \on{Hom}_{D(\mathfrak{L}_{\on{Ran}}P)}(D(\Gr_{G,\Ran}), D(\Gr_{M,\Ran})).
 \]
 
\noindent That is, we can view objects of the semi-infinite category as as $\mathfrak{L}_{\Ran}P$-equivariant functors from $D(\Gr_{G,\Ran})$ to $ D(\Gr_{M,\Ran})$. 

\subsubsection{} We therefore obtain a canonical pairing 
\[D(\Gr_{G,\Ran}) \underset{D(\on{Ran})}{\otimes} \on{SI}_{P,\Ran} \longrightarrow D(\Gr_{M,\Ran}).\]

By forgetting from $\mathfrak{L}_{\Ran}N^-$ equivariance against $\psi$ to $\mathfrak{L}_{\Ran}N^-_M$ equivariance against the restriction of $\chi$ to the latter, we obtain a pairing:
\begin{equation}\label{eq:whitpairing}
\on{Whit}(\Gr_{G,\Ran}) \underset{D(\on{Ran})}{\otimes} \on{SI}_{P,\Ran} \longrightarrow \on{Whit}(\Gr_{M,\Ran}).
\end{equation}

\noindent In particular, from the objects $\IC^{\frac{\infty}{2}}_{P,\Ran}, \mathbf{j}_!\in \on{SI}_{P,\on{Ran}}$, we get functors
\[\on{Jac}^M_{!*}, \, \on{Jac}^M_! : \on{Whit}(\Gr_{G,\Ran}) \longrightarrow \on{Whit}(\Gr_{M,\Ran}),\]

\noindent respectively. 

\subsubsection{}\label{s:heckeandenhdrplofkernels} Since $\sIC$ is an object of $\on{Hecke}_{\check{M},\check{G}}(\on{SI}_{P,\on{Ran}})$ by definition, we get that $\on{Jac}_{!*}^M$ is $\on{Rep}(\check{G})_{\on{Ran}}$-linear. That is:
\[
\on{Jac}_{!*}^M(V\star -)\simeq \on{Res}^{\check{G}}_{\check{M}}(V)\star \on{Jac}_{!*}^M(-),\;\; V\in \on{Rep}(\check{G})_{\on{Ran}}.
\]

\noindent Similarly, from the enhanced Drinfeld-Plücker structure on $\mathbf{j}_!$, we get:
\[
\on{Jac}_{!}^M(V\star -)\simeq C^{\bullet}(\check{\fn}_P,V)\underset{\Omega(\check{\fn}_P)_{\on{Ran}}}{\star} \on{Jac}_{!}^M(-),\;\; V\in \on{Rep}(\check{G})_{\on{Ran}}.
\]

\subsection{Composing Jacquet functors} We have Jacquet functors
\[
\on{Jac}^M_{!*}: \on{Whit}(\Gr_{G,\Ran}) \longrightarrow \on{Whit}(\Gr_{M,\Ran});
\]
\[\on{Jac}^{T_M}_{!*}: \on{Whit}(\Gr_{M,\Ran}) \longrightarrow D(\Gr_{T,\Ran})\]

\noindent associated to the parabolics $P$ and $B_M \coloneqq B \cap M$, respectively. In the notation, we denote by $T_M$ the maximal torus $T \subseteq M$ to emphasize that the domain of $\on{Jac}^{T_M}_{!*}$ is the Whittaker category of $M$. In this section we will show how to compose Jacquet functors for $G$ and $M$.

\subsubsection{}\label{s:Grptilderan} In what follows, denote by $\on{Gr}_{P,\Ran}$ the fiber product 
\[
\on{Gr}_{P,\Ran}:=\Gr_{M,\Ran} \underset{\mathbb{B}\mathfrak{L}^+_{\Ran}M}{\times} \backslash \mathfrak{L}^+_{\Ran}M S^0_{\Ran}.
\]

\noindent Similarly, define:
\[\ot{\Gr}_{P,\Ran} \coloneqq \Gr_{M,\Ran} \underset{\mathbb{B}\mathfrak{L}^+_{\Ran}M}{\times} \mathfrak{L}^+_{\Ran}M\backslash \ot{S}^0_{\Ran}.\]

We will also denote by ${_{\Gr_M}}{\on{IC}^{\frac{\infty}{2}}_{P,\Ran}}$ the $!$-pullback of $\on{IC}^{\frac{\infty}{2}}_{P,\Ran}$ to $\ot{\Gr}_{P,\Ran}$ along the projection:
\[\ot{\Gr}_{P,\Ran} \longrightarrow \mathfrak{L}^+_{\Ran}M\backslash \ot{S}^0_{\Ran}.\]

\noindent We have the following theorem. 

\begin{thm}\label{thm: composing eisenstein} The composition $\on{Jac}^{T_M}_{!*} \circ \on{Jac}^M_{!*}$ is canonically equivalent to the Jacquet functor
\[\on{Jac}^T_{!*}: \on{Whit}(\Gr_{G,\Ran}) \to D(\Gr_{T,\Ran})\]

\noindent associated to the Borel subgroup $B$ of $G$. 
\end{thm}
\begin{proof}
 \step
 We will consider the following commutative diagram
    \[\begin{tikzcd}
	&& {\ot{\Gr}_{P,\Ran} \underset{\Gr_{M,\Ran}}{\times} \overline{\Gr}_{B_M,\Ran}} \\
	& {\ot{\Gr}_{P,\Ran}} && {\overline{\Gr}_{B_M,\Ran}} \\
	{\Gr_{G,\Ran}} && {\Gr_{M,\Ran}} && {\Gr_{T,\Ran}.}
	\arrow["{\ot{\mathfrak{p}}_{P,B_M,\Ran}}", from=1-3, to=2-2]
	\arrow["{\ot{\mathfrak{q}}_{P,B_M,\Ran}}"', from=1-3, to=2-4]
	\arrow["{\ot{\mathfrak{p}}_{P,\Ran}}", from=2-2, to=3-1]
	\arrow["{\ot{\mathfrak{q}}_{P,\Ran}}"', from=2-2, to=3-3]
	\arrow["{\overline{\mathfrak{p}}_{B_M,\Ran}}", from=2-4, to=3-3]
	\arrow["{\overline{\mathfrak{q}}_{B_M,\Ran}}"', from=2-4, to=3-5]
\end{tikzcd}\]

\noindent Base change and the projection formula shows that the functor $\on{Jac}^{T_M}_{!*} \circ \on{Jac}^M_{!*}$ is given by 
\[\cF \longmapsto (\overline{\mathfrak{q}}_{B_M,\Ran} \circ \ot{\mathfrak{q}}_{P,B_M,\Ran})_*((\ot{\mathfrak{p}}_{P,\Ran} \circ \ot{\mathfrak{p}}_{P,B_M,\Ran})^!(\cF) \overset{!}{\otimes} {_{\Gr_M}}\IC^{\frac{\infty}{2}}_{P,\Ran} \widetilde{\boxtimes} {_{\Gr_T}}\IC^{\frac{\infty}{2}}_{B_M,\Ran}).\]

By composing Pl\"{u}cker data and forgetting the $M$-bundle, we obtain a map 
\[\mathfrak{r}: \ot{\Gr}_{P,\Ran} \underset{\Gr_{M,\Ran}}{\times} \overline{\Gr}_{B_M,\Ran} \longrightarrow \overline{\Gr}_{B,\Ran},\]

\noindent which is an equivalence on $\Gr_{P,\Ran} \underset{\Gr_{M,\Ran}}{\times} \Gr_{B_M,\Ran} \simeq \Gr_{B,\Ran}$ and such that 
\[\overline{\mathfrak{q}}_{B_M,\Ran} \circ \ot{\mathfrak{q}}_{P,B_M,\Ran} = \overline{\mathfrak{q}}_{B,\Ran} \circ \mathfrak{r} \,\,\, \text{and} \,\,\, \ot{\mathfrak{p}}_{P,\Ran} \circ \ot{\mathfrak{p}}_{P,B_M,\Ran} = \overline{\mathfrak{p}}_{B,\Ran} \circ \mathfrak{r}.\]

Another application of the projection formula shows that the composition $\on{Jac}^{T_M}_{!*} \circ \on{Jac}^M_{!*}$ is given by 
\[\cF \longmapsto \overline{\mathfrak{q}}_{B,\Ran,*}(\overline{\mathfrak{p}}^!_{B,\Ran}(\cF) \overset{!}{\otimes} \mathfrak{r}_*({_{\Gr_M}}\IC^{\frac{\infty}{2}}_{P,\Ran} \widetilde{\boxtimes} {_{\Gr_T}}\IC^{\frac{\infty}{2}}_{B_M,\Ran})).\]

To conclude the proof, it therefore suffices to produce a canonical isomorphism 
\[{_{\Gr_T}}\IC^{\frac{\infty}{2}}_{B,\Ran} \overset{\sim}{\longrightarrow} \mathfrak{r}_*({_{\Gr_M}}\IC^{\frac{\infty}{2}}_{P,\Ran} \widetilde{\boxtimes} {_{\Gr_T}}\IC^{\frac{\infty}{2}}_{B_M,\Ran}).\]

\step

\noindent To this end, note we have closed embeddings
\[\ot{\Gr}_{P,\Ran} \longhookrightarrow { _{\Gr_{M,\Ran}}}{\Gr_{G,\Ran}}\]

\noindent and
\[\ol{\Gr}_{B_M,\Ran} \longhookrightarrow {_{\Gr_{T,\Ran}}}{\Gr_{M,\Ran}}.\]

\noindent Hence we can identify ${_{\Gr_M}}{\IC^{\frac{\infty}{2}}_{P,\Ran}}$ and ${_{\Gr_T}}{\IC^{\frac{\infty}{2}}_{B_M,\Ran}}$ with their pushforwards along the above maps. Let us further recall that by \eqref{eq: product to fiber product} we have isomorphisms 
\[ { _{\Gr_{M,\Ran}}}{\Gr_{G,\Ran}} \simeq \Gr_{G,\Ran} \underset{\Ran}{\times} \Gr_{M,\Ran}\]

\noindent and

\[  { _{\Gr_{T,\Ran}}}{\Gr_{M,\Ran}} \simeq \Gr_{M,\Ran} \underset{\Ran}{\times} \Gr_{T,\Ran}\]

\noindent that commute with Hecke correspondences for $G$ and $M$ in the first case, and for $M$ and $T$ in the second.

By the definition of the semi-infinite IC sheaf (see §\ref{sec: def of semiinfinite ic}), we have: 
\[\on{IC}_{P,\on{Ran}}^{\frac{\infty}{2}}:=\on{Ind}_{\on{DrPl}_{\check{M},\check{G}}}^{\Hecke_{\check{M},\check{G}}}(\delta_{\on{Gr}_{G,\on{Ran}}})\]

\noindent A similar equation holds for the pair $(M,B_M)$. 

It follows that the pullback of $\on{IC}_{P,\on{Ran}}^{\frac{\infty}{2}}$ to $\Gr_{G,\Ran} \underset{\Ran}{\times} \Gr_{M,\Ran}$ is the Hecke object induced from a Drinfeld-Pl\"{u}cker structure on the delta sheaf ${_{\Gr_M}}{\delta_{\Gr_{G,\Ran}}}$ defined as pushforward of the dualizing sheaf along the closed embedding %
\[\mathbb{B}\fL^+_{\Ran}M \underset{\mathbb{B}\fL^+_{\Ran}M}{\times} \Gr_{M,\Ran} \simeq \Gr_{M,\Ran} \longhookrightarrow \Gr_{G,\Ran} \times_{\Ran} \Gr_{M,\Ran},\]

\noindent and similarly for ${_{\Gr_T}}\IC^{\frac{\infty}{2}}_{B_M,\Ran}$. 

The projection map
\[p_{G,T}: \Gr_{G,\Ran} \underset{\Ran}{\times} \Gr_{M,\Ran} \underset{\Ran}{\times} \Gr_{T,\Ran} \longrightarrow \Gr_{G,\Ran} \underset{\Ran}{\times} \Gr_{T,\Ran}\]

\noindent onto the first and third factors commutes with Hecke functors for $G$ and $T$. It follows that the functor
\[p_{G,T,*}: D(\Gr_{G,\Ran} \underset{\Ran}{\times} \Gr_{M,\Ran} \underset{\Ran}{\times} \Gr_{T,\Ran}) \longrightarrow D(\Gr_{G,\Ran} \underset{\Ran}{\times} \Gr_{T,\Ran})\]

\noindent commutes with the action of the category $\on{Rep}(\check{G} \times \check{M} \times \check{T})_{\Ran}$, where the latter acts on $$D(\Gr_{G,\Ran} \underset{\Ran}{\times} \Gr_{T,\Ran})$$ through the restriction functor 
\[\on{Rep}(\check{G} \times \check{M} \times \check{T})_{\Ran} \longrightarrow \on{Rep}(\check{G} \times \check{T})_{\Ran}\]

\noindent given by forgetting the $\check{M}$ module structure. 

\step

We consider 
\[
\cO(\ol{\check{N}_P\backslash \check{G}})_{\on{Ran}}\otimes \cO(\ol{\check{N}_M\backslash\check{M}})_{\Ran}:=(\cO(\ol{\check{N}_P\backslash \check{G}})\otimes \cO(\ol{\check{N}_M\backslash \check{M}}))_{\on{Ran}}
\]

\noindent as an object of $\on{Rep}(\check{G}\times\check{M}\times\check{T})_{\on{Ran}}$ via the diagonal action of $\check{M}$

The $\cO(\ol{\check{N}_P\backslash \check{G}})_{\Ran}$ module structure on ${_{\Gr_M}}{\delta_{\Gr_G,\Ran}}$ and the $\cO(\ol{\check{N}_M\backslash\check{M}})_{\on{Ran}}$ module structure on ${_{\Gr_T}}{\delta_{\Gr_M,\Ran}}$ give an action of 
\[\cO(\ol{\check{N}_P\backslash \check{G}})_{\Ran} \otimes \cO(\ol{\check{N}_M\backslash \check{M}})_{\Ran}\]

\noindent on the sheaf

\[
{_{\Gr_M}}{\delta_{\Gr_G,\Ran}}\overset{!}{\otimes} {_{\Gr_T}}{\delta_{\Gr_M,\Ran}}\in D(\Gr_{G,\Ran} \underset{\Ran}{\times} \Gr_{M,\Ran} \underset{\Ran}{\times} \Gr_{T,\Ran}),
\]

\noindent and therefore a similar action on 
\[p_{G,T,*}({_{\Gr_M}}{\delta_{\Gr_G,\Ran}} \otimes {_{\Gr_T}}{\delta_{\Gr_M,\Ran}}) \simeq {_{\Gr_T}}{\delta_{\Gr_{G,\Ran}}}.\]

Note that multiplication gives a $\check{G} \times \check{T}$ equivariant morphism 
\[\check{N}_P\backslash \check{G} \times \check{N}_M \backslash \check{M}\longrightarrow \check{N}\backslash \check{G}.\]

\noindent Hence we get a map
\begin{equation}\label{eq: map on functions from action}\cO(\ol{\check{N}\backslash \check{G}})_{\Ran} \longrightarrow \cO(\ol{\check{N}_P\backslash \check{G}})_{\Ran} \otimes \cO(\ol{\check{N}_M\backslash \check{M}})_{\Ran}\end{equation}

\noindent in $\on{Rep}(\check{G} \times\check{M}\times \check{T})_{\Ran}$, where we consider $\check{M}$ as acting trivially on $\cO(\check{N}\backslash \check{G})$.

Restricting along the morphism \eqref{eq: map on functions from action}, we view ${_{\Gr_T}}{\delta_{\Gr_{G,\Ran}}}$ as a module for $\cO(\ol{\check{N}\backslash \check{G}})$; i.e., as a Drinfeld-Pl\"{u}cker object. 

In order to conclude the proof, it is therefore enough to show that the above Drinfeld-Pl\"{u}cker structure on ${_{\Gr_T}}{\delta_{\Gr_{G,\Ran}}}$ coincides with the Drinfeld-Pl\"{u}cker structure constructed in Appendix \eqref{S:DrPLstructure}. The first step towards this is to note the commutativity of the diagram 
\[\begin{tikzcd}
	{\cO(\ol{\check{N}_P\backslash \check{G}})_{\Ran}\otimes\cO(\ol{\check{N}_M\backslash \check{M}})_{\Ran}\text{-mod}(D_{G,T})} & {\cO(\check{G})_{\Ran} \otimes \cO(\check{G})_{\Ran}\text{-mod}(D_{G,T})} \\
	{\cO(\ol{\check{N}\backslash \check{G}})_{\Ran}\text{-mod}(D_{G,T})} & {\cO(\check{G})_{\Ran}\text{-mod}(D_{G,T}),}
	\arrow[from=1-1, to=1-2]
	\arrow[from=1-1, to=2-1]
	\arrow[from=1-2, to=2-2]
	\arrow[from=2-1, to=2-2]
\end{tikzcd}\]

\noindent where we have used the notation
\[D_{G,T} \coloneqq D(\Gr_{G,\Ran} \underset{\Ran}{\times} \Gr_{T,\Ran}).\] 

\noindent Here, the horizontal maps are the induction functors from a Drinfeld-Plücker structure to a Hecke structure (see §\ref{S:factheckedrinf}), and the vertical maps are given by restriction. Commutativity of this diagram follows in turn from the fact that the square 
\[\begin{tikzcd}
	{\check{G} \times \check{G} } & {\check{G}} \\
	{\check{N}_P\backslash\check{G} \times \check{N}_M\backslash \check{M}} & {\check{G}/\check{N}}
	\arrow[from=1-1, to=1-2]
	\arrow[from=1-1, to=2-1]
	\arrow[from=1-2, to=2-2]
	\arrow[from=2-1, to=2-2]
\end{tikzcd}\]

\noindent is Cartesian. 

\step

To identify the Drinfeld-Pl\"{u}cker structures on ${_{\Gr_T}}{\delta_{\Gr_{G,\Ran}}}$, we note that it suffices to identify them over each power $X^I$ of the curve in a compatible way. For simplicity, we identify them when further restricted to a point $x\in X$.\footnote{As remarked in the footnote in §\ref{s:footnote}, all sheaves involved are in the heart of the relative perverse t-structure defined in \cite{hansen2023relative}, and hence it suffices to identify the $\cO(\ol{\check{G}/\check{N}})_{\Ran}$-module structures at a point.}

Let $\sC$ be a category acted on by $\on{Rep}(\check{G} \times \check{M} \times \check{T})$. For $H \in \{\check{G},\check{M},\check{T}\}$, denote by $\underset{H}{\star}$ the corresponding action of $\on{Rep}(H)$ on $\sC$. An action of the algebra $\cO(\ol{\check{N}_P\backslash \check{G}}) \otimes \cO(\ol{\check{N}_M\backslash \check{M}})$ on an object $c$ of $\sC$ amounts to the data of, for every $V \in \on{Rep}(\check{G})$, maps
\[V^{\check{N}_P} \underset{\check{M}}{\star} c \longrightarrow V \underset{\check{G}}{\star} c\]

\noindent satisfying the Pl\"{u}cker compatibilities, together with, for every representation $W \in \on{Rep}(\check{M})$, maps
\[W^{\check{T}} \underset{\check{T}}{\star} c \longrightarrow W \underset{\check{M}}{\star} c\]

\noindent satisfying the Pl\"{u}cker compatibilities. 

\noindent Restricting to an action of $\cO(\ol{\check{N}\backslash \check{G}})$ amounts to taking the composition of the maps:
\[V^{\check{N}} \underset{\check{T}}{\star} c = (V^{\check{N}_P})^{\check{N}_M} \underset{\check{T}}{\star} c \longrightarrow V^{\check{N}_P} \underset{\check{M}}{\star} c \longrightarrow V \underset{\check{G}}{\star} c.\]

The latter assertion can be seen by unwinding the definitions and from the claim that the map \eqref{eq: map on functions from action} identifies $\cO(\ol{\check{N}_P\backslash \check{G}})$ with $\check{M}$ invariants in $\cO(\ol{\check{N}_P\backslash \check{G}}) \otimes \cO(\ol{\check{N}_M\backslash \check{M}})$. 

Lastly, it remains to check that the maps in the Drinfeld-Pl\"{u}cker structures at a point from §\ref{sec: drinfeld-plucker at a point} for the pairs $(\check{G},\check{P})$ and $(\check{M},\check{B}_M)$ compose to the Drinfeld-Pl\"{u}cker structure for the pair $(\check{G},\check{B})$, but this is clear. 
\end{proof}

\subsection{The $!*$-Jacquet functor and restriction}\label{s:!*jac} Let $\on{Res}^{\check{G}}_{\check{M}}$ denote the symmetric monoidal functor 
\[\on{Res}^{\check{G}}_{\check{M}}: \on{Rep}(\check{G})_{\Ran} \longrightarrow \on{Rep}(\check{M})_{\Ran} \]

\noindent obtained from the symmetric monoidal functor of restriction of representations via twisted arrows.

Moreover, denote by $C^{\bullet}(\check{\mathfrak{n}}_P,-)$ the functor 
\[
C^{\bullet}(\check{\mathfrak{n}}_P,-):  \on{Rep}(\check{G})_{\Ran}\to  \on{Rep}(\check{M})_{\Ran}
\]

\noindent given by restricting along $\check{P}\to \check{G}$ and taking Lie algebra cohomology against $\check{\fn}_P$, cf. §\ref{s:inv}.

\subsubsection{} In this subsection we will prove the following theorem. 

\begin{thm}\label{thm: restriction and !* Jacquet are the same} The following diagram commutes:
\[\begin{tikzcd}
	{\on{Whit}(\Gr_{G,\Ran})} & {\on{Whit}(\Gr_{M,\Ran})} \\
	{\on{Rep}(\check{G})_{\Ran}} & {\on{Rep}(\check{M})_{\Ran}.}
	\arrow["{\on{Jac}^M_{!*}}", from=1-1, to=1-2]
	\arrow["{\on{CS}_G}"', from=1-1, to=2-1]
	\arrow["{\on{CS}_M}", from=1-2, to=2-2]
	\arrow["{\on{Res}^{\check{G}}_{\check{M}}}"', from=2-1, to=2-2]
\end{tikzcd}\]
\end{thm}

By Koszul duality we obtain the following as a corollary to the above theorem. 

\begin{cor}\label{cor: Koszul dual functor} The following diagram commutes:
\[\begin{tikzcd}
	{\on{Whit}(\Gr_{G,\Ran})} & {\on{Whit}(\Gr_{M,\Ran})} \\
	{\on{Rep}(\check{G})_{\Ran}} & {\on{Rep}(\check{M})_{\Ran}.}
	\arrow["{\on{Jac}^M_!}", from=1-1, to=1-2]
	\arrow["{\on{CS}_G}"', from=1-1, to=2-1]
	\arrow["{\on{CS}_M}", from=1-2, to=2-2]
	\arrow["{C^{\bullet}(\check{\mathfrak{n}}_P,-)}"', from=2-1, to=2-2]
\end{tikzcd}\]
\end{cor}
\begin{proof}
    Since $\IC^{\frac{\infty}{2}}_{P,\Ran}$ has a coaction of $\mathcal{O}(\check{N}_P)_{\Ran}$, so does the functor $\on{Jac}^M_{!*}$.

    The functor $\on{Inv}_{\mathcal{O}(\check{N}_P)_{\Ran}}(\on{Jac}^M_{!*})$ given by post-composing with the functor of taking $\mathcal{O}(\check{N}_P)$-invariants is therefore given by the kernel $\on{Inv}_{\mathcal{O}(\check{N}_P)_{\Ran}}(\IC^{\frac{\infty}{2}}_{P,\Ran})$ under the pairing (\ref{eq:whitpairing}). Now the result follows from Theorem \ref{thm: restriction and !* Jacquet are the same} and Proposition \ref{p:restostrata!}, observing that we have an isomorphism of functors:
    \[\on{Inv}_{\mathcal{O}(\check{N}_P)_{\Ran}}\circ \on{Res}^{\check{G}}_{\check{M}} \simeq C^{\bullet}(\check{\mathfrak{n}}_P,-). \] 
\end{proof}

\subsubsection{} It remains to prove Theorem \ref{thm: restriction and !* Jacquet are the same}. We will reduce this theorem to a certain vanishing assertion when $P=B$ where we can invoke results of Raskin \cite{raskin2021chiral}.

Consider the commutative diagram:

\[\begin{tikzcd}[column sep=0.0001]
	&& {\ot{\Gr}_{P,B^-,B^-_M,\Ran}} \\
	& {S^{-,0}_{B,\Ran} \underset{\on{Gr}_{G,\on{Ran}}}{\times}\ot\Gr_{P,\Ran}} && {\ot\Gr_{P,\Ran} \underset{\on{Gr}_{M,\on{Ran}}}{\times} S^{-,0}_{B_M,\Ran}} \\
	{S^{-,0}_{B,\Ran}} && {\ot\Gr_{P,\Ran}} && {S^{-,0}_{B_M,\Ran}} \\
	& {\Gr_{G,\Ran}} && {\Gr_{M,\Ran}}
	\arrow[from=1-3, to=2-2]
	\arrow[from=1-3, to=2-4]
	\arrow[from=2-2, to=3-1]
	\arrow[from=2-2, to=3-3]
	\arrow[from=2-4, to=3-3]
	\arrow[from=2-4, to=3-5]
	\arrow[from=3-1, to=4-2]
	\arrow[from=3-3, to=4-2]
	\arrow[from=3-3, to=4-4]
	\arrow[from=3-5, to=4-4]
\end{tikzcd}\]

\noindent all of whose squares are Cartesian. For convenience, we have denoted the fiber product
\[S^{-,0}_{B,\Ran} \underset{\on{Gr}_{G,\on{Ran}}}{\times} \ot\Gr_{P,\Ran} \underset{\on{Gr}_{M,\on{Ran}}}{\times} S^{-,0}_{B_M,\Ran}\]

\noindent by $\ot{\Gr}_{P,B^-,B^-_M,\Ran}$.

\subsubsection{}
We have the vacuum Whittaker object $\psi_{G,\on{Ran}}$, which is cleanly extended to $\Gr_{G,\Ran}$ from $S^{-,0}_{B,\Ran}=\on{Gr}_{N^-,\on{Ran}}$. We would like to understand the sheaf
\begin{equation}\label{eq:calcsheaf}
\mathscr{F}_{\psi} \coloneqq \iota_{B_M}^!(\on{Jac}_{!*}^M(\psi_{G,\on{Ran}}))
\end{equation}

\noindent as an object of the category $\on{Whit}(S^{-,0}_{B_M,\Ran})$. Here $\iota_{B_M}$ denotes the inclusion $S^{-,0}_{B_M,\Ran}\to \on{Gr}_{M,\on{Ran}}$.

We claim that in the diagram above, we can replace $\ot\Gr_{P,\Ran}$ with the open locus
\[
\Gr_{P,\Ran}\subset \ot\Gr_{P,\Ran}.
\]

\noindent That is, we claim that the open embedding
\[\Gr_{P,B^-,B^-_M,\Ran} \coloneqq S^{-,0}_{B,\Ran} \underset{\on{Gr}_{G,\on{Ran}}}{\times} \Gr_{P,\Ran} \underset{\on{Gr}_{M,\on{Ran}}}{\times} S^{-,0}_{B_M,\Ran}\into \ot{\Gr}_{P,B^-,B^-_M,\Ran}\]

\noindent is an isomorphism. Indeed, that we are taking the fiber over $S^{-,0}_{B_M,\Ran}$ forces the $P$-reduction over the disc to be non-degenerate. In more detail, the data of an $S$-point of the fiber product
\[S^{-,0}_{B,\Ran} \underset{\on{Gr}_{G,\on{Ran}}}{\times} \ot\Gr_{P,\Ran}\]

\noindent is the data, for every irreducible representation $V$ of $G$, of a point $(x_I,\sP_G,\alpha)$ of $\Gr_{G,\Ran}$ and a point $(x_I,\sP_M, \sP_G)$ of $\ot{\Gr}_{P,\Ran}$ such that the corresponding meromorphic maps
\[V^{N_P}_{\sP_M} \longrightarrow V_{\sP_G} \longrightarrow \mathcal{O}_{X \times S} = V^{N_B^-}_{\sP^0_T}\]

\noindent are regular and the second map is non-degenerate. Note we are assuming $V$ is irreducible for the last equality. Requiring that the corresponding point of $\Gr_{M,\Ran}$ lands in $S^{-,0}_{B_M,\Ran}$ is the condition that the composition of the above maps is non-degenerate for every representation $V$. Whenever $V^{N_P}$ is one-dimensional, non-degeneracy of the composition implies non-degeneracy of the first map. Moreover, non-degeneracy of the generalized $P$-reduction can be checked using only representations $V$ such that $V^{N_P}$ is one-dimensional.

It follows that $\mathscr{F}_{\psi}$ can be computed by base-change along the above diagram, replacing $\ot\Gr_{P,\Ran}$ by $\Gr_{P,\Ran}$.

\begin{lem}\label{l: triple fiber product lemma}
    The composition \[\Gr_{P,B^-,B^-_M,\Ran} \to \Gr_{P,\Ran} \underset{\Gr_{M,\Ran}}{\times} S^{-,0}_{B_M,\Ran} \to S^{-,0}_{B_M,\Ran}\] is an isomorphism. 
\end{lem}
\begin{proof}
    The group $\mathfrak{L}_{\Ran}N_M^-$ acts on $\Gr_{P,B^-,B^-_M,\Ran}$ in such a way that the projection 
    \[\on{pr}_{B_M}:\Gr_{P,B^-,B^-_M,\Ran} \to S^{-,0}_{B_M,\Ran}\]

    \noindent is equivariant for the usual action of $\mathfrak{L}_{\Ran}N^-_M$ on $S^{-,0}_{B_M,\Ran}$. Since the latter action is transitive, it follows that $\on{pr}_{B_M}$ is an fppf locally trivial fiber bundle with typical fiber 
    \[\Gr_{P,B^-,B^-_M,\Ran} \underset{S^{-,0}_{B_M,\Ran}}{\times} \Ran \simeq S^{-,0}_{B,\Ran} \underset{\Gr_{G,\Ran}}{\times} \Gr_{P,\Ran} \underset{\Gr_{M,\Ran}}{\times} \Ran\]

    \noindent where the map $\Ran \to S^{-,0}_{B_M,\Ran}$ is the canonical section to the factorization morphism. 

    We have
    \[S^{-,0}_{B,\Ran} \underset{\Gr_{G,\Ran}}{\times} \Gr_{P,\Ran} \underset{\Gr_{M,\Ran}}{\times} \Ran \simeq S^{-,0}_{B,\Ran} \underset{\Gr_{G,\Ran}}{\times} S_{B,\on{Ran}}^0\simeq\on{Ran}, \]

    \noindent and the result follows. 
\end{proof}

\begin{cor}\label{c:lastcanequiv}
    There is a canonical isomorphism:
    \[\mathscr{F}_{\psi} \overset{\sim}{\longrightarrow} \iota_{B_M}^!(\psi_{M,\Ran})\]
\end{cor}
\begin{proof}
    This follows by diagram chase along
    \[\begin{tikzcd}[column sep=0.0001]
	&& {\Gr_{P,B^-,B^-_M,\Ran}} \\
	& {S^{-,0}_{B,\Ran} \underset{\on{Gr}_{G,\on{Ran}}}{\times}\Gr_{P,\Ran}} && {\Gr_{P,\Ran} \underset{\on{Gr}_{M,\on{Ran}}}{\times} S^{-,0}_{B_M,\Ran}} \\
	{S^{-,0}_{B,\Ran}} && {\Gr_{P,\Ran}} && {S^{-,0}_{B_M,\Ran}} \\
	& {\Gr_{G,\Ran}} && {\Gr_{M,\Ran}}
	\arrow[from=1-3, to=2-2]
	\arrow[from=1-3, to=2-4]
	\arrow[from=2-2, to=3-1]
	\arrow[from=2-2, to=3-3]
	\arrow[from=2-4, to=3-3]
	\arrow[from=2-4, to=3-5]
	\arrow[from=3-1, to=4-2]
	\arrow[from=3-3, to=4-2]
	\arrow[from=3-3, to=4-4]
	\arrow[from=3-5, to=4-4]
\end{tikzcd}\]

\noindent using base change, taking into consideration the equivalence
\[D(S^{-,0}_{B_M,\Ran}) \overset{\sim}{\longrightarrow} D(\Gr_{P,B^-,B_M^-,\Ran})\]

\noindent afforded by Lemma \ref{l: triple fiber product lemma}.
\end{proof}

\subsubsection{} Consider the counit map:
\begin{equation}\label{eq:counit map}
\psi_{M,\Ran}\simeq \iota_{B_M,!}\circ\iota_{B_M}^!(\on{Jac}_{!*}^M(\psi_{G,\Ran}))\to \on{Jac}_{!*}^M(\psi_{G,\Ran}).
\end{equation}

\noindent Here the first isomorphism comes from Corollary \ref{c:lastcanequiv}. In the remaining part of this subsection, we will prove:
\begin{thm}\label{t:counit map}
The counit map (\ref{eq:counit map}) is an isomorphism.
\end{thm}

We start by noting that the above theorem is sufficient for our purposes:

\begin{lem}\label{l:impimplication}
Suppose the counit map (\ref{eq:counit map}) is an isomorphism. Then the assertion of Theorem \ref{thm: restriction and !* Jacquet are the same} holds.
\end{lem}

\begin{proof}
    Recall that the Hecke structure on the semi-infinite IC sheaf $\IC^{\frac{\infty}{2}}_{P,\Ran}$ gives $\on{Jac}^M_{!*}$ the structure of a $\on{Rep}(\check{G})_{\Ran}$ linear functor, see §\ref{s:heckeandenhdrplofkernels}.

    By the Casselman-Shalika formula for $G$ and $M$, we have
    \begin{equation}\label{eq: hom equivalence cass sha}\Hom_{\on{Rep}(\check{G})_{\Ran}}(\on{Whit}(\Gr_{G,\Ran}), \on{Whit}(\Gr_{M,\Ran})) \simeq \Hom_{\on{Rep}(\check{G})_{\Ran}}(\on{Rep}(\check{G})_{\Ran}, \on{Rep}(\check{M})_{\Ran}),\end{equation}

\noindent while the right-hand side is nothing other than $\on{Rep}(\check{M})_{\Ran}$. It follows that any $\on{Rep}(\check{G})_{\on{Ran}}$-linear functor $\on{Whit}(\Gr_{G,\Ran})\to \on{Whit}(\Gr_{M,\Ran})$ is determined by where it sends $\psi_{G,\on{Ran}}$. Thus, we may check that $\on{Jac}_{!*}^M$ and $\on{Res}^{\check{G}}_{\check{M}}$ coincide by checking that they both send $\psi_{G,\on{Ran}}$ to $\psi_{M,\on{Ran}}$.
\end{proof}

\subsubsection{} Hence we need to prove Theorem \ref{t:counit map}. To do this, we will use Theorem \ref{thm: composing eisenstein} to reduce to the principal case; i.e., when $P=B$. As such, let us assume that Theorem \ref{t:counit map} holds for an arbitrary reductive group (e.g., $G$ or $M$) whenever the corresponding Levi is a torus. Under this assumption, let us prove Theorem \ref{t:counit map} for an arbitrary Levi.

We write $\on{Jac}_{!*}^{T_M}: \on{Whit}(\on{Gr}_{M,\on{Ran}})\to \on{Whit}(\on{Gr}_{T,\on{Ran}})$ for the $!*$-Jacquet functor from the Levi $M$ to the maximal torus $T$. It suffices to check that (\ref{eq:counit map}) is an isomorphism after applying the functor $\on{Jac}_{!*}^{T_M}$. Indeed, having assumed Theorem \ref{t:counit map} is true in the principal case, Lemma \ref{l:impimplication} implies that Theorem \ref{thm: restriction and !* Jacquet are the same} is true in the principal case. This in turn implies that $\on{Jac}^{T_M}_{!*}$ is conservative.

However, by Theorem \ref{thm: composing eisenstein} we have: 
\[\on{Jac}^{T_M}_{!*}(\on{Jac}^M_{!*}(\psi_{G,\Ran})) \simeq \on{Jac}^T_{!*}(\psi_{G,\Ran}) \simeq \delta_{\on{Gr}_{T,\Ran}}\simeq \on{Jac}^{T_M}_{!*}(\psi_{M,\on{Ran}}).\]

\noindent Hence we are reduced to showing that Theorem \ref{t:counit map} holds in the principal case. 

\subsubsection{} We will produce two proofs in the principal case. The first argument uses a result of Raskin from \cite{raskin2021chiral}, and the second uses techniques from \cite{gaitsgory2019smallfle}. Since the first argument will be rigorous, we will allow ourselves to only give a sketch of the second. 

The reason we will not provide full details for the second argument is that in \cite{gaitsgory2019smallfle} the authors work in a twisted setting with a non-degeneracy assumption on the twisting. As such, we cannot cite their results directly. However, one can show that their arguments can be adapted to the non-twisted setting and that the main difficulty in doing so is notational. Although we allow ourselves to forego complete rigor, we still believe an indication of the second proof is valuable in that it avoids the use of perverse t-structures.

\begin{proof}[First proof of Theorem \ref{t:counit map} in the principal case] We need to show that for any $k$-point 
\[ a: \on{pt} \longrightarrow \Gr_{T,\Ran}\]

\noindent that does not lie in the image of the unit section 
\[\Ran \longrightarrow \Gr_{T,\Ran},\]

\noindent the vector space $a^!(\on{Jac}^T_{!*}(\psi_{G,\Ran}))$ vanishes. Since we are in the principal case, we write $\overline{\on{Gr}}_{B,\on{Ran}}$ instead of $\widetilde{\on{Gr}}_{B,\on{Ran}}$.

Consider the diagram 
\[\begin{tikzcd}
	& {\overline{\Gr}_{B,x}} \\
	{\Gr_{G,x}} && {\Gr_{T,x}} \\
	& {\overline{\Gr}_{B,\Ran}} \\
	{\Gr_{G,\Ran}} && {\Gr_{T,\Ran},}
	\arrow["{\overline{p}_{B,x}}"', from=1-2, to=2-1]
	\arrow["{\overline{q}_{B,x}}", from=1-2, to=2-3]
	\arrow["{i_{B,x}}"', from=1-2, to=3-2]
	\arrow["{i_{G,x}}"', from=2-1, to=4-1]
	\arrow["{i_{T,x}}", from=2-3, to=4-3]
	\arrow["{\overline{p}_{B,\Ran}}", from=3-2, to=4-1]
	\arrow["{\overline{q}_{B,\Ran}}"', from=3-2, to=4-3]
\end{tikzcd}\]

\noindent where $x$ is the composition of $a$ with the projection $\Gr_{T,\Ran} \to \Ran$. By factorization, we may assume that $x$ factors through the image of the main diagonal $X \to \Ran$. Base change along this diagram shows that the functor $i^!_{T,x}\circ \on{Jac}^T_{!*}: \on{Whit}(\on{Gr}_{G,\on{Ran}})\to D(\on{Gr}_{T,x})$ is given by:
\[\cF \longmapsto \overline{q}_{B,x,*}(i^!_{B,x}({_{\Gr_{T,\Ran}}}{\IC^{\frac{\infty}{2}}_{B,\Ran}}) \overset{!}{\otimes} \overline{p}^!_{B,x}(i^!_{G,x}(\cF))).\]

\noindent Here, we have denoted by ${_{\Gr_{T,\Ran}}}{\IC^{\frac{\infty}{2}}_{B,\Ran}}$ the $!$-pullback of $\on{IC}_{B,\on{Ran}}^{\frac{\infty}{2}}$ along $\overline{\on{Gr}}_{B,\on{Ran}}\to \fL^+_{\on{Ran}}T\backslash \widetilde{S}_{B,\on{Ran}}^0$.

Applying this to $\cF = \psi_{G,\Ran}$ gives 
\[i^!_{T,x}\circ \on{Jac}^T_{!*}(\psi_{G,\on{Ran}})\simeq \overline{q}_{B,x,*}(i^!_{B,x}({_{\Gr_{T,\Ran}}}{\IC^{\frac{\infty}{2}}_{B,\Ran}}) \overset{!}{\otimes} \overline{p}^!_{B,x}(\psi_{G,x})),\]

\noindent and we want to show that the right hand side is supported only in the $0$'th component of $\Gr_{T,x}$. 

Let us now further base change along the inclusion of a connected component 
\[t^{\lambda} \longrightarrow \Gr_{T,x}\]

\noindent with $\lambda \neq 0$, where we recall that $\Gr_{T,x}$ is (on the level of reduced schemes) the disjoint union of $t^{\lambda}$'s, with $\lambda \in \Lambda$. 

The fiber of $\overline{\Gr}_{B,x}$ over $t^{\lambda}$ is the closure $\overline{S}^{\lambda}_x$ of the semi-infinite orbit $S^{\lambda}_x$ in $\Gr_{G,x}$ indexed by $\lambda$. Unwinding the definitions, we have note the $!$-pullback of ${_{\Gr_{T,\Ran}}}{\IC^{\frac{\infty}{2}}_{B,\Ran}}$ to $\overline{S}^{\lambda}_x$ is equivalent to the translation $t^{\lambda}\IC^{\frac{\infty}{2}}_{B,x}$ of $\IC^{\frac{\infty}{2}}_{B,x}=i_{G,x}^!(\on{IC}_{B,\on{Ran}}^{\frac{\infty}{2}})$ by the automorphism induced by multiplying by $t^{\lambda}$. We therefore want to show that
\begin{equation}\label{eq: cohomology that vanishes}H^*_{dR}(\overline{S}^{\lambda}_x \cap S^{-,0}_x, t^{\lambda}\IC^{\frac{\infty}{2}}_{B,x} \overset{!}{\otimes} \psi_{G,x})= 0\end{equation} 

\noindent for each $\lambda \in \Lambda\setminus 0$. 

If $\overline{S}^{\lambda}_x \cap S^{-,0}_x$ is empty there is nothing to show. Otherwise, it is a closed subscheme of a finite-dimensional scheme $\mathcal{Z}$ called the \emph{Zastava space}. By \cite[Prop. 3.8.3]{gaitsgory2021semi}, the semi-infinite IC sheaf restricted to the intersection $\overline{S}^{\lambda}_x \cap S^{-,0}_x$ is equivalent to the perverse intersection cohomology sheaf of $\mathcal{Z}$ restricted to the same intersection, up to a cohomological shift.\footnote{We remark that given Theorem \ref{t:IC-to-hecke}, the proof of Proposition 3.8.3 in \cite{gaitsgory2021semi} reduces to a relative version of \cite[Prop. 3.6.5(a)]{gaitsgory2018semi}, which compares the IC-sheaf on $\overline{\on{Bun}}_N$ and the Zastava space.} Hence the vanishing of \eqref{eq: cohomology that vanishes} is exactly Theorem $3.4.1$ of \cite{raskin2021chiral}.
\end{proof}

\begin{proof}[Second proof of vanishing in the principal case] The Hecke structure on $\IC^{\frac{\infty}{2}}_{B,\Ran}$ gives a canonical commutative triangle:
\[\begin{tikzcd}
	{\on{Whit}(\Gr_{G,\Ran})} & {\on{Whit}(\Gr_{G,\Ran}) \underset{\on{Rep}(\check{G})_{\Ran}}{\otimes} \on{Rep}(\check{T})_{\Ran}} \\
	& {D(\Gr_{T,\Ran}).}
	\arrow["{\on{Res}^{\check{G}}_{\check{T}}}", from=1-1, to=1-2]
	\arrow["{\on{Jac}^T_{!*}}"', from=1-1, to=2-2]
	\arrow["{\on{Jac}^{T,\on{Hecke}}_{!*}}", from=1-2, to=2-2]
\end{tikzcd}\]

By construction, under geometric Casselman-Shalika, the functor $\Res^{\check{G}}_{\check{T}}$ corresponds to the same named functor $\on{Rep}(\check{G})_{\Ran} \to \on{Rep}(\check{T})_{\Ran}$. Hence it suffices to show that taking the fiber of $\on{Jac}^{T,\on{Hecke}}_{!*}$ at a point $x \in X$ coincides with the identity functor 
\[\on{id}_{\on{Rep}(\check{T})}: \on{Rep}(\check{T}) \longrightarrow \on{Rep}(\check{T}).\]

Note the functor $\on{Jac}^{T,\on{Hecke}}_{!*,x}$ given by $\on{Jac}^{T,\on{Hecke}}_{!*}$ followed by taking the fiber at $x$ is given by the kernel $\IC^{\frac{\infty}{2}}_{B,x}\in D(\on{Gr}_{G,x})$.

We will carry out the following construction. For each $\lambda \in \Lambda$, we will produce an object $\mathcal{M}^{\lambda}$ in $\on{Whit}(\Gr_{G,x}) \underset{\on{Rep}(\check{G})}{\otimes} \on{Rep}(\check{T})$. The collection of $\mathcal{M}^{\lambda}$'s will satisfy the following:
\begin{enumerate}
    \item Each $\mathcal{M}^{\lambda}$ belongs to the heart of the natural t-structure of the category
    \[\on{Whit}(\Gr_{G,x}) \underset{\on{Rep}(\check{G})}{\otimes} \on{Rep}(\check{T}) \simeq \on{Rep}(\check{T}).\]
    \item For each $\lambda$, there is a surjective map 
    \[\mathcal{M}^{\lambda} \longrightarrow k^{\lambda},\]
    
    \noindent where $k^{\lambda}$ is the $1$-dimensional representation of $\check{T}$ given by $\lambda$. 
    \item The functor $\Phi^{\lambda}$ given by $\on{Jac}^{T,\on{Hecke}}_{!*,x}$ followed by taking the $\lambda$-component in $\on{Rep}(\check{T})$ is corepresented by $\mathcal{M}^{\lambda}$. 
    \item We have that
    \[\Hom_{\on{Rep}(\check{T})}(\mathcal{M}^{\lambda},\mathcal{M}^{\mu}) = k\]
    
    \noindent if $\mu = \lambda$, and is $0$ otherwise. 
\end{enumerate}

\noindent Assuming the existence of the $\mathcal{M}^{\lambda}$'s, let us finish the proof. First, $(1)$ and $(4)$ together imply that $\mathcal{M}^{\lambda}$ is irreducible, and hence that the map $\mathcal{M}^{\lambda} \to k^{\lambda}$ from $(2)$ is an isomorphism. By $(3)$, for every 
\[V \in \on{Rep}(\check{T}),\]

\noindent we therefore have 
\[\Phi^{\lambda}(V) = \Hom_{\on{Rep}(\check{T})}(\mathcal{M}^{\lambda},V) \simeq \Hom_{\on{Rep}(\check{T})}(k^{\lambda},V) = V(\lambda),\]

\noindent where $V(\lambda)$ is the $\lambda$-component of the $\check{T}$-representation $V$. But 
\[\on{Jac}^{T,\on{Hecke}}_{!*,x} = \bigoplus_{\lambda \in \Lambda} \Phi^{\lambda},\]

\noindent from which we conclude that $\on{Jac}^{T,\on{Hecke}}_{!*,x}$ is the identity. 

Let us now define the $\mathcal{M}^{\lambda}$'s. First, define 
\[\mathcal{M}^0 \coloneqq \Av_*^{\mathfrak{L}_x N^-,\psi}(\on{gr}\IC^{\frac{\infty}{2}}_{B,x}).\]

\noindent Here, $\on{gr}\IC^{\frac{\infty}{2}}_{B,x}$ is the graded Hecke object corresponding to $\IC^{\frac{\infty}{2}}_{B,x}$. That is, 
\[
\on{gr}\IC^{\frac{\infty}{2}}_{B,x}\in D(\on{Gr}_{G,x})\underset{\on{Rep}(\check{G})}{\otimes} \on{Rep}(\check{T})
\]

\noindent is the image of $\on{IC}^{\frac{\infty}{2}}_{B,x}$ under the functor
\[
\on{Hecke}_{\check{M},\check{G}}(\on{SI}_{P,x})=\on{SI}_{P,x}\underset{\on{Rep}(\check{T})\otimes \on{Rep}(\check{G})}{\otimes} \on{Rep}(\check{T})\to \on{SI}_{P,x}\underset{\on{Rep}(\check{G})}{\otimes} \on{Rep}(\check{T})\to D(\on{Gr}_{G,x})\underset{\on{Rep}(\check{G})}{\otimes} \on{Rep}(\check{T}).
\]

\noindent Here, the first functor is right adjoint to the projection:
\[
\on{SI}_{P,x}\underset{\on{Rep}(\check{G})}{\otimes} \on{Rep}(\check{T})\to \on{SI}_{P,x}\underset{\on{Rep}(\check{T})\otimes \on{Rep}(\check{G})}{\otimes} \on{Rep}(\check{T}).
\]

\noindent The second functor is induced by the forgetful functor $\on{SI}_{P,x}\to D(\on{Gr}_{G,x})$.

For an arbitrary $\lambda\in \Lambda$, let $\mathcal{M}^{\lambda} \coloneqq \mathcal{M}^0 \otimes k^{\lambda}$.

The properties $(1)-(4)$ can be verified following \cite{gaitsgory2019smallfle}. To direct the reader to the relevant arguments: Property $(1)$ is proved in the same way as Corollary $25.2.3$ in \emph{loc.cit.}, Property $(2)$ is Corollary $25.4.4$, Property $(3)$ is the content of Section $26$, and Property $(4)$ is Theorem $25.5.2$.
\end{proof}

\newpage




\appendix

\section{Colimit description of intersection cohomology sheaf}\label{S:APPB}

The goal of this Appendix is to prove lemma \ref{l:ICisSI} and lemma \ref{l:rescomp}. To do so, we will realize (the fibers of) $\sIC$ as a suitable colimit, following \cite{gaitsgory2018semi}.

\subsection{Proof of Lemma \ref{l:ICisSI}}
\subsubsection{}

Let $i_x: x\to X\to \on{Ran}$ be a $k$-point of $X$. Let $\on{IC}_{P,x}^{\frac{\infty}{2}}:=i_x^!(\sIC)$. We have
\[
\on{IC}_{P,x}^{\frac{\infty}{2}}=\on{Ind}_{\on{DrPl}_{\check{M},\check{G}}}^{\Hecke_{\check{M},\check{G}}}(\delta_{\on{Gr}_{G,x}}),
\]

\noindent where we now consider $D(\mathfrak{L}^+_{x}M\backslash \on{Gr}_{G,x})$ as acted on by $\on{Rep}(\check{M})\otimes \on{Rep}(\check{G})$,\footnote{We remind that the action of $\on{Rep}(\check{M})$ is always twisted by the functor (\ref{eq:shift2}).} and with $\delta_{\on{Gr}_{G,x}}$ defining an object of $\on{DrPl}(D(\mathfrak{L}^+_{x}M\backslash \on{Gr}_{G,x}))$. That is:
\[
\on{IC}_{P,x}^{\frac{\infty}{2}}\simeq \on{Fun}(\check{N}_P\backslash \check{G})\underset{\on{Fun}(\overline{\check{N}_P\backslash \check{G})}}{\otimes}\delta_{\on{Gr}_{G,x}}.
\]

\subsubsection{}

Let $\Lambda_M^+\subset \Lambda$ be the set of dominant cocharacters that are orthogonal to the roots of $M$. That is, $\lambda\in \Lambda$ lies in $\Lambda_M^+$ if and only if it is dominant and saitisfies $\langle \alpha_i,\lambda\rangle=0$ for all $\alpha_i\in \sJ_M$.

Note that such a $\lambda$ defines a character of $\check{M}$, and so the corresponding highest weight representation $V_M^{\lambda}$ of $M$ is $1$-dimensional. In this case, we write $e^{\lambda}:=V_M^{\lambda}$. We continue to write $V^{\lambda}$ for the highest weight representation of $\check{G}$ corresponding to $\lambda$.

\subsubsection{} We consider $\Lambda_M^+$ as a poset via the relation $\lambda_1\leq \lambda_2\Leftrightarrow \lambda_2-\lambda_1\in \Lambda_M^+$. We note that this poset is filtered due to the assumption that $G$ is semisimple and simply connected.

\subsubsection{} Let $\sC$ be a category acted on by $\on{Rep}(\check{M})\otimes \on{Rep}(\check{G})$, and let $c\in \on{DrPl}(\sC)$ be an object of $c$ equipped with a Drinfeld-Plücker structure. We remind (cf. §\ref{s:whatdrplactuallymeans}) that these correspond in particular to a coherent family of maps
\begin{equation}\label{eq:restricteddrpl}
e^{\lambda}\star c\to c\star V^{\lambda},\;\; \lambda\in \Lambda_M^+.
\end{equation}

\noindent In this case, we get a natural functor $\Lambda_M^+\to \sC$ given by
\[
\lambda\mapsto e^{-\lambda}\star c\star V^{\lambda},
\]

\noindent and for $\lambda_2=\lambda_1+\lambda$, the corresponding transition map is:
\[
e^{-\lambda_1}\star c\star V^{\lambda_1}\to e^{-\lambda_1}\star e^{-\lambda} \star c\star V^{\lambda}\star V^{\lambda_1}\to e^{-\lambda_2}\star c\star V^{\lambda_2}.
\]

\noindent Here, the first map is given by $c\to e^{-\lambda}\star c\star V^{\lambda}$ induced by (\ref{eq:restricteddrpl}), and the last map is induced by the Plücker map $V^{\lambda}\otimes V^{\lambda_1}\to V^{\lambda+\lambda_1}$

As such, we may form the corresponding colimit:
\[
\underset{\lambda\in \Lambda_M^+}{\on{colim}}\; e^{-\lambda}\star c\star V^{\lambda}.
\]

\subsubsection{} We have the following generalization of \cite[Prop. 6.2.4]{gaitsgory2018semi} to the parabolic setting. The proof follows that of \emph{loc.cit}.

\begin{lem}\label{l:colimdescrip}
The functor
\[
\on{DrPl}_{\check{M},\check{G}}(\sC)\xrightarrow{\on{Ind}_{\on{DrPl}_{\check{M},\check{G}}}^{\Hecke_{\check{M},\check{G}}}} \on{Hecke}_{\check{M},\check{G}}(\sC)\xrightarrow{\on{oblv}} \sC
\]

\noindent coincides with the functor
\[
c\mapsto \underset{\lambda\in \Lambda_M^+}{\on{colim}}\; e^{-\lambda}\star c\star V^{\lambda}.
\]
\end{lem}

\begin{proof}

It suffices to consider the universal case where $\sC=\on{Rep}(\check{M})\otimes \on{Rep}(\check{G})$ and $c=\on{Fun}(\overline{\check{N}_P\backslash \check{G}})$.

We have a natural map of $\check{M}\times\check{G}$-representations:
\[
\underset{\lambda\in \Lambda_M^+}{\on{colim}}\; e^{-\lambda}\otimes \on{Fun}(\overline{\check{N}_P\backslash \check{G}})\otimes V^{\lambda}\to \on{Fun}(\check{G})
\]

\noindent induced by the algebra map $\on{Fun}(\overline{\check{N}_P\backslash \check{G}})\to \on{Fun}(\check{G})$ and then multiplying by the function $g\mapsto \langle \phi, g\cdot v\rangle$. Here, $\phi\in (V^{\lambda})^*$ is the lowest weight vector $e^{-\lambda}\to (V^{\lambda})^*$ and $v\in V^{\lambda}$.

We need to verify that this map is an isomorphism. Since $\Lambda_M^+$ is filtered, the colimit lies in the heart of the $t$-structure of $\on{Rep}(\check{M})\otimes \on{Rep}(\check{G})$. As such, it suffices to show that the map is an isomorphism after applying the functor $\on{Hom}_G(V^{\mu}, -)$ for each dominant coweight $\mu$.

On the one hand, we have $\on{Hom}_{\check{G}}(V^{\mu},\on{Fun}(\check{G}))\simeq (V^{\mu})^*$.

On the other hand, suppose $\lambda\in \Lambda_M^+$ is sufficiently large compared to $\mu$. That is, $\lambda+\nu$ is dominant for all $M$-dominant coweights $\nu$ appearing as highest weights in $(V^{\mu})^*$ considered as a $\check{M}$-representation. For such an $M$-dominant coweight, we write $V_{\check{M}}^{\nu}$ for the corresponding highest weight representation of $\check{M}$.

In this case, we have:
\[
\on{Hom}_{\check{G}}(V^{\mu}, e^{-\lambda}\otimes \on{Fun}(\overline{\check{N}_P\backslash \check{G}})\otimes V^{\lambda})\simeq e^{-\lambda}\otimes \on{Hom}_{\check{G}}(k, \on{Fun}(\overline{\check{N}_P\backslash \check{G}})\otimes (V^{\mu})^*\otimes V^{\lambda}).
\]

\noindent By Lemma \ref{l:reptheoryl} below, we may further rewrite this as:
\[
= \underset{\nu}{\bigoplus}\; e^{-\lambda} \otimes \on{Hom}_{\check{G}}(k, \on{Fun}(\overline{\check{N}_P\backslash \check{G}})\otimes V^{\lambda+\nu}\otimes \on{Hom}_{\check{M}}(V_{\check{M}}^{\nu}, (V^{\mu})^*))
\]
\[
\simeq \underset{\nu}{\bigoplus}\; e^{-\lambda} \otimes \on{Hom}_{\check{M}}(V_{\check{M}}^{\nu}, (V^{\mu})^*)\otimes \on{Hom}_{\check{G}}((V^{\lambda+\nu})^*, \on{Fun}(\overline{\check{N}_P\backslash \check{G}}))
\]
\[
\simeq \underset{\nu}{\bigoplus}\; e^{-\lambda} \otimes \on{Hom}_{\check{M}}(V_{\check{M}}^{\nu}, (V^{\mu})^*)\otimes V_{\check{M}}^{\lambda+\nu}
\]
\[
\simeq \underset{\nu}{\bigoplus}\; \on{Hom}_{\check{M}}(V_{\check{M}}^{\nu}, (V^{\mu})^*)\otimes V_{\check{M}}^{\nu}\simeq (V^{\mu})^*.
\]

\noindent By construction, the induced map
\[
(V^{\mu})^*\simeq \on{Hom}_{\check{G}}(V^{\mu}, e^{-\lambda}\otimes \on{Fun}(\overline{\check{N}_P\backslash \check{G}})\otimes V^{\lambda})\to \on{Hom}_{\check{G}}(V^{\mu}, \on{Fun}(\check{G}))\simeq (V^{\mu})^*
\]

\noindent is the identity, as required.
\end{proof}

\begin{lem}\label{l:reptheoryl}
Let $\lambda\in \Lambda_M^+$ and let $V$ be a $\check{G}$-representation. Suppose that for all $M$-dominant coweights $\nu$ that appear as highest weights in $\on{Res}^{\check{G}}_{\check{M}}(V)$, the coweight $\lambda+\nu$ is dominant (for $\check{G}$). Then we have a canonical decomposition
\[
V\otimes V^{\lambda}\simeq \underset{\nu}{\bigoplus}\; V^{\lambda+\nu}\otimes \on{Hom}_{\check{M}}(V_{\check{M}}^{\nu}, V)
\]

\noindent as $\check{G}$-representations. Here, the direct sum is over all $M$-dominant coweights $\nu$.
\end{lem}

\begin{proof}
Denote by
\[
\on{Res}_{\check{P}}^{\check{G}}: \on{Rep}(\check{G})\to \on{Rep}(\check{P})
\]

\noindent the restriction functor, and denote by 
\[
\on{coInd}_{\check{P}}^{\check{G}}: \on{Rep}(\check{P})\to \on{Rep}(\check{G})
\]

\noindent its right adjoint given by coinduction. Then we have:
\[
V\otimes V^{\lambda}\simeq V\otimes \on{coInd}_{\check{P}}^{\check{G}}(e^{\lambda})\simeq \on{coInd}_{\check{P}}^{\check{G}}(\on{Res}_{\check{P}}^{\check{G}}(V)\otimes e^{\lambda})
\]
\[
\simeq \underset{\nu}{\bigoplus}\; \on{coInd}_{\check{P}}^{\check{G}}(V_{\check{M}}^{\nu}\otimes e^{\lambda})\otimes \on{Hom}_{\check{M}}(V_{\check{M}}^{\nu}, V)
\]
\[
\simeq \underset{\nu}{\bigoplus}\; \on{coInd}_{\check{P}}^{\check{G}}(V_{\check{M}}^{\lambda+\nu})\otimes \on{Hom}_{\check{M}}(V_{\check{M}}^{\nu}, V)
\]
\[
\simeq \underset{\nu}{\bigoplus}\; V^{\lambda+\nu}\otimes \on{Hom}_{\check{M}}(V_{\check{M}}^{\nu}, V).
\]
\end{proof}

\subsubsection{} We proceed with the proof of Lemma \ref{l:ICisSI}.
Recall the semi-infinite IC sheaf:
\[
\sIC\in D(\mathfrak{L}^+_{\on{Ran}}M\backslash \on{Gr}_{G,\on{Ran}}).
\]

\noindent First, we need to verify that $\sIC$ is equivariant for the group $\mathfrak{L}_{\on{Ran}}N_P$. That is, we need to verify that $\sIC$ lies in the subcategory
\[
\on{SI}_{M,\on{Ran}}=D(\mathfrak{L}_{\on{Ran}}N_P\mathfrak{L}^+_{\on{Ran}}M\backslash \on{Gr}_{G,\on{Ran}})\subset D(\mathfrak{L}^+_{\on{Ran}}M\backslash \on{Gr}_{G,\on{Ran}}).
\]

\subsubsection{} For a category $\sC$ and a field-extension $k\subset K$, we write $\sC_K:= \sC\underset{\on{Vect}_k}{\otimes} \on{Vect}_K$ for its base change to $K$.

The following lemma allows us to check equivariance on field-valued points:
\begin{lem}\label{l:redtopts}
Let $\sC, \sD$ be $D(\on{Ran})$-module categories and suppose we are given a fully faithful embedding $\sD\into \sC$. Then an object $c\in \sC$ lies in $\sD$ if and only if for all finite sets $I$ and field-valued points $j:\on{Spec}(K)\to X^I$, the image of $c$ under the functor:
\[
\sC\xrightarrow{-\underset{k}{\otimes}K}\sC_K\simeq \sC_K\underset{D(\on{Ran})_K}{\otimes} D(\on{Ran})_K\xrightarrow{\on{id}\otimes j^!} \sC_K\underset{D(\on{Ran})_K}{\otimes} \on{Vect}_K
\]

\noindent lies in $\sD_K\underset{D(\on{Ran})_K}{\otimes} \on{Vect}_K$.
\end{lem}

\begin{proof}
It suffices to check that $c$ lies in $\sD$ when restricted to $X^I\to \on{Ran}$ for each finite set $I$. By a standard Cousin complex argument (see e.g. \cite[Lemma 9.2.8]{arinkin2020stack}), we may write the dualizing sheaf $\omega_{X^I}$ as a colimit of sheaves of the form $j_*(K)$, where $j: \on{Spec}(K)\to X_{\on{dR}}^I$ is a field-valued point. Since $\omega_{X^I}$ acts as the identity on $\sC\underset{D(\on{Ran})}{\otimes} D(X^I)$, any object in the latter category admits a similar colimit description.
\end{proof}

\begin{lem}\label{l:ICisSI0}
The sheaf $\sIC$ is $\mathfrak{L}_{\on{Ran}}N_P$-equivariant.
\end{lem}

\begin{proof}

\step
By Lemma \ref{l:redtopts}, it suffices to check equivariance when restricting to each field-valued point $\on{Spec}(K)\to X^I$. For ease of notation, we assume that $K=k$. Since $\sIC$ is a factorization algebra, we may further assume that the point $\on{Spec}(k)\to X^I$ factors through $X$. In other words, it suffices to show that $\on{IC}_{P,x}^{\frac{\infty}{2}}\in D(\mathfrak{L}^+_{x}M\backslash \on{Gr}_{G,x})$ is $\mathfrak{L}_{x}N_P$-equivariant.

\step 
Since $N_P$ is unipotent, the loop group $\mathfrak{L}_{x}N_P$ is an ind-group scheme. That is, we may write $\mathfrak{L}_{x}N_P$ as a filtered colimit of group schemes $\lbrace N_{\alpha}\rbrace _{\alpha\in A}$, where $A$ is a filtered category. As such, it suffices to establish equivariance against each $N_{\alpha}$.

The subset $\Lambda_{M,\alpha}^+\subset \Lambda_M^+$ of those $\lambda$ such that $\on{Ad}_{t^{\lambda}}(N_{\alpha})\subset N(O_x)$ is cofinal. For such a $\lambda$ and $\cF\in D(\mathfrak{L}^+_{x}N_P\mathfrak{L}^+_{x}M\backslash \on{Gr}_{G,x})$, note that $e^{-\lambda}\star \cF$ is $N_{\alpha}$-equivariant. This shows that the colimit
\[
\underset{\lambda\in \Lambda_M^+}{\on{colim}}\; e^{-\lambda}\star \delta_{\on{Gr}_{G,x}}\star V^{\lambda}
\]

\noindent is $N_{\alpha}$-equivariant. By Lemma \ref{l:colimdescrip}, the above colimit identifies with $\on{IC}_{P,x}^{\frac{\infty}{2}}$.
\end{proof}

\subsubsection{} The above lemma together with the following finish the proof of Lemma \ref{l:ICisSI}:
\begin{lem}\label{l:ICsuppd}
The sheaf $\sIC$ is supported on $\tildePloc$.
\end{lem}

\begin{proof}
As in Step 1 of the proof of Lemma \ref{l:ICisSI0}, we may check this after restricting to points $x\in X(k)$. Let $\theta\in \Lambda_{G,P}\setminus \Lambda_{G,P}^{\on{neg}}$ and choose a lift $\theta'$ of $\theta$ to $\Lambda_{G,P}$. Let $S_x^{\theta'}$ be the corresponding semi-infinite orbit (i.e., the $\mathfrak{L}_{x}N$-orbit through $t^{\theta'}$). Since $S_x^{\theta'}$ contains $S^{\theta}_{P,x}$, it suffices to show that the restriction of $\on{IC}_{P,x}^{\frac{\infty}{2}}$ to $S_x^{\theta'}$ is zero. However this follows from the colimit description of $\on{IC}_{P,x}^{\frac{\infty}{2}}$ and the fact that if $S_x^{\theta'}\cap\overline{\on{Gr}^{\lambda}}_{G,x}\neq\emptyset$, then $\lambda-\theta'$ is a sum of positive simple roots with non-negative integral coefficients.
\end{proof}

\subsection{Proof of Lemma \ref{l:rescomp}}\label{s:APPB2}

\subsubsection{}
Recall the notation of \S \ref{s:!rescomp}. We need need to show that the sheaf $\cF^{\theta}$ lies in perverse degrees $\geq 1+\langle 2(\rho_G-\rho_M),\theta\rangle$ whenever $\theta\neq 0$.

Since the sheaf $\sIC$ factorizes over $\on{Ran}$, the sheaf $\cF=\underset{\theta\in\Lambda_{G,P}^{\on{neg}}}{\bigoplus}\; \cF^{\theta}$ factorizes over $\on{Conf}_{G,P}$. As such, it suffices to show that $\Delta^!(\cF^{\theta})\in D(\on{Gr}_{M,\on{Conf}}^{+,\theta}\underset{X^{\theta}}{\times} X)$ lies in perverse degrees $\geq 1+\langle 2(\rho_G-\rho_M),\theta\rangle$, where
\[
\Delta: \on{Gr}_{M,\on{Conf}}^{+,\theta}\underset{X^{\theta}}{\times} X\to \on{Gr}_{M,\on{Conf}}^{+,\theta}
\]

\noindent is induced by the diagonal map:
\[
\Delta: X\to X^{\theta}\;\; x\mapsto \theta\cdot x.
\]

\begin{lem}\label{l:ULA}
$\Delta^!(\cF^{\theta})$ is ULA over $X$.
\end{lem}

\begin{proof}
By working étale-locally on $X$, we may assume that $X=\mathbb{A}^1$. In this case, we have:
\[
\on{Gr}_{G,X}\simeq \on{Gr}_G\times X,\;\; \on{Gr}_{M,X}\simeq \on{Gr}_M\times X,\;\; S^{\theta}_{P,X}\simeq S^{\theta}_P\times X.
\]

\noindent Moreover by construction, the restriction of the factorization algebra $\cO(\overline{\check{N}_P\backslash \check{G}})$ (resp. $\cO(\check{G})$) along
\[
(\on{Rep}(\check{M})\otimes \on{Rep}(\check{G}))_{\on{Ran}}\to (\on{Rep}(\check{M})\otimes \on{Rep}(\check{G}))_{\on{Ran}}\underset{D(\on{Ran})}{\otimes} D(X)\simeq \on{Rep}(\check{M})\otimes \on{Rep}(\check{G})\otimes D(X)
\]

\noindent is isomorphic to $\on{Fun}(\overline{\check{N}_P\backslash \check{G}})\boxtimes \omega_X$ (resp. $\on{Fun}(\check{G})\boxtimes \omega_X$). As such, the restriction of $\sIC$ to $\on{Gr}_{G,X}\simeq \on{Gr}_G\times X$ is isomorphic to $(\on{Fun}(\check{G})\underset{\on{Fun}(\overline{\check{N}_P\backslash \check{G}})}{\otimes} \delta_{\on{Gr}_G})\boxtimes \omega_X$. Since $\on{Gr}^+_{M,X}$ maps isomorphically onto $\on{Gr}_{M,\on{Conf}}^{+}\underset{X^{\theta}}{\times} X$, the result follows.
\end{proof}

\begin{proof}[Proof of Lemma \ref{l:rescomp}]
Let $x\in X$. Recall that $\on{IC}_{P,x}^{\frac{\infty}{2}}$ denotes the $!$-restriction of $\sIC$ to $x\in\on{Ran}$.

Let $i^{\theta}_x: \on{Gr}_{M,x}^{+,\theta}\to \on{Gr}_{G,x}$ be the natural embedding. By factorization of $\cF^{\theta}$ and Lemma \ref{l:ULA}, it suffices to show that the $!$-restriction of $\cF^{\theta}$ under
\[
\on{Gr}_{M,x}^{+,\theta}\simeq\on{Gr}_{M,\on{Conf}}^{+,\theta}\underset{X^{\theta}}{\times} \lbrace x\rbrace \to \on{Gr}_{M,\on{Conf}}^{+,\theta}
\]

\noindent lies in perverse degrees $\geq 2+\langle 2(\rho_G-\rho_M),\theta\rangle$. However, this restriction coincides with $i^{\theta,!}_x(\on{IC}_{P,x}^{\frac{\infty}{2}})$. Recall the isomorphism (see Lemma \ref{l:colimdescrip}):
\[
\on{IC}_{P,x}^{\frac{\infty}{2}}\simeq\underset{\lambda\in \Lambda_M^+}{\on{colim}}\; e^{-\lambda}\star \delta_{\on{Gr}_{G,x}}\star V^{\lambda}.
\]

\noindent The colimit is filtered, and since the t-structure on $D(\on{Gr}^+_M)$ is compatible with filtered colimits, it suffices to show that $i^{\theta,!}_x(e^{-\lambda}\star \delta_{\on{Gr}_{G,x}}\star V^{\lambda})$ lies in perverse cohomological degrees $\geq 2+\langle 2(\rho_G-\rho_M),\theta\rangle$ for any $\lambda$.

Let $\on{IC}^{\lambda}$ be the IC-sheaf on $\overline{\on{Gr}_{G,x}^{\lambda}}$ corresponding to $V^{\lambda}$. By definition, the assertion exactly amounts to showing that $i^{\lambda+\theta}(\on{IC}^{\lambda})$ is in perverse degrees $\geq 2$.

Note that when $\theta\neq 0$, the intersection $\on{Gr}_{M,x}^{+,\lambda+\theta}\cap \on{Gr}_{G,x}^{\lambda}$ is empty. Thus, the result follows from the parity vanishing of $\on{IC}^{\lambda}$.
 \end{proof}
 
\section{Drinfeld-Plücker structure on delta sheaf}\label{S:DrPLstructure}

In this section, we construct a natural Drinfeld-Plücker structure on $\delta_{\on{Gr}_G,\on{Ran}}$.

\subsubsection{External convolution}\label{s:extconvo}

Recall that the category $\on{Rep}(\check{G})_{\on{Ran}}$ comes equipped with an \emph{external convolution} monoidal structure.

For a finite set $I$, let $\on{Rep}(\check{G})_{X^I}:=\on{Rep}(\check{G})_{\on{Ran}}\underset{D(\on{Ran})}{\otimes} D(X^I)$. Let $\star$ denotes the pointwise convolution monoidal structure on $\on{Rep}(\check{G})_{\on{Ran}}$ defined in §\ref{s:pointwise}. Then external convolution is defined as the composition
\[
\on{Rep}(\check{G})_{X^I}\otimes \on{Rep}(\check{G})_{X^J}\to \on{Rep}(\check{G})_{X^{I\sqcup J}}\otimes \on{Rep}(\check{G})_{X^{I\sqcup J}}\xrightarrow{\star} \on{Rep}(\check{G})_{X^{I\sqcup J}}.
\]

\noindent Here, the first functor comes inserting the unit along $I\to I\sqcup J$ and $J\to I\sqcup J$, cf. the unital structure on $\on{Rep}(\check{G})_{\on{Ran}}$, see §\ref{s:unitality}. This defines an external convolution product:
\[
\overset{\on{ext}}{\star}: \on{Rep}(\check{G})_{\on{Ran}}\otimes \on{Rep}(\check{G})_{\on{Ran}}\to \on{Rep}(\check{G})_{\on{Ran}}.
\]

\noindent By construction, applying external convolution to an objects supported over $x$ and $y$, respectively, the result is supported over $x\sqcup y$.

\subsubsection{Naive geometric Satake}

In addition to the usual (pointwise) convolution structure on the spherical category $\on{Sph}_{G,\on{Ran}}$, this category similarly comes equipped with an external convolution product defined in the same way as for $\on{Rep}(\check{G})_{\on{Ran}}$. By construction, the (naive) geometric Satake functor
\[
\on{Sat}^{\on{nv}}: \on{Rep}(\check{G})_{\on{Ran}}\to \on{Sph}_{G,\on{Ran}}
\]

\noindent is monoidal for both the pointwise and external convolution structures.

\subsubsection{} We will apply the above setup for $\check{G}$ replaced by $\check{M}\times \check{G}$.

We need to construct an action
\[
\cO(\overline{\check{N}_P\backslash \check{G}})\curvearrowright \delta_{\on{Gr}_G,\on{Ran}},
\]

\noindent where $\cO(\overline{\check{N}_P\backslash \check{G}})$ is considered as an algebra object of $\on{Rep}(\check{M}\times \check{G})_{\on{Ran}}\simeq \on{Rep}(\check{M})\underset{D(\on{Ran})}{\otimes} \on{Rep}(\check{G})_{\on{Ran}}$ with its pointwise convolution product.

By definition of $\cO(\overline{\check{N}_P\backslash \check{G}})$, cf. §\ref{s:algsinquestion}, this amounts to a family of compatible maps
\begin{equation}\label{eq:familyofmaps}
\on{Sat}^{\on{nv}}(\on{Fun}(\overline{\check{N}_P\backslash \check{G}})^{\otimes I}\boxtimes \omega_{X^J})\star \delta_{\on{Gr}_G, X^I}\to \delta_{\on{Gr}_G, X^I}
\end{equation}

\noindent for each $(I\onto J)\in \on{TwArr}$.

\subsubsection{}\label{sec: drinfeld-plucker at a point}

First suppose that $I=J=*$ is a singleton. Then a map
\begin{equation}\label{eq:I=*}
\on{Sat}^{\on{nv}}(\on{Fun}(\overline{\check{N}_P\backslash \check{G}})\boxtimes \omega_{X})\star \delta_{\on{Gr}_G, X}\to \delta_{\on{Gr}_G, X}
\end{equation}

\noindent amounts to a compatible family of maps
\begin{equation}\label{eq:/X}
V_{\check{M},X}^{\lambda}\star \delta_{\on{Gr}_G, X}\to \delta_{\on{Gr}_G, X}\star V^{\lambda}_X
\end{equation}

\noindent for each dominant coweight $\lambda$. As usual, $V_{\check{M}}^{\lambda}$ denotes the irreducible representation of $\check{M}$ with highest weight $\lambda$, and $V^{\lambda}$ is the irreducible representation of $\check{G}$ with highest weight $\lambda$. The sheaves $V_{\check{M},X}^{\lambda}\in \on{Sph}_M, V^{\lambda}_X\in \on{Sph}_G$ denote the corresponding spherical sheaves over $X$ constructed in the usual way (see e.g. \cite[App. B]{gaitsgory2007jong}).

We have a canonical such map. Indeed, denote by $\on{IC}_M^{\lambda}\in D(\on{Gr}_M),\on{IC}^{\lambda}\in D(\on{Gr}_G)$ the IC-sheaves on $\overline{\on{Gr}_M^{\lambda}},\overline{\on{Gr}_G^{\lambda}}$, respectively. We consider $\on{IC}_M^{\lambda}$ as a sheaf on $\on{Gr}_G$ via the embedding $\on{Gr}_M\to \on{Gr}_G$. From the canonical identification
\[
\on{IC}^{\lambda}[2\langle \rho_G-\rho_M,\lambda\rangle]\vert_{\on{Gr}_M^{\lambda}}\simeq \on{IC}_M^{\lambda}\vert_{\on{Gr}_M^{\lambda}},
\]

\noindent we get a canonical map
\[
\on{IC}_M^{\lambda}\to \on{IC}^{\lambda}.
\]

\noindent This gives the desired map (\ref{eq:/X}).

\subsubsection{} Note that (\ref{eq:I=*}) gives a map
\begin{equation}\label{eq:I=*ext}
\on{Sat}^{\on{nv}}(\on{Fun}(\overline{\check{N}_P\backslash \check{G}})\boxtimes \omega_{X})\overset{\on{ext}}{\star}\delta_{\on{Gr}_G,X^I}\to \delta_{\on{Gr}_G,X^{I\sqcup *}}
\end{equation}

\noindent for each finite set $I$. Indeed, (\ref{eq:I=*ext}) is obtained from (\ref{eq:I=*}) by applying $\on{ins}_*(\omega_{X^I}\boxtimes -): D(\on{Gr}_{G,X})\to D(\on{Gr}_{G,X^{I\sqcup *}})$, where $\on{ins}: X^I\times\on{Gr}_{G,X}\to \on{Gr}_{G,X^{I\sqcup *}}$ inserts an $I$-tuple of points from which we further restrict the trivialization of the point $(x,\sP_G,\alpha)\in \on{Gr}_{G,X}$.

\subsubsection{} Next, suppose $I=J$. By definition of $\on{Sat}^{\on{nv}}$, we have:
\[
\on{Sat}^{\on{nv}}(\on{Fun}(\overline{\check{N}_P\backslash \check{G}})^{\otimes I}\boxtimes \omega_{X^I})\simeq \underset{I}{\overset{\on{ext}}{\star}} \on{Sat}^{\on{nv}}(\on{Fun}(\overline{\check{N}_P\backslash \check{G}})\boxtimes \omega_X).
\]

\noindent By iteratively applying (\ref{eq:I=*ext}), we thus obtain a map
\[
\on{Sat}^{\on{nv}}(\on{Fun}(\overline{\check{N}_P\backslash \check{G}})^{\otimes I}\boxtimes \omega_{X^I})\overset{\on{ext}}{\star} \delta_{\on{Gr}_G,X^I}\to \delta_{\on{Gr}_G,X^{I\sqcup I}}.
\]

\noindent $!$-pulling back along the diagonal $X^I\to X^{I\sqcup I}$, we obtain the map
\begin{equation}\label{eq:I=J}
\on{Sat}^{\on{nv}}(\on{Fun}(\overline{\check{N}_P\backslash \check{G}})^{\otimes I}\boxtimes \omega_{X^I})\star \delta_{\on{Gr}_G,X^I}\to \delta_{\on{Gr}_G,X^I}.
\end{equation}

\subsubsection{} Finally, let $(I\onto J)\in \on{TwArr}$ be arbitrary. Let $\Delta_{I\onto J}: X^J\to X^I$ be the corresponding diagonal embedding. We now get a map:
\[
\on{Sat}^{\on{nv}}(\on{Fun}(\overline{\check{N}_P\backslash \check{G}})^{\otimes I}\boxtimes \omega_{X^J})\star \delta_{\on{Gr}_G,X^I}\simeq\on{Sat}^{\on{nv}}(\on{Fun}(\overline{\check{N}_P\backslash \check{G}})^{\otimes I}\boxtimes \Delta_{I\onto J,_*}(\omega_{X^J}))\star \delta_{\on{Gr}_G,X^I}\to
\]
\[
\to \on{Sat}^{\on{nv}}(\on{Fun}(\overline{\check{N}_P\backslash \check{G}})^{\otimes I}\boxtimes \omega_{X^I})\star \delta_{\on{Gr}_G,X^I}\xrightarrow{(\ref{eq:I=J})} \delta_{\on{Gr}_G,X^I}.
\]

\noindent Here, all maps take place in $D(M(O)_{\on{Ran}}\backslash \on{Gr}_{G,\on{Ran}})$. Note that the above factors as:
\[
\on{Sat}^{\on{nv}}(\on{Fun}(\overline{\check{N}_P\backslash \check{G}})^{\otimes I}\boxtimes \omega_{X^J})\star \delta_{\on{Gr}_G,X^I}\to \Delta_{I\onto J_*}(\delta_{\on{Gr}_{G,X^J}})\to \delta_{\on{Gr}_{G,X^I}}.
\]

\subsubsection{}\label{s:footnote} As such, we have constructed the necessary family of maps (\ref{eq:familyofmaps}). They naturally come equipped with the higher coherence data required to give an action of $\cO(\overline{\check{N}_P\backslash \check{G}})$ on $\delta_{\on{Gr}_{G,\on{Ran}}}$.\footnote{An alternative way to see this is to use the relative perverse t-structure of Hansen-Scholze \cite{hansen2023relative}. Indeed, both $\cO(\overline{\check{N}_P\backslash \check{G}})_{X^I}$ and $\delta_{\on{Gr}_{G,X^I}}$ naturally lie in the heart of the relative perverse t-structure. Moreover, pointwise convolution on $\on{Sph}_{G,X^I}$ is exact for this t-structure, so the higher compatibilities are automatic.}

\newpage

\bibliographystyle{alpha}
\bibliography{Bip}
\end{document}